\documentclass[12pt, a4paper,leqno]{amsart}
\usepackage{amsmath,amsthm,amscd,amssymb,amsfonts, amsbsy}
\usepackage{latexsym}
\usepackage{exscale}
\usepackage{mathrsfs}
\usepackage[utf8]{inputenc}

\usepackage[colorlinks,citecolor=red,pagebackref,hypertexnames=false, breaklinks]{hyperref}

\calclayout
\allowdisplaybreaks

\renewcommand\labelenumi{(\roman{enumi})}
\renewcommand\theenumi\labelenumi

\newcommand{\abs}[1]{{\left\lvert{#1}\right\rvert}}
\newcommand{\norm}[1]{{\left\lVert{#1}\right\rVert}}
\newcommand{\br}[1]{{\left({#1}\right)}}
\newcommand{\bbr}[1]{{\left\{{#1}\right\}}}

\newcommand{\fiint}{\fint\!\!\!\fint}

\def \Int{\,\Int\,}
\def \iint{\int\!\!\!\int}

\def\Xint#1{\mathchoice
    {\XXint\displaystyle\textstyle{#1}}%
    {\XXint\textstyle\scriptstyle{#1}}%
    {\XXint\scriptstyle\scriptscriptstyle{#1}}%
    {\XXint\scriptscriptstyle\scriptscriptstyle{#1}}%
    \!\int}
    \def\XXint#1#2#3{{\setbox0=\hbox{$#1{#2#3}{\int}$}
    \vcenter{\hbox{$#2#3$}}\kern-.5\wd0}}
    \def\fint{\Xint-}

\def\Xint#1{\mathchoice
    {\XXint\displaystyle\textstyle{#1}}%
    {\XXint\textstyle\scriptstyle{#1}}%
    {\XXint\scriptstyle\scriptscriptstyle{#1}}%
    {\XXint\scriptscriptstyle\scriptscriptstyle{#1}}%
    \!\int}
    \def\XXint#1#2#3{{\setbox0=\hbox{$#1{#2#3}{\int}$}
    \vcenter{\hbox{$#2#3$}}\kern-.5\wd0}}
    \def\dashint{\Xint-}
  \renewcommand{\emptyset}{\mbox{\textup{\O}}}

\def\div{\mathop{\rm div}}
\def \A{ \mathbb{A} }

\def \L{ \mathbb{L} }
\def \N{ \mathbb{N} }
\def \R{ \mathbb{R} }
\def \Z{ \mathbb{Z} }

\def \Ncal { \mathcal{N} }
\def \Mcal { \mathcal{M} }
\def \Ocal{ \mathcal{O} }
\def \Scal{ \mathcal{S} }

\def \hh{ \mathrm{H} }
\def \pp{ \mathrm{P} }
\def \Grm{ \mathrm{G} }

\newcommand{\wt}{\widetilde}

\newcommand{\loc}{\text{{\rm loc}}}

\newcommand{\divv}{{\text{{\rm div}}}}
\renewcommand{\Re}{{\rm Re}\,}

\def \re{ \mathbb{R} }

\def\esssup{\mathop\mathrm{\,ess\,sup\,}}
\renewcommand{\Int}{{\rm Int}\,}

\DeclareMathOperator{\supp}{supp}

\renewcommand{\chi}{{\bf 1}}

\theoremstyle{plain}
\newtheorem{theorem}[equation]{Theorem}
\newtheorem{lemma}[equation]{Lemma}
\newtheorem{corollary}[equation]{Corollary}
\newtheorem{proposition}[equation]{Proposition}

\theoremstyle{definition}
\newtheorem{definition}[equation]{Definition}

\theoremstyle{remark}
\newtheorem{remark}[equation]{Remark}

\numberwithin{equation}{section}

\begin{document}
\allowdisplaybreaks

\author{Pascal Auscher}

\address{Laboratoire de Math\'{e}matiques d'Orsay, Université Paris-Saclay, CNRS, Universit\'{e} Paris-Saclay, 91405 Orsay, France; and Laboratoire Ami\'enois de Math\'ematique  Fondamentale et Appliqu\'ee, 
CNRS-UMR 7352, Universit\'e de Picardie-Jules Verne, 80039 Amiens France}
\email{pascal.auscher@universite-paris-saclay.fr}

\author{Li Chen}
\address{Li Chen
	\\
	Department of Mathematics
	\\
	Louisiana State University
	\\
	Baton Rouge, LA 70803-4918, USA}
\email{lichen@lsu.edu}

\author{Jos\'e Mar{\'\i}a Martell}
\address{Jos\'e Mar{\'\i}a Martell
\\
Instituto de Ciencias Matem\'aticas CSIC-UAM-UC3M-UCM
\\
Consejo Superior de Investigaciones Cient{\'\i}ficas
\\
C/ Nicol\'as Cabrera, 13-15
\\
E-28049 Madrid, Spain} \email{chema.martell@icmat.es}

\author{Cruz Prisuelos-Arribas}
\address{Cruz Prisuelos-Arribas
	\\
	Departamento de Física y Matemáticas	
	\\
	Universidad de Alcalá de Henares
	\\
	Plaza de San Diego, s/n
	\\
	E-28801 Alcalá de Henares, Madrid, Spain} 
\email{cruz.prisuelos@uah.es}

\title[The regularity problem for degenerate elliptic operators]{The regularity problem for degenerate elliptic operators in weighted spaces}

\thanks{The second, third, and last authors  acknowledge that the research leading to these results has received funding from the European Research 
	Council under the European Union's Seventh Framework Programme (FP7/2007-2013)/ ERC	agreement no. 615112 HAPDEGMT. 	The third author acknowledges financial support from the Spanish Ministry of Science and Innovation, through the ``Severo Ochoa Programme for Centres of Excellence in R\&D'' (CEX2019-000904-S). The third and last authors were partially supported by the Spanish Ministry of Science and Innovation,  MTM PID2019-107914GB-I00.}

\date{\today}


\begin{abstract}
We study the solvability of the regularity problem for degenerate elliptic operators in the block case for data in weighted spaces.
More precisely, let $L_w$ be a degenerate elliptic operator with degeneracy given by a fixed weight $w\in A_2(dx)$ in $\R^n$, and consider the associated block second order degenerate elliptic problem in the upper-half space $\R_+^{n+1}$. We obtain non-tangential bounds for the full gradient of the solution of the block case operator given by the Poisson semigroup in terms of the gradient of the boundary data. All this is done in the spaces $L^p(vdw)$ where $v$ is a Muckenhoupt weight with respect to the underlying natural weighted space $(\R^n, wdx)$. We recover earlier results in the non-degenerate case (when $w\equiv 1$, and with or without weight $v$). Our strategy is also different and more direct thanks in particular  to recent observations on change of angles in weighted square function estimates and  non-tangential maximal functions. 
Our method gives as a consequence the (unweighted) $L^2(dx)$-solvability of the regularity problem for the block operator 
\[
\L_\alpha u(x,t)
=
-|x|^{\alpha} \divv_x \big(|x|^{-\alpha }\,A(x) \nabla_x u(x,t)\big)-\partial_{t}^2 u(x,t)
\]
for any complex-valued uniformly elliptic matrix $A$ and for all $-\epsilon<\alpha<\frac{2\,n}{n+2}$,
where $\epsilon$ depends just on the dimension and the ellipticity constants of $A$.
\end{abstract}

\maketitle

\tableofcontents

\bigskip
\section{Introduction}\label{section:intro}

The study of divergence form  degenerate elliptic equations was pioneered in the series of papers \cite{FJK82,FKJ83,FKS82}, where
real symmetric elliptic matrices with some degeneracy expressed in terms of $A_2(dx)$-weights were considered (here and elsewhere $ A_{2}(dx)\equiv A_{2}(\R^n,dx)$). The goal of this paper is to obtain the solvability of the regularity problem  for second order divergence form degenerate elliptic operators with complex coefficients and with boundary data in weighted Lebesgue spaces. To set the stage, let us introduce the class of operators that we consider here. Let  $A$ be an $n\times n$ matrix of complex $L^\infty$-valued coefficients defined on $\R^n$, $n\geq 2$. We assume that
this matrix satisfies the following uniform ellipticity (or \lq\lq
accretivity\rq\rq) condition: there exist $0<\lambda\le\Lambda<\infty$ such that
\begin{equation}
\lambda\,|\xi|^2
\le
\Re A(x)\,\xi\cdot\bar{\xi}
\quad\qquad\mbox{and}\qquad\quad
|A(x)\,\xi\cdot \bar{\zeta}|
\le
\Lambda\,|\xi|\,|\zeta|,
\label{eq:elliptic-intro}
\end{equation}
for all $\xi,\zeta\in\mathbb{C}^n$ and almost every $x\in \R^n$. We have used the notation
$\xi\cdot\zeta=\xi_1\,\zeta_1+\cdots+\xi_n\,\zeta_n$ and therefore
$\xi\cdot\bar{\zeta}$ is the usual inner product in $\mathbb{C}^n$.  
Associated with this matrix and a given weight $w\in A_{2}(dx)$ (which is fixed from now on, unless stated otherwise) we define the second order divergence
form degenerate elliptic operator
\begin{align}\label{degenerateL}
L_w u
=
-w^{-1}\div(w\,A\,\nabla u),
\end{align}
which is understood as a maximal-accretive operator on $L^2(\R^n,w dx)\equiv L^2(w)$ with domain $\mathcal{D}(L_w)$ by means of a sesquilinear form. Note that writing $A_w=w\,A$, one has that $A_w$ is a degenerate elliptic matrix in the sense that 
\begin{equation}
\lambda\,|\xi|^2\,w(x)
\le
\Re A_w(x)\,\xi\cdot\bar{\xi}
\quad\qquad\mbox{and}\qquad\quad
|A_w(x)\,\xi\cdot \bar{\zeta}|
\le
\Lambda\,|\xi|\,|\zeta|\,w(x),
\label{eq:elliptic-intro:w}
\end{equation}
for all $\xi,\zeta\in\mathbb{C}^n$ and almost every $x\in \R^n$. Conversely, if $A_w$ is degenerate elliptic matrix satisfying the previous conditions one can trivially see that $A:=w^{-1}\,A_w$ is uniformly elliptic.

The prominent case $w\equiv 1$ gives the class of uniformly elliptic operators. The celebrated resolution of the Kato problem in \cite{AHLMT02} established that if $L$ is a uniformly divergence form elliptic operator (that is, $L=L_w$ with $w\equiv 1$) then $\sqrt{L}f$ is comparable to $\nabla f$ in $L^2(\R^n,dx) \equiv L^2(dx)$. This led to a new Calderón-Zygmund theory developed by the first named author in \cite{Au07} to establish the boundedness in Lebesgue spaces of the associated functional calculus, vertical square function, Riesz transforms, reverse inequalities, etc. A key ingredient in that theory is the use of the so-called off-diagonal or Gaffney estimates satisfied by the associated heat semigroup and its  gradient. This was later extended in \cite{AMIII06,AMI07,AMII07} where the same operators were shown to satisfy weighted norm inequalities with Muckenhoupt weights. Conical square functions have been also considered in \cite{AHM, MaPAI17}. Some of the previous results in conjunction with the theory of Hardy spaces for uniformly elliptic operators from \cite{HMay09, HMayMc11} led to \cite{May10} where the solvability of the regularity problem in the block case  for data in Lebesgue spaces was obtained. This amounted to control non-tangentially the full gradient of the solution given by the Poisson semigroup in terms of the gradient of the boundary datum. In turn, using the weighted Hardy space theory developed in \cite{MaPAI17, MaPAII17, PA17}, the solvability of the regularity problem in the block case  for data in Lebesgue spaces with Muckenhoupt weights has been recently studied in \cite{ChMP-A18}.

Concerning the Kato problem in the general case, where $L_w$ is a degenerate elliptic operator as above with a generic $w\in A_2(dx)$, \cite{CUR15} (see also \cite{CUR08,CUR12})  showed that $\sqrt{L_w} f$ is comparable to $\nabla f$  in $L^2( w)$. The boundedness of the associated operators (functional calculus, Riesz transform, reverse inequalities, vertical square functions, etc.), both in the natural Lebesgue spaces $L^p(w)$ and also in weighted spaces $L^p(v\,dw)$ with $v\in A_\infty(w)$ was considered in \cite{CMR15}. A particular case of interest is that on which under further assumptions in $w$ one can show equivalence of $\sqrt{L_w} f$ and $\nabla f$ in $L^2(dx)$ by simply taking $v=w^{-1}$, that is, the $L^2(dx)$-problem Kato problem was solved for a class of degenerate elliptic operators that goes beyond that of uniformly elliptic ---e.g., one can take $L_\gamma=-|\cdot|^\gamma\,\div(|\cdot|^{-\gamma} A(\cdot)\nabla)$ where $A$ is a uniformly elliptic matrix, $\gamma\in (-\varepsilon, 2n/(n+2))$, and $\varepsilon$ depends on the dimension and the ellipticity constants of $A$.  Some work has been also done  concerning  conical square functions with respect to the heat or Poisson semigroup (or theirs gradients) generated by $L_w$. For example, in \cite{ChMPA16} the last three authors of the present paper established the boundedness and the comparability of some conical square functions extending to the degenerate case the results from \cite{MaPAI17}. Moreover, in \cite{PAII18}, the last named author has made a deeper study of the vertical and conical square functions and some non-tangential maximal functions  arising from  degenerate elliptic  operators. On another direction, in \cite{ARR15}  the authors considered $L^2$-boundary value problems for degenerate elliptic equations and systems.  In particular, they initiated the study of Dirichlet and Neumann problems in the degenerate setting using the so-called first order method.

Our goal in this paper to contribute to this theory studying the solvability of the regularity problem for degenerate elliptic operators and also propose other methods. More precisely, consider the degenerate elliptic operator $L_w =-w^{-1}\div(w\,A\,\nabla )$ where $w\in A_2(dx)$ and $A$ is an  $n\times n$ matrix of complex $L^\infty$-valued coefficients defined on $\R^n$, $n\geq 2$, which is uniformly elliptic matrix (see \eqref{eq:elliptic-intro}). Introduce the  $(n+1)\times (n+1)$ block matrix
\[
\A = \left(\begin{array}{cc}A & 0 \\0 & 1\end{array}\right),
\]
which is $(n+1)\times (n+1)$ uniformly elliptic. This gives rise to a degenerate elliptic operator in $\R^{n+1}$, 
\begin{equation}\label{def:L-block}
\L_w u
=-w^{-1} \divv_{x,t} (w\, \A \nabla_{x,t}u)
=
-w^{-1} \divv_x (w\,A \nabla_x u)-\partial_{t}^2u
=(L_w)_x u-\partial_{t}^2u.
\end{equation}
Here and elsewhere, $\nabla_{x,t}$  denotes the full gradient, while the symbol $\nabla$ refers just to the spatial derivatives. Note that in the previous equality we have used that $w$ does not depend on the $t$ variable, hence it is trivial to see that with a slight abuse of notation if we write $w(x,t):=w(x)$ for every $(x,t)\in\R^{n+1}$ then $w\in A_2(\R^{n+1}, dx)$ since $w\in A_2(dx)$. 

The operator $-L_w$ generates a $C^0$-semigroup $\{e^{-t L_w}\}_{t>0}$ of contractions on $L^2(w)$ which is called the heat semigroup. This and the subordination formula (see \eqref{subordinationformula} below) yield that $\{e^{-t \sqrt{L_w}}\}_{t>0}$ is a $C^0$-semigroup of contractions on $L^2(w)$. Hence, whenever $f\in C_c^\infty(\R^n)$, one has that $u(x,t):=e^{-t\sqrt L_w}f(x)$, with  $(x,t)\in\R_+^{n+1}$, is a strong solution  of $\L_w u=0$ in $\R_{+}^{n+1}$. It is also a weak solution: by this we mean that $u\in W_{{\rm loc}}^{1,2}(\R_+^{n+1}, dw\,dt)$ satisfies 
\begin{equation}\label{weka-sol}
\iint_{\R_+^{n+1}} \A(x)\nabla_{x,t} u(x,t) \cdot \nabla_{x,t} \psi(x,t)\, dw(x)\,dt =0,\qquad \forall\,\psi \in C_{0}^{\infty}(\R_+^{n+1}).
\end{equation}
Also, $u(\cdot,t)\to f$ in $L^2(w)$ as $t\to 0^+$ by the semigroup representation. As usual, $dw(x)\equiv w(x)dx$.

Consider the $L^2$-non-tangential maximal function $\Ncal_w $ defined in
\cite{ARR15}:
\begin{equation}\label{eq:NTdeg}
\Ncal_w h(x):= \sup_{t>0} \left(\,\fiint_{W(x,t)} |h(y,s)|^2 dw(y)\, ds\right)^{\frac12},\quad  h\in L_{\loc}^2(\R_+^{n+1}, dw\,dt).
\end{equation}
where $W(x,t):=(c_0^{-1}t,c_0t)\times B(x,c_1t)$ is a Whitney region and $c_0>1,\,c_1>0$ are fixed parameters throughout the paper.

Given $1<p<\infty$ and $v\in A_\infty(w)$ (note that our assumption $w\in A_2(dx)$ implies that $w$ is a doubling measure and hence $(\R^n,w,|\cdot|)$ is a space of homogeneous type), we say that the weighted regularity problem $(R^{\L_w})_{L^p(vdw)}$ is solvable if for every 
$f\in C_c^\infty(\R^n)$ the weak solution of $\L_w u=0$ in $\R_{+}^{n+1}$ given by $u(x,t):=e^{-t\sqrt{L_w}}f(x)$, $(x,t)\in \R^{n+1}_+$, 
satisfies the following weighted non-tangential maximal function estimate:
\begin{equation}\label{ntangestimate}
\norm{\Ncal_w(\nabla_{x,t}u)}_{L^p(vdw)} \leq C\norm{\nabla f}_{L^p(vdw)}.
\end{equation}
Once this estimate is under control, one can extend the semigroup to general data. However, the status of convergence to the boundary of the solution needs a specific treatment that is not addressed here. 

As in \cite{Au07, AMII07, CMR15}, we denote by $(p_-(L_w),p_+(L_w))$ and by $(q_-(L_w),q_+(L_w))$ the maximal open intervals on which the heat semigroup $\{e^{-tL_w}\}_{t>0}$ and the gradient of the heat semigroup $\{\sqrt{t}\nabla e^{-tL_w}\}_{t>0}$ are respectively uniformly bounded on $L^p(w)$. That is,
\begin{align*}
p_-(L_w) &:= \inf\left\{p\in(1,\infty): \sup_{t>0} \|e^{-tL_w}\|_{L^p(w)\rightarrow L^p(w)}< \infty\right\},
\\[4pt]
p_+(L_w)& := \sup\left\{p\in(1,\infty) : \sup_{t>0} \|e^{-tL_w}\|_{L^p(w)\rightarrow L^p(w)}< \infty\right\},
\end{align*}
\begin{align*}
q_-(L_w) &:= \inf\left\{p\in(1,\infty): \sup_{t>0} \|\sqrt{t}\nabla e^{-t^2L_w}\|_{L^p(w)\rightarrow L^p(w)}< \infty\right\},
\\[4pt]
q_+(L_w)& := \sup\left\{p\in(1,\infty) : \sup_{t>0} \|\sqrt{t}\nabla e^{-t^2L_w}\|_{L^p(w)\rightarrow L^p(w)}< \infty\right\}.
\end{align*}

We need to introduce some extra notation (see Section \ref{sec:preliminaries}). Set $r_w:=\inf\big\{p:\ w\in A_p(dx)\big\}$, and note that $1\le r_w<2$ since $w\in A_2(dx)$. Given $0\le p_0<q_0\le\infty$ and $v\in A_{\infty}(w)=A_{\infty}(\R^n,w\,dx)$ define
\[
\mathcal{W}_v^w(p_0,q_0):=\left\{p\in(p_0,q_0):\ v\in A_{\frac{p}{p_0}}(w)\cap RH_{\left(\frac{q_0}{p}\right)'}(w)\right\}.
\]
We are now ready to state our main result:
\begin{theorem}\label{thm:maindegenerate}
	Let $w\in A_2(dx)$ and let $\L_w$ be a block  degenerate elliptic operator in $\R_+^{n+1}$ as above. Let $v\in A_{\infty}(w)$ be such that 
\begin{equation}\label{main:compa}
	\mathcal{W}_v^w\big(\max\{r_w,q_-(L_w)\},q_+(L_w)\big)\neq \emptyset.
\end{equation}
	Then, for every $p\in \mathcal{W}_v^w\left(\max\left\{r_w,\frac{nr_w q_-(L_w)}{nr_w+q_-(L_w)}\right\},q_+(L_w)\right)$ and every $f\in C_c^\infty(\R^n)$, 
	if one sets $u(x,t)=e^{-t\sqrt L_w}f(x)$, $(x,t)\in\R_+^{n+1}$, then
	\begin{equation}\label{main:main-est}
	\norm{\Ncal_w (\nabla_{x,t} u)}_{L^p(vdw)} \leq C\norm{\nabla f}_{L^p(vdw)}
	\end{equation}
and  $(R^{\L_w})_{L^p(vdw)}$ is solvable.
\end{theorem}

Let us compare this result with some previous work. When $w\equiv 1$ (that is, we are working with the class of uniformly ---or non-degenerate--- elliptic operators) and $v\equiv 1$ then, clearly,  $r_w=1$,  $\mathcal{W}_v^w\big(\max\{r_w,q_-(L_w)\},q_+(L_w)\big)=(q_-(L_w),q_+(L_w))\neq\emptyset$, and our result gives \eqref{main:main-est} in the range
$\left(\max\left\{1,\frac{nq_-(L_w)}{n+q_-(L_w)}\right\},q_+(L_w)\right)$, hence we fully  recover \cite[Theorem 4.1]{May10}. If we still assume that  $w\equiv 1$ and we let $v\in A_\infty(w)=A_\infty(dx)$, then our assumption \eqref{main:compa} agrees with that in \cite[Theorem 1.10]{ChMP-A18} and the range of $p$'s here is slightly worse than the one in that result (the lower end-point in \cite[Theorem 1.10]{ChMP-A18} has been pushed down using an extra technical argument that we have chosen not to follow here). 

Our methods to prove Theorem \ref{thm:maindegenerate}, in particular the estimate involving  $\partial_{t} u$, are  also novel.  The above works used  advanced technology of Hardy spaces adapted to operators: developing them in our context is probably a new challenge in  itself.   Instead, we rely on recent change of angle formulas for weighted conical square function estimates (see Section \ref{section:change-of-angles}) and also the ones we prove for non-tangential weighted maximal functions (see Lemma \ref{lema:changeofangelN})   which allow us to implement more directly standard tools in the field.

An important consequence of our method is that we can obtain the solvability of the regularity problem corresponding to data in unweighted Lebesgue spaces. The main idea consists in taking $v=w^{-1}$ in Theorem \ref{thm:maindegenerate}. The following result focuses in the case of the $L^2$-solvability (more general results are presented in Section \ref{section:unweighted}, see Corollaries \ref{corol:weights} and \ref{corol:weights:power}):

\begin{corollary}\label{corol:weighted:intro}
Let $w\in A_2(dx)$ and let $\L_w$ be a block  degenerate elliptic operator in $\R_+^{n+1}$ as above. Given $\Theta\ge 1$ there exists $\epsilon_0=\epsilon_0(\Theta,n,\Lambda/\lambda)\in (0,\frac1{2n}]$, such that for every $w\in A_{1+\epsilon}(dx)\cap  RH_{\max\{\frac2{1-\epsilon}, 1+(1+\epsilon)\,\frac{n}2\}}(dx)$ with  $0\le\epsilon<\epsilon_0$ and $[w]_{A_2(dx)}\le \Theta$, then
\begin{equation}\label{intro-L2}
\norm{\Ncal_w (\nabla_{x,t} u)}_{L^2(dx)} \leq C\norm{\nabla f}_{L^2(dx)}.
\end{equation}
for every $f\in C_c^\infty(\R^n)$ and where $u(x,t)=e^{-t\sqrt L_w}f(x)$, $(x,t)\in\R_+^{n+1}$. Hence $(R^{\L_w})_{L^2(dx)}$ is solvable.

Furthermore, if we set 
\[
\L_\alpha u(x,t)
=
-|x|^{\alpha} \divv_x \big(|x|^{-\alpha}\,A(x) \nabla_x u(x,t)\big)-\partial_{t}^2 u(x,t)
\]
where $A$ is an $n\times n$ matrix of complex $L^\infty$-valued coefficients defined on $\R^n$, $n\geq 2$, satisfying the uniform ellipticity condition 
\eqref{eq:elliptic-intro}, then there exists $0<\epsilon<\frac12$ small enough (depending only on the dimension and the ratio $\Lambda/\lambda$) such that 
if $-\epsilon<\alpha<\frac{2\,n}{n+2}$ then \eqref{intro-L2} holds in this scenario and   $(R^{\L_\alpha})_{L^2(dx)}$ is solvable.
\end{corollary}

The plan of the paper is as follows.  In Section \ref{sec:preliminaries} we introduce notations and definitions, and we recall some known results. We also obtain  estimates for some inhomogeneous vertical and conical square functions which are interesting on its own right (see Propositions \ref{thm:wtG} and \ref{thm:boundednesswidetildeS}). To prove our main result, Theorem \ref{thm:maindegenerate}, we split the main estimate into two independent pieces, one regarding $\Ncal_w (\nabla_x u)$ and the other one related to $\Ncal_w (\partial_{t} u)$, see respectively Propositions \ref{prop:spaceD} and \ref{prop:timeD} in Section \ref{sec:proofmain}. 
In Section \ref{section:unweighted} we study the solvability of the regularity problem in unweighted Lebesgue spaces and, in particular, we prove Corollary \ref{corol:weighted:intro}.

\section{Preliminaries}\label{sec:preliminaries}
We shall use the following notation: $dx$ denotes the usual Lebesgue measure in $\R^n$, $dw$ denotes the measure in $\R^n$ given by the weight $w$, and $vdw$ or $d(vw)$ denotes the one given by the product weight $vw$. Besides, throughout the paper
$n$ will denote the dimension of the underlying space
$\R^n$ and we shall always assume
$n\geq 2$.

Given a ball $B$, let $r_B$ denote the radius of $B$.  We write $\lambda B$ for the concentric ball with radius $\lambda\, r_B$, $\lambda>0$. Moreover, we set $C_1(B)=4B$ and, for $j\geq 2$, $C_j(B)=2^{j+1}B \backslash 2^j B$.

\subsection{Weights} We need to introduce some classes of Muckenhoupt weights. Namely,  $A_\infty(dx)$, on which the underlying measure space is $(\R^n,dx)$, and then fix $w\in A_\infty(dx)$ and consider the class $A_\infty(w)$ where the ``weighted'' underlying space is $(\R^n,dw)$.

\subsubsection{$A_\infty(dx)$ weights} By a weight $w$ we mean a non-negative, locally integrable function.
For brevity, we will often write $dw$ for $w\,dx$. In particular, we write $w(E)=\int_E\,dw$ and $L^p(w)=L^{p}(\R^n,dw)$.  We will use the following notation for averages:   given a set $E$ such that
$0<w(E)<\infty$,
\[ \dashint_E f\,dw = \frac{1}{w(E)}\int_E f\,dw, \]
or, if $0<|E|<\infty$,
\[ \dashint_E f\,dx = \frac{1}{|E|}\int_E f\,dx. \]
Abusing slightly the notation, for  $j\geq 1$, we set
$$
\dashint_{C_j(B)} f\, dw= \frac{1}{w(2^{j+1}B)} \int_{C_j(B)} f\,dw.
$$

We state some definitions and basic properties of Muckenhoupt
weights.  For further details,
see~\cite{Du01, GCRF85, Gra14}. Consider the Hardy-Littlewood maximal function
$$
\mathcal{M} f(x):=\sup_{B\ni x} \dashint_B|f(y)|\,dy.
$$
It is well-known that given a weight $w$, $\mathcal{M}$ is bounded on $L^p(w)$, if and only if,  $w\in A_p(dx)$, $1<p<\infty$ where we say that $w\in A_p(dx)$, $1<p<\infty$, if
\[ [w]_{A_p(dx)} := \sup_B \left(\dashint_B w(x)\,dx\right) \left(\dashint_B
  w(x)^{1-p'}\,dx\right)^{p-1} < \infty. \]
Here and below the sups run over the collection of balls $B\subset\re^n$.
When $p=1$, $\mathcal{M}$ is bounded from $L^1(w)$ to $L^{1,\infty}(w)$ if and only if $w\in A_1(dx)$, that is, if
\[ [w]_{A_1(dx)} := \sup_B \left(\dashint_B w(x)\,dx\right)  \left(\esssup_{x\in B} w(x)^{-
1}\right)<
\infty.  \]

We also introduce the reverse H\"older classes. We say that $w\in RH_s(dx)$, $1<s<\infty$ if
\[ [w]_{RH_s(dx)} := \sup_B \left(\dashint_B w(x)\,dx\right )^{-1}
\left(\dashint_B w(x)^s\,dx\right )^{\frac{1}{s}} < \infty, \]
and
\[ [w]_{RH_\infty(dx)} := \sup_B\left(\dashint_B w(x)\,dx\right)^{-1} \left(\esssup_{x\in B} w(x)\right)  <
\infty.  \]
It is also well-known that
\[ A_\infty(dx) := \bigcup_{1\leq p <\infty} A_p (dx) = \bigcup_{1<s\le \infty}
RH_s(dx).  \]
%

Throughout the paper we shall use in several places the following  properties. Namely, if $w\in RH_s(dx)$, $1<s\le\infty$, 
\begin{align}\label{pesosineqw:RHq}
\frac{w(E)}{w(B)}\leq  [w]_{RH_{s}(dx)}\left(\frac{|E|}{|B|}\right)^{\frac{1}{s'}}, \quad \forall\,E\subset B,
\end{align}
where $B$ is any ball in $\re^n$. 
Analogously, if
$w\in A_p(dx)$, $1\leq p<\infty$, then
\begin{align}\label{pesosineqw:Ap}
 \left(\frac{|E|}{|B|}\right)^{p}\le [w]_{A_{p}(dx)}\frac{w(E)}{w(B)}, \quad \forall\,E\subset B.
\end{align}
This implies in particular that $w$ is a doubling measure, that is, 
\begin{align}\label{doublingcondition}
w(\lambda B)
\le
[w]_{A_p(dx)}\,\lambda^{n\,p}w(B),
\qquad \forall\,B,\ \forall\,\lambda>1.
\end{align}

 We continue by introducing  some important notation.
 Weights in the $A_p(dx)$ and $RH_s(dx)$ classes have a self-improving
 property: if $w\in A_p(dx)$, there exists $\epsilon>0$ such that $w\in
 A_{p-\epsilon}(dx)$, and similarly if $w\in RH_s(dx)$, then $w\in
 RH_{s+\delta}(dx)$ for some $\delta>0$.  Hereafter, given $w\in A_p(dx)$, let
 \begin{equation}
 r_w=\inf\big\{p:\ w\in A_p(dx)\big\}, \qquad s_w=\inf\big\{q:\ w\in RH_{q'}(dx)\big\}.
 \label{eq:defi:rw}
 \end{equation}
 Note that according to our definition $s_w$ is the conjugated exponent of the one defined in \cite[Lemma 4.1]{AMI07}.
 Given $0\le p_0<q_0\le \infty$ and $w\in A_{\infty}(dx)$,  \cite[Lemma 4.1]{AMI07} implies that
 \begin{align}\label{intervalrs}
 \mathcal{W}_w(p_0,q_0):=\left\{p\in (p_0, q_0): \ w\in A_{\frac{p}{p_0}}(dx)\cap RH_{\left(\frac{q_0}{p}\right)'}(dx)\right\}
 =
 \left(p_0r_w,\frac{q_0}{s_w}\right).
 \end{align}
 If $p_0=0$ and $q_0<\infty$ it is understood that the only condition that stays is $w\in RH_{\left(\frac{q_0}{p}\right)'}(dx)$. Analogously, 
 if $0<p_0$ and $q_0=\infty$ the only assumption is $w\in A_{\frac{p}{p_0}}(dx)$. Finally $\mathcal{W}_w(0,\infty)=(0,\infty)$. 

Furthermore, given $p\in (0,\infty)$ and a weight $w\in A_{\infty}(dx)$, we define the following Sobolev exponents with respect to $w$
\begin{align}\label{p_{w,*}}
(p)_{w,*}:=\frac{nr_wp}{nr_w+p},
\end{align}
and, for $k\in \N$,
\begin{align}\label{p_w^*}
p_{w}^{k,*}:=\begin{cases}
\frac{nr_wp}{nr_w-kp}& \textrm{ if }nr_w>kp,
\\
\infty & \textrm{ otherwise.}
\end{cases}
\end{align}
We write $p_w^*:=p_w^{1,*}$.

\subsubsection{$A_\infty(w)$ weights}
Fix now $w\in A_{\infty}(dx)$. As mentioned above, \eqref{doublingcondition} says that $w$ is a doubling measure, hence 
$(\re,dw,|\cdot|)$ is a space of homogeneous type (here and elsewhere $|\cdot|$ stands for the ordinary Euclidean distance).
One can then introduce the weighted maximal operator
	\begin{align}\label{weightedHLM}
	\mathcal{M}^wf(x):=\sup_{B\ni x}\dashint_B |f(y)|\,dw(y).
	\end{align}
Much as before,  $\mathcal{M}^w$ is bounded on $L^p(v dw)$, $1<p<\infty$, if and only if, $v\in A_p(w)$, which  means that 
\begin{align}\label{apclassv}
[v]_{A_p(w)} = \sup_B \left(\dashint_B v(x)\,dw\right) \left(\dashint_B
v(x)^{1-p'}\,dw\right)^{p-1} < \infty.
\end{align}
Analogously, we can define the classes $RH_s(w)$ by replacing the Lebesgue measure in the definitions above with
$dw$: $v\in RH_s(w)$, $1<s<\infty$ if
\begin{align}\label{rhclassv}
[v]_{RH_s(w)} = \sup_B \left(\dashint_B v(x)\,dw\right)^{-1} \left(\dashint_B
  v(x)^{s}\,dw\right)^{\frac{1}{s}} < \infty.
\end{align}
From these definitions, it follows at once  that there is a
``duality'' relationship between the weighted and unweighted
$A_p(dx)$ and $RH_s(dx)$ conditions:  $w^{-1} \in A_p(w)$ if and only if $w \in
RH_{p'}(dx)$ and $w^{-1}\in RH_s(w)$ if and only if $w\in A_{s'}(dx)$.

For every measurable set $E\in \R^n$, we write $vw(E)=\int_Ed(vw)=\int_E v dw=(v dw)(E)$ and $L^p(v dw)=L^p(\re^n, v(x)\,w(x)\,dx)$. In this direction, 
for every $w\in A_p(dx)$, $v\in A_q(w)$, $1\le p,q<\infty$, it follows that
\begin{align}\label{pesosineq:Ap}
\left(\frac{|E|}{|B|}\right)^{p\,q}
\le
[w]_{A_{p}(dx)}^q\left(\frac{w(E)}{w(B)}\right)^{q}
\le
[w]_{A_{p}(dx)}^q[v]_{A_{q}(w)}
\frac{vw(E)}{vw(B)},\quad \forall\,E\subset B.
\end{align}
Analogously, if $w\in RH_{p}(dx)$ and $v\in RH_{q}(w)$ $1< p,q\le\infty$, one has 
\begin{align}\label{pesosineq:RHq}
\frac{vw(E)}{vw(B)}
\leq
[v]_{RH_{q}(w)}\left(\frac{w(E)}{w(B)}\right)^{\frac{1}{q'}}\leq [v]_{RH_{q}(w)}[w]_{RH_{p}(dx)}^{\frac1{q'}}\left(\frac{|E|}{|B|}\right)^{\frac{1}{p'\,q'}},\quad \forall\,E\subset B.
\end{align}

As before, for a weight $v\in A_{\infty}(w)$ (recall that $w\in A_{\infty}(dx)$ is fixed) we set
\begin{align}\label{eq:defi:rvw}
\mathfrak{r}_v(w):=\inf\big\{r:\ v\in A_{r}(w)\big\}\quad \textrm{and}\quad
\mathfrak{s}_v(w):=\inf\big\{s:\ v\in RH_{s'}(w)\big\}.
\end{align}
For $0\le p_0<q_0\le \infty$ and $v\in A_{\infty}(w)$, by a similar argument to that of \cite[Lemma 4.1]{AMI07}, we have
\begin{align}\label{intervalrsw}
\mathcal{W}_v^w(p_0,q_0):=\left\{p\in(p_0,q_0):\ v\in A_{\frac{p}{p_0}}(w)\cap RH_{\left(\frac{q_0}{p}\right)'}(w)\right\}
=
\left(p_0\mathfrak{r}_v(w),\frac{q_0}{\mathfrak{s}_v(w)}\right).
\end{align}
If $p_0=0$ and $q_0<\infty$, as before, it is understood that the only condition that stays is $v\in RH_{\left(\frac{q_0}{p}\right)'}(w)$. Analogously, if $0<p_0$ and $q_0=\infty$ the only assumption is $v\in A_{\frac{p}{p_0}}(w)$. Finally $\mathcal{W}_v^w(0,\infty)=(0,\infty)$.

\begin{remark}\label{remark:product-weight}
The proof of our main result will use the Calder\'on-Zygmund decomposition from Lemma \ref{lem:CZweighted} with respect to the underlying measure $v(x)\,dw(x)= v(x)\,w(x)\,dx$ where $w\in A_{\infty}(dx)$ and $v\in A_\infty(w)$.  In that scenario it was shown in  \cite[Remark 2.15]{PAII18} that $wv\in A_{\infty}(dx)$ and moreover $r_{vw}\leq r_w\mathfrak{r}_v(w)$. The converse inequality is false in general: let $w(x):=|x|^{n}$ and $v:=w^{-1}$, then one can easily see that $r_w\mathfrak{r}_v(w)=r_w\,s_w=2$ and $r_{vw}=1$. 
\end{remark}

We state some lemma which will be useful in the sequel.
\begin{lemma}\label{ARHsinpesoconpeso}
Given $0<p\leq q<\infty$, let $B\subset\R^n$ be a ball and let $j\geq 1$, the following holds:
 \begin{list}{$(\theenumi)$}{\usecounter{enumi}\leftmargin=1cm \labelwidth=1cm\itemsep=0.2cm\topsep=.0 cm \renewcommand{\theenumi}{\alph{enumi}}}
\item If $v\in A_{\frac{q}{p}}(w)$ then
\[
\left(\dashint_{C_j(B)}|f(x)|^{p}dw(x)\right)^{\frac{1}{p}}
\lesssim
\left(\dashint_{C_j(B)}|f(x)|^{q}d(vw)(x)\right)^{\frac{1}{q}}.
\]

 \item If $v\in RH_{\left(\frac{q}{p}\right)'}(w)$ then
\[
\left(\dashint_{C_j(B)}|f(x)|^{p}d(vw)(x)\right)^{\frac{1}{p}}
\lesssim
\left(\dashint_{C_j(B)}|f(x)|^{q}dw(x)\right)^{\frac{1}{q}}.
\]
\end{list}
\end{lemma}
\begin{proof}
We prove the case $p<q$ (when $p=q$ the proof follows similarly and is left to the interested reader).
Assume first $v\in A_{\frac{q}{p}}(w)$. 
We obtain $(a)$ by applying H\"older's inequality and \eqref{apclassv}
\begin{align*}
\left(\dashint_{C_j(B)}|f|^{p}dw\right)^{\frac{1}{p}}
&=
\left(\dashint_{C_j(B)}|f|^{p}v^{\frac{p}{q}}v^{-\frac{p}{q}}dw\right)^{\frac{1}{p}}
\\
&\lesssim
\left(\dashint_{C_j(B)}|f|^{q}vdw\right)^{\frac{1}{q}}
\left(\dashint_{2^{j+1}B}v^{1-\left(\frac{q}{p}\right)'}dw\right)^{\frac{1}{q}\left(\frac{q}{p}-1\right)}
\\
&\lesssim
\left(\dashint_{C_j(B)}|f|^{q}vdw\right)^{\frac{1}{q}}
\left(\dashint_{2^{j+1}B}vdw\right)^{-\frac{1}{q}}
=
\left(\dashint_{C_j(B)}|f|^{q}d(vw)\right)^{\frac{1}{q}}.
\end{align*}
Next assume $v\in RH_{\left(\frac{q}{p}\right)'}(w)$, we obtain $(b)$ by applying H\"older's inequality and \eqref{rhclassv}
\begin{align*}
\left(\dashint_{C_j(B)}|f|^{p}d(vw)\right)^{\frac{1}{p}}
&=
\left(\frac{w(2^{j+1}B)}{vw(2^{j+1}B)}\right)^{\frac{1}{p}}\left(\dashint_{C_j(B)}|f|^{p}vdw\right)^{\frac{1}{p}}
\\
&\lesssim\left(\dashint_{2^{j+1}B}vdw\right)^{-\frac{1}{p}}\left(\dashint_{C_j(B)}|f|^{q}dw\right)^{\frac{1}{q}}
\left(\dashint_{2^{j+1}B}v^{\left(\frac{q}{p}\right)'}dw\right)^{\frac{1}{p}\frac{1}{\left(\frac{q}{p}\right)'}}
\\
&\lesssim
\left(\dashint_{C_j(B)}|f|^{q}dw\right)^{\frac{1}{q}}.
\end{align*}

\end{proof}

\medskip


\subsection{Square functions and non-tangential maximal functions}
\


In this section, we introduce several auxiliary operators (vertical and conical square functions, non-tangential maximal functions) which will be needed at various points along the proofs. 

Consider, for  $\kappa\geq 1$, the non-tangential maximal function $\mathcal{N}^{\kappa,w}$ defined as
\begin{align}\label{eq:NTkappa}
\mathcal{N}^{\kappa,w}F(x):=\sup_{t>0}\left(\int_{B(x,\kappa t)}|F(y,t)|^2\frac{dw(y)}{w(B(x,t))}\right)^{\frac{1}{2}}.
\end{align}
We write $\Ncal^w$ when $\kappa=1$.  We are particularly interested in the non-tangential maximal functions associated with the heat or Poisson semigroup.
For $f\in L^2(w)$, define
\begin{align}\label{nontangheat}
\Ncal_{\hh}^{\kappa,w}f(x):=\sup_{t>0}\left(\int_{B(x,\kappa t)}\left|e^{-t^2{L_w}}f(y)\right|^2\frac{dw(y)}{w(B(x,t))}\right)^{\frac{1}{2}};
\end{align}
\begin{align}\label{nontangpoisson}
\Ncal_{\pp}^{\kappa,w}f(x):=\sup_{t>0}\left(\int_{B(x,\kappa t)}\left|e^{-t\sqrt{L_w}}f(y)\right|^2\frac{dw(y)}{w(B(x,t))}\right)^{\frac{1}{2}}.
\end{align}
Again, when $\kappa=1$ we write  $\Ncal_{\pp}^{w}$ and $\Ncal_{\hh}^{w}$. We shall obtain weighted boundedness of these operators in Section 4.2. 

We also consider several invariants of the vertical square functions associated with the heat semigroup which were studied in \cite[Sections 5 and 10]{CMR15}:
\begin{align}\label{verticalheat-1}
\mathrm{g}_{\hh}^wf(x):=\left(\int_{0}^{\infty}\left|t^2 L_we^{-t^2{L_w}}f(x)\right|^2\frac{dt}{t}\right)^{\frac{1}{2}};
\end{align}
\begin{align}\label{verticalheat-2}
\mathrm{G}_{1/2,\hh}^wf(x):=\left(\int_{0}^{\infty} \abs{t \nabla(t^2L_w)^{\frac{1}{2}} e^{-t^2 L_w}f(x)}^2 \frac{dt}{t}\right)^{\frac{1}{2}};
\end{align}
\begin{align}\label{verticalheat}
\mathrm{G}_{\hh}^w f(x):=\left(\int_{0}^{\infty} \abs{t \nabla t^2L_w e^{-t^2 L_w}f(x)}^2 \frac{dt}{t}\right)^{\frac{1}{2}}.
\end{align}

Proceeding as in \cite[Propositions 5.1 and 10.1]{CMR15}, by a standard argument, we have the following lemma.
\begin{lemma}\label{lem:wVSF}
Let $L_w$ be a degenerate elliptic operator with $w\in A_2(dx)$ and let $v\in A_{\infty}(w)$.  Then
 \begin{list}{$(\theenumi)$}{\usecounter{enumi}\leftmargin=1cm \labelwidth=1cm\itemsep=0.2cm\topsep=.0cm \renewcommand{\theenumi}{\alph{enumi}}}
 
\item $\mathrm{g}_{\hh}^w$ is bounded on $L^p(vdw)$ for all $p\in \mathcal{W}_v^w\left(p_-(L_w),p_+(L_w)\right)$;

\item $\mathrm{G}_{1/2,\hh}^w$ and $\mathrm{G}_{\hh}^w$ are bounded on $L^p(vdw)$ for all $p\in \mathcal{W}_v^w\left(q_-(L_w),q_+(L_w)\right)$.
\end{list}
\end{lemma}

Now we recall the following conical square functions studied by the authors in \cite{ChMPA16}.
\begin{align}\label{eq:conicalheat}
\Scal_{\hh}^{\alpha,w}f(x):=\left(\iint_{\Gamma^{\alpha}(x)}\left|t^2 L_we^{-t^2{L_w}}f(y)\right|^2\frac{dw(y)\,dt}{tw(B(y,t))}\right)^{\frac{1}{2}},
\end{align}
where $\Gamma^{\alpha}(x):=\{(y,t)\in \R^{n+1}_+:|x-y|<\alpha t\}$ is the cone with vertex at $x$ and aperture $\alpha>0$. When $\alpha=1$ we write $\Gamma(x)$ and $\Scal_{\hh}^{w}$. According to \cite[Proposition 3.1]{ChMPA16}, we have that $\Scal_{\hh}^{w}$ is bounded on $L^p(vdw)$ for all $p\in \mathcal{W}_v^w\left(p_-(L_w),\infty\right)$.

Finally, we introduce  the following ``inhomogeneous'' vertical and conical square functions:
\begin{align}\label{verticalnabaheat}
\widetilde\Grm_{\hh}^wf(x):=\left(\int_{0}^{\infty}\left|\nabla t^2L_we^{-t^2{L_w}}f(x)\right|^2\frac{dt}{t}\right)^{\frac{1}{2}};
\end{align}
\begin{align}\label{conicalheat}
\widetilde\Scal_{\hh}^{w}f(x):=\left(\iint_{\Gamma(x)}\abs{t^{-1}(t^2L_w)e^{-t^2{L_w}}f(y)}^2\frac{dw(y)\,dt}{tw(B(y,t))}\right)^{\frac{1}{2}}.
\end{align}
By non-homogeneity, we mean that the power of $t$ inside the square functions is not in accordance with the order of the operator $L_w$, we are modifying respectively $\Grm_{\hh}^w$ and $\Scal_{\hh}^w$ by removing one power of $t$. The analogues of the above two square functions in other settings turn to be very useful in the study of Riesz transform and Hardy space theory, see for instance  \cite{CD03,HMay09}. 
Sections \ref{section:nono-homo-vertical} and \ref{section:nono-homo-conical} below study the boundedness of $\widetilde\Grm_{\hh}^w$ and $\widetilde\Scal_{\hh}^{w}$ on weighted Sobolev spaces, which plays an essential role in the proof of our main results. 

We finish this subsection by recalling the results about the  reverse inequality of the Riesz transform associated with the operator $L_w$ proved in \cite{CMR15}. The Riesz transform $\nabla L_w^{-\frac{1}{2}}$ associated with the operator $L_w$ can be written as
\begin{align*}
\nabla L_w^{-\frac{1}{2}}=\frac{2}{\sqrt{\pi}}\int_0^{\infty}t\nabla e^{-t^2L_w}\frac{dt}{t},
\end{align*}
 consider also the following square root representation (see for instance \cite{ARR15,CUR15}):
\begin{align}\label{representationsquarerootofL}
\sqrt{L_w}=\frac{2}{\sqrt{\pi}}\int_0^{\infty}tL_w e^{-t^2L_w}\frac{dt}{t}.
\end{align}

\begin{proposition}[{\cite[Proposition 6.1]{CMR15}}]\label{prop:wRR}
Let  $\max\{r_w,(p_-(L_w))_{w,*}\}<p<p_+(L_w)$. Then for all $f\in \Scal$,
\begin{equation*}
\norm{\sqrt{L_w}f}_{L^p(w)} \lesssim \norm{\nabla f}_{L^p(w)}.
\end{equation*}
Furthermore, if $p\in \mathcal{W}_v^w\left(\max\{r_w,p_-(L_w)\},p_+(L_w)\right)$. Then for all $f\in \Scal$,
\begin{equation*}
\norm{\sqrt{L_w}f}_{L^p(vdw)} \lesssim \norm{\nabla f}_{L^p(vdw)}.
\end{equation*}
\end{proposition}

\bigskip

\subsection{Off-diagonal estimates}\label{section:off}
\

\begin{definition}
Let $\{T_t\}_{t>0}$ be a family of sublinear operators and let $1< p< \infty$. Given a doubling measure $\mu$ we say that $\{T_t\}_{t>0}$ satisfies 
$L^p(\mu)-L^p(\mu)$ full off-diagonal estimates, denoted by $T_t\in\mathcal F(L^p(\mu)-L^p(\mu))$, if there exist constants $C,c>0$ such that 
for all closed sets $E$ and $F$, all $f\in L^p(\R^n)$, and all $t>0$ we have 
\begin{equation}\label{off:full}
\left(\int_F \abs{T_t(\chi_E f)}^p d\mu\right)^{\frac{1}{p}} 
\leq C\, 
 e^{-\frac{cd(E,F)^2}{t}} \left(\int_E |f|^p d\mu\right)^{\frac{1}{p}},
\end{equation}
where $d(E,F)=\inf\{|x-y|:x\in E, y\in F\}$.
\end{definition}

Set $\Upsilon(s)=\max\{s,s^{-1}\}$ for $s>0$. Recall that, given a ball $B$, we use the notation $C_j(B)=2^{j+1}B\backslash 2^jB$ for $j\geq 2$, and for any doubling measure $\mu$
$$
\fint_B hd\mu=\frac{1}{\mu(B)}\int_B hd\mu,\quad \fint_{C_j(B)} hd\mu=\frac{1}{\mu(2^{j+1}(B))} \int_{C_j(B)} hd\mu.
$$

\begin{definition}
Given $1\leq p\leq q\leq \infty$ and any doubling measure $\mu$, we say that a family of sublinear operators $\{T_t\}_{t>0}$ satisfies 
$L^p(\mu)-L^q(\mu)$ off-diagonal estimates on balls, denoted by $T_t\in  \mathcal O(L^p(\mu)-L^q(\mu))$, if there exist $\theta_1,\theta_2>0$ and $c>0$ such that for all $t>0$ and for all ball $B$ with radius $r_B$, 
\begin{equation}\label{off:BtoB}
\left(\fint_B \abs{T_t(f \chi_B)}^q d\mu\right)^{\frac{1}{q}} \lesssim \Upsilon\left(\frac{r_B}{\sqrt t}\right)^{\theta_2} \left(\fint_B |f|^p d\mu\right)^{\frac{1}{p}},
\end{equation}
and for $j\geq 2$, 
\begin{equation}\label{off:CtoB}
\left(\fint_B \abs{T_t(f \chi_{C_j(B)})}^q d\mu\right)^{\frac{1}{q}} 
\lesssim 2^{j\theta_1} \Upsilon\left(\frac{2^j r_B}{\sqrt t}\right)^{\theta_2} e^{-\frac{c4^j r_B^2}{t}} \Bigg(\fint_{C_j(B)} |f|^p d\mu\Bigg)^{\frac{1}{p}},
\end{equation}
and
\begin{equation}\label{off:BtoC}
\left(\fint_{C_j(B)} \abs{T_t(f \chi_B)}^q d\mu\right)^{\frac{1}{q}} 
\lesssim 2^{j\theta_1} \Upsilon\left(\frac{2^j r_B}{\sqrt t}\right)^{\theta_2} e^{-\frac{c4^j r_B^2}{t}} \left(\fint_B |f|^p d\mu\right)^{\frac{1}{p}}.
\end{equation}
\end{definition}

Let us recall some results about off-diagonal estimates on balls for the heat semigroup associated with $L_w$.

\begin{lemma}[{\cite[Section 2]{AMII07},\,\cite[Sections 3 and 7]{CMR15}}]\label{lem:ODweighted}
Let $L_w$ be a degenerate elliptic operator with $w\in A_2(dx)$. 
\begin{list}{$(\theenumi)$}{\usecounter{enumi}\leftmargin=1cm \labelwidth=1cm\itemsep=0.2cm\topsep=.2cm \renewcommand{\theenumi}{\alph{enumi}}}

\item If $p_-(L_w)<p\leq q<p_+(L_w)$, then $e^{-t L_w}$ and $(tL_w)^m e^{-t L_w}$, for every $m\in\N$, belong to $\mathcal O(L^p(w)-L^q(w))$. 

\item Let $p_-(L_w)<p\le q<p_+(L_w)$. If $v\in A_{p/p_-(L_w)}(w) \cap RH_{(p_+(L_w)/q)'}(w)$, then $e^{-t L_w}$ and $(tL_w)^m e^{-t L_w}$, for every $m\in\N$, belong to $\mathcal O(L^p(vdw)-L^q(vdw))$.

\item There exists an interval $\mathcal K(L_w)$ such that if $p,q \in \mathcal K(L_w)$, with $p\leq q$, then $\sqrt t\nabla e^{-t L_w} \in \mathcal O(L^p(w)-L^q(w))$. Moreover, denoting by $q_-(L_w)$ and $q_+(L_w)$ the left and right endpoints of $\mathcal K(L_w)$, then $q_-(L_w)=p_-(L_w)$, $2<q_+(L_w)\leq(q_+(L_w))^*_w\leq p_+(L_w)$. 

\item Let $q_-(L_w)<p\le q<q_+(L_w)$. If $v\in A_{p/q_-(L_w)}(w) \cap RH_{(q_+(L_w)/q)'}(w)$, then $\sqrt t\nabla e^{-t L_w}\in \mathcal O(L^p(vdw)-L^q(vdw))$.

\item If $p=q$  and $\mu$ is a doubling measure then $\mathcal{F}(L^p(\mu)-L^p(\mu))$ and $\mathcal{O}(L^p(\mu)-L^p(\mu))$ are equivalent.
\end{list}
\end{lemma}

\begin{remark}\label{rem:stable}
Since off-diagonal estimates on balls are stable under composition (see \cite[Theorem 2.3]{AMII07}), it follows from 
$ (b)$ and $ (d)$ that $\sqrt{t}\nabla tL_w e^{-t L_w}\in \mathcal O(L^p(vdw)-L^q(vdw))$ for $q_-(L_w)<p\le q<q_+(L_w)$ and $v\in A_{p/q_-(L_w)}(w) \cap RH_{(q_+(L_w)/q)'}(w)$. 
\end{remark}
Moreover, in the following result, which is a weighted version of \cite[(5.12)]{MaPAII17} (see also \cite{HMay09}), we show off-diagonal estimates for the family $\{\mathcal{T}_{t,s}\}_{s,t>0}:=\{(e^{-t^2L_w}-e^{-(t^2+s^2)L_w})^M\}_{s,t>0}$, for all $M\in \N$.

\begin{proposition}\label{prop:lebesgueoff-dBQ}
Let $p\in (p_-(L_w),p_+(L_w))$ and let $0<t,s<\infty$. Then, for all sets $E_1,E_2\subset \R^n$ and $f\in L^p(w)$ such that $\supp (f)\subset  E_1$,  we have that $\{\mathcal{T}_{t,s}\}_{s,t>0}$ satisfies the following $L^p(w)-L^p(w)$ off-diagonal estimates: 
\begin{align}\label{AB}
\left\|\chi_{E_2}\mathcal{T}_{t,s}f\right\|_{L^{p}(w)}
\lesssim
\left(\frac{s^2}{t^2}\right)^M
e^{-c\frac{d({E}_1,{E}_2)^2}{t^2+s^2}}\|f\chi_{E_1}\|_{L^{p}(w)}.
\end{align}
In particular, there holds
\begin{equation}\label{boundednesstsr}
\|\mathcal{T}_{t,s}f\|_{L^{p}(w)}\lesssim \left(\frac{s^2}{t^2}\right)^M\|f\|_{L^{p}(w)}.
\end{equation}
\end{proposition}

\begin{proof}
Note that  we have
\begin{align*}
&\left\|\chi_{E_2}\mathcal{T}_{t,s}f\right\|_{L^{p}(w)}
\\
&=
\Bigg\|\chi_{E_2}\left(e^{-t^2L_w}-e^{-(t^2+s^2)L_w}\right)^M f\Bigg\|_{L^{p}(w)}
=
\Bigg\|\chi_{E_2}\Bigg(\int_{0}^{s^2}\partial_r e^{-(r+t^2)L_w}\,dr\Bigg)^M f\,  \Bigg\|_{L^{p}(w)}
\\ \nonumber
&\leq
\int_{0}^{s^2}\!\!\!\dots \int_{0}^{s^2}\!\Big\|\chi_{E_2}
\Big(\Big(\sum_{i=1}^M r_i+M t^2\Big) L_w \Big)^{\!\!M} \!e^{-\big(\sum_{i=1}^M r_i+M t^2\big) L_w}f\Big\|_{L^{p}(w)} \ \!\!\!\!\frac{dr_1 \dots dr_M}{\big(\sum_{i=1}^M r_i+M t^2\big)^M}
\\ \nonumber
&\lesssim
\int_{0}^{s^2}\!\!\!\dots \int_{0}^{s^2}
e^{-c\frac{d({E}_1,{E}_2)^2}{\sum_{i=1}^M r_i+M t^2}}\frac{dr_1 \dots dr_M}{\big(\sum_{i=1}^M r_i+M t^2\big)^M}
\|\chi_{E_1}f\|_{L^{p}(w)}
\\ \nonumber
&
\lesssim
\left(\frac{s^2}{t^2}\right)^M
e^{-c\frac{d({E}_1,{E}_2)^2}{t^2+s^2}}\|\chi_{E_1}f\|_{L^{p}(w)},
\end{align*}
where we have applied the fact that  $(t L_w)^M e^{-tL_w}\in \mathcal{F}(L^{p}(w)- L^{p}(w))$, for all $p\in (p_-(L_w),p_+(L_w))$, (see Lemma \ref{lem:ODweighted}).
\end{proof}

\medskip

We conclude this section by introducing the following  off-diagonal estimates on Sobolev spaces (for non-degenerate elliptic operators see \cite{Au07}).

\begin{lemma}\label{lem:Gsum}
Let $q\in (q_-(L_w),q_+(L_w))$ and $\alpha>0$. Assume that $p$ satisfies $\max\left\{r_w, (q_{-}(L_w))_{w,*}\right\}<p\le q$. Then, for every $(x,t)\in \R^{n+1}_+$, there exists $\theta>0$ such that
\begin{equation}\label{eq:Gsum}
\left(\,\fint_{B(x,\alpha t)} \abs{\nabla e^{-t^2L_w}f}^q dw\right)^{\frac{1}{q}} \lesssim \Upsilon(\alpha)^{\theta}\sum_{j=1}^{\infty} e^{-c4^j} \left(\,\fint_{B(x,2^{j+1}\alpha t)} |\nabla f|^{p} dw\right)^{\frac{1}{p}}.
\end{equation}
\end{lemma}

\begin{proof} 
For simplicity, we write $B:=B(x,\alpha t)$ and $h:=f-f_{4B,w}$, where, for every $\lambda>0$, $f_{\lambda B,w}$ is the average of $f$ in $\lambda B$ with respect to the measure $dw$. By the conservation property,  that  is $ e^{-t^2L_w}1=1$,
$$
\nabla e^{-t^2L_w}f=\nabla e^{-t^2L_w}(f-f_{4B,w})=\sum_{j=1}^{\infty} \nabla e^{-t^2L_w}h_j,
$$
with $h_j:=h\chi_{C_j(B)}$. By Lemma \ref{lem:ODweighted},  for any $q_-(L_w)<q_0<q$, we have that $\sqrt{t}\nabla e^{-tL_w} \in \Ocal(L^{q_0}(w)-L^q(w))$,  then 
\begin{align*} 
\left(\,\fint_{B} \abs{\nabla e^{-t^2L_w}f}^q dw\right)^{\frac{1}{q}}
& \le \sum_{j=1}^{\infty} \left(\,\fint_{B} \abs{\nabla e^{-t^2L_w}h_j}^q dw\right)^{\frac{1}{q}}
\\ &\lesssim \Upsilon(\alpha)^{\theta_2}
\sum_{j=1}^{\infty} \frac{2^{j(\theta_1+\theta_2)} e^{-c4^j}}{t}\left(\fint_{C_j(B)} \abs{ h}^{q_0} dw\right)^{\frac{1}{q_0}}.
\end{align*}
Using the weighted Poincar\'e-Sobolev inequality (see \cite[Theorem 2.1]{CMR15} and also \cite[Theorem 1.6]{FKS82}), we obtain that for any $p>\max\{r_w,(q_0)_{w,*}\}$,
\begin{align*}
\left(\,\fint_{C_j(B)} \abs{ h}^{q_0} dw\right)^{\frac{1}{q_0}} 
&\le \left(\,\fint_{2^{j+1}B} \abs{ f-f_{2^{j+1}B,w}}^{q_0} dw\right)^{\frac{1}{q_0}}+ \sum_{l=2}^{j} \abs{f_{2^{l+1}B,w}-f_{2^{l}B,w}}
\\ &\lesssim 
 \sum_{l=2}^{j} \left(\,\fint_{2^{l+1}B} \abs{ f-f_{2^{l+1}B,w}}^{q_0} dw\right)^{\frac{1}{q_0}}
 \\ &\lesssim 
 \sum_{l=2}^{j} 2^{l}\alpha t \left(\,\fint_{2^{l+1}B} \abs{ \nabla f}^{p} dw\right)^{\frac{1}{p}}.
\end{align*}
Hence,
\begin{multline*}
\left(\,\fint_{B} \abs{\nabla e^{-t^2L_w}f}^q dw\right)^{\frac{1}{q}}
\lesssim
\Upsilon(\alpha)^{\theta}
\sum_{j=1}^{\infty} \frac{e^{-c4^j}}{t}  \sum_{l=2}^{j} 2^{l}t \left(\,\fint_{2^{l+1}B} \abs{ \nabla f}^{p} dw\right)^{\frac{1}{p}}
\\ 
\lesssim \Upsilon(\alpha)^{\theta}
\sum_{j=1}^{\infty} e^{-c4^j} \left(\,\fint_{2^{j+1}B} \abs{ \nabla f}^{p} dw\right)^{\frac{1}{p}}.
\end{multline*}
This completes the proof.
\end{proof}
\subsection{Change of angles}\label{section:change-of-angles}
We shall use two change of angles results. The  first one is a  version of \cite[Proposition 3.30]{MaPAI17} in the weighted degenerate case.
\begin{proposition}{\cite[Proposition A.2]{ChMPA16}}\label{prop:Q} 
Let $w\in A_{\widehat{r}}(dx)$ and $v\in RH_{r'}(w)$ with $1\le \widehat{r},r<\infty$. For every $1\le q\le \widehat{r}$, $0<\alpha\leq 1$ and $t>0$, there holds
\begin{multline}\label{G-alpha}
\int_{\mathbb{R}^n}\left(\int_{B(x,\alpha t)}|h(y,t)| \, \frac{dw(y)}{w(B(y,\alpha t))} \right)^{\frac{1}{q}}v(x)dw(x)
\\ \lesssim 
\alpha^{n\widehat{r}\left(\frac{1}{r}-\frac{1}{q}\right)} 
\int_{\mathbb{R}^n}\left(\int_{B(x,t)}|h(y,t)| \, \frac{dw(y)}{w(B(y,t))} \right)^{\frac{1}{q}}v(x)dw(x).
\end{multline}
\end{proposition}
The second result was proved for  the unweighted non-degenerate   case in \cite{ChanglesAuscher} and for the weighted non-degenerate case in \cite[Proposition 3.2]{MaPAI17}. Consider, for $\alpha>0$, the following operator acting over measurable functions $F$ defined in $\R^{n+1}_+$: 
\begin{align*}
\mathcal{A}_w^{\alpha}F(x):=\left(\iint_{\Gamma^{\alpha}(x)}|F(y,t)|^2\frac{dw(y)\,dt}{tw(B(y,t))}\right)^{\frac{1}{2}},\quad x\in \R^n,
\end{align*}
where $\Gamma^{\alpha}(x)$ is the cone of aperture $\alpha$ and vertex at $x$, $\Gamma^{\alpha}(x)=\{(y,t)\in\R^{n+1}_+:|x-y|<\alpha t\}$.
\begin{proposition}{\cite[Proposition 4.9]{ChMPA16}}\label{prop:alpha}
Let $0< \alpha\leq \beta<\infty$.
\begin{list}{$(\theenumi)$}{\usecounter{enumi}\leftmargin=1cm \labelwidth=1cm\itemsep=0.2cm\topsep=.2cm \renewcommand{\theenumi}{\alph{enumi}}}
\item  For every $w\in A_{\widetilde{r}}(dx)$ and  $v\in A_r(w)$, $1\leq r,\widetilde{r}<\infty$,
there holds
\begin{align}\label{change-alph-1}
\norm{\mathcal{A}_w^{\beta}F}_{L^p(vdw)}\
\leq 
C \left(\frac{\beta}{\alpha}\right)^{\frac{n\,\widetilde{r}\,r}{p}}
 \norm{\mathcal{A}^{\alpha}_w F}_{L^p(vdw)} \quad \textrm{for all} \quad 0<p\leq 2r,
\end{align}
 where   $C\ge 1$ depends on $n$, $p$, $r$, $\widetilde{r}$, $[w]_{A_{\widetilde{r}}(dx)}$, and  $[v]_{A_r(w)}$, but it is independent of  $\alpha$ and $\beta$.

\item  For every $w\in RH_{\widetilde{s}'}(dx)$ and $v\in RH_{s'}(w)$, $1\leq s,\widetilde{s}<\infty$, there holds
\begin{align}\label{change-alph-2}
\norm{\mathcal{A}^{\alpha}_w F}_{L^p(vdw)}
\leq  
C\left(\frac{\alpha}{\beta}\right)^{\frac{n}{s\,\widetilde{s}\,p}} \norm{\mathcal{A}^{\beta}_w F}_{L^p(vdw)}
\quad \text{for all} \quad \frac{2}{s}\leq p<\infty,
\end{align}
 where   $C\ge 1$ depends on $n$, $p$, $s$, $\widetilde{s}$, $[w]_{RH_{\widetilde{s}'}(dx)}$, and  $[v]_{RH_{s'}(w)}$, but it is independent of  $\alpha$ and $\beta$.
\end{list}
\end{proposition}

\subsection{Calder\'on-Zygmund decomposition on Sobolev spaces}

Our proofs rely on the following Calder\'on-Zygmund decomposition on Sobolev spaces. 
\begin{lemma}[{\cite[Lemma 6.6]{AMIII06}}]\label{lem:CZweighted}
Let $n\ge 1$, $\alpha>0$, $\varpi\in A_{\infty}(dx)$, and let $1\le p<\infty$ be such that $\varpi\in A_p(dx)$. Assume that $f\in \Scal $ is such that $\|\nabla f\|_{L^p(\varpi)}<\infty$.  Then, there exist a collection of balls $\{B_i\}_{i}$ with radii $r_{B_i}$, smooth functions $\{b_i\}_i$ and  a function $g\in L_{\loc}^1(\varpi)$ such that
\begin{equation}\label{CZ:decomp}
f=g+\sum_{i} b_i
\end{equation}
and the following properties hold:
\begin{equation}\label{CZ:g}
|\nabla g(x)| \leq C\alpha,\quad \text{for }\mu\text{-a.e. }x,
\end{equation}
\begin{equation}\label{CZ:b}
\supp b_i\subset B_i\quad \text{and}\quad \int_{B_i}|\nabla b_i|^p d\varpi\leq C\alpha^p \varpi(B_i),
\end{equation}
\begin{equation}\label{CZ:sum}
\sum_{i}\varpi(B_i)\leq \frac{C}{\alpha^p}\int_{\R^n}|\nabla f|^p d\varpi,
\end{equation}
\begin{equation}\label{CZ:overlap}
\sum_{i} \chi_{4B_i}\leq N,
\end{equation}
where $C$ and $N$ depend only on $n$, $p$, and $\varpi$. In addition,  for $1\le q<p_{\varpi}^*$, where $p_{\varpi}^*$ is defined in \eqref{p_w^*},
\begin{equation}\label{CZ:PS}
\left(\fint_{B_i} |b_i|^q d\varpi\right)^{\frac{1}{q}} \lesssim \alpha\, r_{B_i}.
\end{equation}
\end{lemma}


\subsection{Non-homogeneous vertical square function}\label{section:nono-homo-vertical}
In this section, we study the weighted boundedness of $\widetilde\Grm_{\hh}^w$. Our result is the following.
\begin{theorem}\label{thm:wtG}
Let $w\in A_2(dx)$ and let $L_w$ be a degenerate elliptic operator. Given  $v\in A_{\infty}(w)$, 
assume that $\mathcal{W}_v^w\left(\max\{r_w, q_{-}(L_w)\},q_+(L_w)\right)\neq \emptyset$. Then,  for every $f\in \Scal$  and $p\in\mathcal{W}_v^w\left(\max\{r_w, (q_{-}(L_w))_{w,*}\},q_+(L_w)\right)$, it holds,
\begin{equation}\label{eq:Lp-wtG}
\big\|\widetilde\Grm_{\hh}^{w}f\big\|_{L^p(vdw)} \lesssim \norm{\nabla f}_{L^p(vdw)}.
\end{equation}
\end{theorem}
Before starting with the proof, we make some remarks and prove Lemma \ref{lemma:lastestimate} (stated below).  These results not only will be useful in this proof but also in the remainder of the paper.
\begin{remark}\label{remark:intervalnotempty} 
Given $w\in A_2(dx)$ and $v\in A_{\infty}(w)$, let $0<q_-<q_+<\infty$ and  $p>\mathfrak{r}_v(w)\max\{r_w,(q_-)_{w,*}\}$. Assuming that 
$$
\big(\mathfrak{r}_v(w)\max\{r_w,q_-\},q_+/\mathfrak{s}_v(w)\big)=\mathcal{W}_v^w(\max\{r_w,q_-\},q_+)\neq \emptyset,
$$
we claim  that
\begin{align}\label{intervalnotempty2}
\big(\mathfrak{r}_v(w)\max\{r_w,q_-\},\min\{q_+/\mathfrak{s}_v(w),p_{vw}^*\}\big)\neq \emptyset,
\end{align}
where we recall that  by Remark \ref{remark:product-weight} $vw\in A_{\infty}(dx)$, and $p_{vw}^*$ is defined in \eqref{p_w^*}.
Indeed,  since by hypothesis $\mathfrak{r}_v(w)\max\{r_w,q_-\}<q_+/\mathfrak{s}_v(w)$, this can be seen from the fact that
\begin{align}\label{intervalnotempty}
\mathfrak{r}_v(w)\max\{r_w,q_-\}<p_{vw}^*.
\end{align}
 To prove \eqref{intervalnotempty}, we distinguish two cases. 
If $\mathfrak{r}_v(w)\max\{r_w,q_-\}=\mathfrak{r}_v(w)r_w$, since we are taking $p$ such that $p>\mathfrak{r}_v(w)\max\{r_w,(q_-)_{w,*}\}$  and since $(q_-)_{w,*}\leq q_-$ (see \eqref{p_{w,*}}), then
\[
\mathfrak{r}_v(w)\max\{r_w,q_-\}=\mathfrak{r}_v(w)\max\{r_w,(q_-)_{w,*}\}<p<p_{vw}^*.
\]
If now $\mathfrak{r}_v(w)\max\{r_w,q_-\}=\mathfrak{r}_v(w)q_-$, we can assume that $nr_{vw}>p$ (otherwise $p_{vw}^*=\infty$ and the inequality is trivial). Hence, by hypothesis and by \eqref{p_w^*},
\begin{align}\label{lateruse}
\frac{1}{p_{vw}^*}&=\frac{1}{p}-\frac{1}{nr_{vw}}
<\frac{1}{\mathfrak{r}_v(w)(q_{-})_{w,*}}-\frac{1}{nr_{vw}}
=
\frac{nr_w+q_-}{\mathfrak{r}_v(w)q_{-}nr_w}-\frac{1}{nr_{vw}}
\\\nonumber&
=
\frac{1}{\mathfrak{r}_v(w)q_{-}}-\frac{1}{nr_{vw}}\left(1-\frac{r_{vw}}{r_w\mathfrak{r}_v(w)}\right)
\leq 
\frac{1}{\mathfrak{r}_v(w)q_{-}}=\frac{1}{\mathfrak{r}_v(w)\max\{r_w,q_-\}}.
\end{align}
\end{remark}

\begin{remark}\label{remark:Kolmogorov}
Let $\{B_i\}_{i}$ be a collection of balls with bounded overlap, $w\in A_{\infty}(dx)$, and $v\in A_{\infty}(w)$. Besides, consider
$1< \widetilde{p}<\infty$,  $u\in L^{\widetilde{p}'}(vdw)$ such that $\|u\|_{L^{\widetilde{p}'}(vdw)}= 1$, and $\mathcal{M}^{vw}$  the weighted maximal operator defined as
$$
\mathcal{M}^{vw}f(x):=\sup_{B\ni x} \dashint_{B}|f(y)|d(vw)(y),
$$
then, by Kolmogorov's inequality,
we have that
\begin{multline}\label{maximal-u}
\left(\sum_{i} \int_{B_i}\left(\mathcal{M}^{vw}(|u|^{\widetilde{p}'})\right)^{\frac{1}{\widetilde{p}'}}vdw\right)^{\widetilde{p}}
\lesssim
\left(\int_{\cup_{i}B_i}\left(\mathcal{M}^{vw}(|u|^{\widetilde{p}'})\right)^{\frac{1}{\widetilde{p}'}}vdw\right)^{\widetilde{p}}
\\
\lesssim
vw\Big(\bigcup_{i}B_i\Big)\|u\|_{L^{\widetilde{p}'}(vdw)}^{\widetilde{p}}\lesssim vw\Big(\bigcup_{i}B_i\Big).
\end{multline}
\end{remark}

We next state a technical lemma which will be used 
several times. We notice that the statement, which may appear slightly clumsy, is written so that it can be easily invoked in  some of our proofs.
\begin{lemma}\label{lemma:lastestimate}
Given $w\in A_2(dx)$ and $v\in A_{\infty}(w)$, fix $\alpha>0$, $1< p_1<\infty$, $\{B_i\}_{i}$ a collection of balls in $\R^n$ with bounded overlap.
Assume that there is a sequence of positive numbers $\{\mathcal{I}_{ij}\}_{i,j}$ (whose significance will become clear when applying the result) so that 
\begin{align}\label{hypothesislastestimate}
\mathcal{I}_{ij}\,\le \widehat{C} \alpha vw(2^{j+1}B_i)^{\frac{1}{p_1}}
2^{-j(2M-\widetilde{C})},\quad j\geq 4,
\end{align}
where $\widehat{C}, \widetilde{C}$ are fixed (harmless) constants, and $2M>\widetilde{C}+nr_w\mathfrak{r}_v(w)$, then 
$$
\sup_{\|u\|_{L^{p_1'}(vdw)}=1}\sum_{i}\sum_{j\geq 4}\mathcal{I}_{ij}\|u\chi_{C_j(B_i)}\|_{L^{p_1'}(vdw)}
\lesssim
\alpha  vw\Big(\bigcup_{i}B_i\Big)^{\frac1{p_1}}.
$$
\end{lemma}
\begin{proof}
Fix $u$ so that $\|u\|_{L^{p_1'}(vdw)}=1$. 
Note that we can find $p>r_w$, $q>\mathfrak{r}_v(w)$ so that $2M>\widetilde{C}+nr$ with $r=pq$. In particular, $w\in A_p(dx)$, $v\in A_q(w)$ and we have \eqref{pesosineq:Ap} at our disposal. 
This, together with \eqref{hypothesislastestimate} and \eqref{maximal-u} with $\widetilde{p}=p_1$, allows us   to show  that
\begin{align*}
&\sum_{i}\sum_{j\geq 4}\mathcal{I}_{ij}\|u\chi_{C_j(B_i)}\|_{L^{p_1'}(vdw)}\\
&\qquad\lesssim
\alpha
\sum_{i}\sum_{j\geq 4}vw(B_i)2^{-j(2M-\widetilde{C}-n r)}\left(\dashint_{C_j(B_i)}|u(x)|^{p_1'}d(vw)(x)\right)^{\frac{1}{p_1'}}
\\
&\qquad\lesssim
\alpha\sum_{i}vw(B_i)\inf_{x\in B_i}\left(\mathcal{M}^{vw}\big(|u|^{p_1'}\big)(x)\right)^{\frac{1}{p_1'}}
\\
&\qquad\lesssim
\alpha
\sum_{i}\int_{B_i}\left(\mathcal{M}^{vw}\big(|u|^{p_1'}\big)(x)\right)^{\frac{1}{p_1'}}v(x)dw(x)
\lesssim
\alpha vw\Big(\bigcup_{i}B_i\Big)^{\frac1{p_1}}.
\end{align*}
This readily leads to the desired estimate.
\end{proof}
\subsubsection*{Proof of Theorem \ref{thm:wtG}}
Throughout the proof fix $w\in A_2(dx)$ and denote   $q_-:=q_-(L_w)$ and $q_+:=q_+(L_w)$. 

If $p\in\mathcal{W}_v^w\left(\max\{r_w,q_-\},q_+\right)$, then \eqref{eq:Lp-wtG} follows easily from Lemma \ref{lem:wVSF} and  Proposition \ref{prop:wRR}. Indeed, we have
\begin{multline}\label{verticalGtildesmallrange}
\big\|\widetilde\Grm_{\hh}^w f\big\|_{L^p(vdw)}=
\Bigg\|\left(\int_{0}^{\infty} 
\big|t \nabla (t^2L_w)^{\frac{1}{2}} e^{-t^2 L_w}\big(\sqrt{L_w}f\big)\big|^2 \frac{dt}{t}\right)^{\frac{1}{2}}\Bigg\|_{L^p(vdw)}
\\
\lesssim \big\|\sqrt{L_w} f\big\|_{L^p(vdw)} \lesssim \|\nabla f\|_{L^p(vdw)}.
\end{multline}

In accordance with \eqref{intervalrsw}, to go below $\mathfrak{r}_v(w)\max\{r_w,q_-\}$, we shall show that if
{  $p$ satisfies}
\begin{align}\label{choicepvertical}
\mathfrak{r}_v(w)\max\left\{r_w, (q_-)_{w,*}\right\}<p<\mathfrak{r}_v(w)\max\{r_w,q_-\},
\end{align}
then  for any  $\alpha>0$ and $f\in\mathcal S$ it follows that 
 \begin{align}\label{weaknormvertical}
vw\left(\bbr{x\in \R^n: \widetilde\Grm_{\hh}^w f(x)>\alpha}\right) \lesssim \frac1{\alpha^p} \int_{\R^n} |\nabla f|^p vdw.
\end{align}
Hence using interpolation between Sobolev spaces (see \cite{Ba09}), we shall conclude the desired estimate.

In order to prove \eqref{weaknormvertical}, we apply to $f$ the Calder\'on-Zygmund decomposition in Lemma \ref{lem:CZweighted}  at height $\alpha>0$ for the product weight $vw$ (recall that $r_{vw}\leq r_{w}\mathfrak{r}_v(w)<p$, see Remark \ref{remark:product-weight}). 
Thus by \eqref{CZ:decomp}
\begin{multline*}
vw\left(\bbr{x\in \R^n: \widetilde\Grm_{\hh}^w f(x)>\alpha}\right)
\leq 
vw\left(\bbr{x\in \R^n: \widetilde\Grm_{\hh}^w g(x)>\frac{\alpha}{3}}\right)
\\
+
vw\Big(\Big\{x\in \R^n: \widetilde\Grm_{\hh}^w \Big(\sum_{i} b_i\Big)(x)>\frac{2\alpha}{3}\Big\}\Big) =:I+II.
\end{multline*}

Note that by {  Remark \ref{remark:intervalnotempty}} we can pick $q$ such that 
\begin{align}\label{pickq}
\mathfrak{r}_v(w)\max\{r_w,q_-\}<q<\min\Big\{\frac{q_+}{\mathfrak{s}_v(w)},p_{vw}^*\Big\}.
\end{align}
Keeping this choice of $q$, by \eqref{verticalGtildesmallrange} we have $\big\|\widetilde\Grm_{\hh}^w f\big\|_{L^{q}(vdw)}\lesssim  \|\nabla f\|_{L^{q}(vdw)}$. Besides, since $p<q$ (see \eqref{choicepvertical}), properties \eqref{CZ:g}-\eqref{CZ:overlap} yield
$$
I \lesssim \frac1{\alpha^{q}}\int_{\R^n} |\widetilde\Grm_{\hh}^w  g|^{q} vdw 
\lesssim \frac1{\alpha^{q}}\int_{\R^n} |\nabla g|^{q} vdw \lesssim \frac1{\alpha^p} \int_{\R^n} |\nabla f|^p vdw.
$$

To estimate  term $II$, for every $k\in \Z$, let $r_i:=2^{k}$ if $2^k\le r_{B_i}< 2^{k+1}$.  Then,
\begin{align*}
II&\le\,
vw\Big(\bigcup_{i} 16B_i\Big)
\\
&\quad+
vw\Bigg(\Bigg\{ x\in \R^n:\Bigg(\int_{0}^{\infty} \Big|t^2 \nabla L_w e^{-t^2 L_w}\Big(\sum_{i: r_i\le t}b_i\Big)(x)\Big|^2 \frac{dt}{t}\Bigg)^{\frac{1}{2}}>\frac{\alpha}{3}\Bigg\}\Bigg)
\\ &\quad+
vw\Bigg(\Bigg\{ x\in \R^n\setminus \bigcup_{i}16B_i: \Bigg(\int_{0}^{\infty} \Big|t^2 \nabla L_w e^{-t^2 L_w}\Big(\sum_{i: r_i>t}b_i\Big)(x)\Big|^2 \frac{dt}{t}\Bigg)^{\frac{1}{2}}>\frac{\alpha}{3}\Bigg\}\Bigg)
\\& \lesssim
 \frac{1}{\alpha^p}\int_{\R^n}|\nabla f|^pvdw+II_1+II_2,
\end{align*}
where we have used \eqref{doublingcondition} and \eqref{CZ:sum}. 

In order to estimate term $II_1$, write
\begin{multline}\label{1000}
\Bigg(\int_{0}^{\infty} \Big|\nabla t^2L_w e^{-t^2 L_w}\Big(\sum_{i: r_i\le t}b_i\Big)(x)\Big|^2 \frac{dt}{t}\Bigg)^{\frac{1}{2}}
\\
=
 \Bigg(\int_{0}^{\infty} \Big|t^3 \nabla L_w e^{-t^2 L_w}\Big(\frac1t\sum_{i: r_i \le t} b_i\Big)(x)\Big|^2 \frac{dt}{t}\Bigg)^{\frac{1}{2}}
 =\left(\int_0^{\infty}|T_tf_t(x)|^2\frac{dt}{t}\right)^{\frac{1}{2}},
\end{multline}
where $T_{t}:=t^3 \nabla L_w e^{-t^2 L_w}$ and $f_t(x):=\frac1t\sum_{i: r_i \le t} b_i(x)$.

Moreover, note that 
$
2\in \left(\max\{r_w,q_-\},q_+\right),
$
 then, by Remark \ref{rem:stable}, for every $v_0\in A_{2/\max\{r_w,q_-\}}(w)\cap RH_{\left(q_+/2\right)'}(w)$ we have $t^{\frac{3}{2}} \nabla L_w e^{-tL_w}\in \mathcal{O}(L^2(v_0dw)-L^2(v_0dw))$. In particular, $T_t$ is bounded from $L^2(v_0dw)$ to $L^2(v_0dw)$.
Consequently,
\begin{align*}
\left\|\left(\int_{0}^{\infty} |T_{t} f_t|^2 \frac{dt}{t}\right)^{\frac{1}{2}}\right\|_{L^2(v_0dw)}^2
 &=
\int_{0}^{\infty} \int_{\R^n}|T_{t} f_t|^2  v_0dw\frac{dt}{t}
\\ &\lesssim
\int_{0}^{\infty} \int_{\R^n}|f_t|^2  v_0dw \frac{dt}{t} =\norm{\left(\int_{0}^{\infty} |f_t|^2 \frac{dt}{t}\right)^{\frac{1}{2}}}_{L^2(v_0dw)}^2.
\end{align*}
Now, by extrapolation (see \cite[Theorem A.1]{ChMPA16} and also \cite[Theorem 3.31]{CUMPe11}), we obtain that for $\widetilde{v}\in A_{\infty}(w)$ and any $\widetilde{q}\in \mathcal{W}_{\widetilde{v}}^w(\max\{r_w,q_-\},q_+)$,
\begin{equation}\label{eq:2q}
\norm{\left(\int_{0}^{\infty} |T_{t} f_t|^2 \frac{dt}{t}\right)^{\frac{1}{2}}}_{L^{\widetilde{q}}(\widetilde{v}dw)}
\lesssim \norm{\left(\int_{0}^{\infty} |f_t|^2 \frac{dt}{t}\right)^{\frac{1}{2}}}_{L^{\widetilde{q}}(\widetilde{v}dw)}.
\end{equation}
In particular the above inequality holds for our choices of $q$ and $v$.

Next, the proof follows much as in \cite[p. 543]{AC05}, but we write the details for the sake of completeness. Consider the following sum:
$$
\beta_k:=\sum_{i:r_i=2^k} \frac{b_i}{r_i},
$$
and note that
$$
f_t=\frac1t\sum_{i: r_i \le t} b_i = \sum_{k:2^k\le t} \frac{2^k}t\sum_{i: r_i =2^k} \frac{b_i}{r_i} = \sum_{k:2^k\le t} \frac{2^k}t \beta_k.
$$
By  Cauchy-Schwartz inequality, for every $t>0$,
$$
|f_t|^2 \le \Bigg(\sum_{k:2^k\le t} \frac{2^k}t |\,\beta_k|^2\Bigg) \Bigg(\sum_{k:2^k\le t} \frac{2^k}t \Bigg)\lesssim \sum_{k:2^k\le t} \frac{2^k}t |\,\beta_k|^2 =\sum_{k\in \Z} \frac{2^k}t |\,\beta_k|^2\chi_{[2^k,\infty)}(t),
$$
and hence,
$$
\int_{0}^{\infty} |f_t|^2 \frac{dt}{t} \lesssim \sum_{k\in \Z} \int_{2^k}^{\infty}  \frac{2^k}t \frac{dt}{t} \, |\,\beta_k|^2 =\sum_{k\in \Z}  |\,\beta_k|^2.
$$
Using the bounded overlap property \eqref{CZ:overlap}, the fact that $r_i\approx r_{B_i}$, and also \eqref{CZ:PS}, we have
\begin{align*}
\left\|\left(\int_{0}^{\infty} |f_t|^2 \frac{dt}{t}\right)^{\frac{1}{2}}\right\|_{L^q(vdw)}^q 
&\lesssim
\Bigg\|\Big(\sum_{k\in \Z}  |\,\beta_k|^2\Big) ^{\frac{1}{2}} \Bigg\|_{L^q(vdw)}^q \!\!
\lesssim\!\! \int_{\R^n} \sum_{i} \frac{|b_i|^q}{r_i^q} vdw
\\
&\lesssim \alpha^q \sum_{i} vw(B_i) \lesssim \alpha^{q-p} \int_{\R^n} |\nabla f|^p vdw.
\end{align*}
This estimate, \eqref{1000}, and \eqref{eq:2q} with $q$ and $v$, yield as desired
$$
II_1 \lesssim \frac{1}{\alpha^q} \left\|\left(\int_{0}^{\infty} |T_{t} f_t|^2 \frac{dt}{t}\right)^{\frac{1}{2}}\right\|_{L^q(vdw)}^q
\!\!\!\lesssim \frac{1}{\alpha^q} \Bigg\|\Big(\sum_{k\in \Z}  |\,\beta_k|^2\Big) ^{\frac{1}{2}} \Bigg\|_{L^q(vdw)}^q
\!\!\!\lesssim \frac{1}{\alpha^p} \int_{\R^n} |\nabla f|^p vdw.
$$

In order to estimate term $II_2$,
notice that
$$
\Bigg(\int_{0}^{\infty} \Big|t^2 \nabla L_w e^{-t^2 L_w}\Big(\sum_{i: r_i> t}b_i\Big)\Big|^2 \frac{dt}{t}\Bigg)^{\frac{1}{2}}
\!\!\!\le\! \sum_{i} \Bigg(\int_{0}^{r_i} \Big|t^2 \nabla L_w e^{-t^2 L_w}b_i\Big|^2 \frac{dt}{t}\Bigg)^{\frac{1}{2}}\!=: \sum_{i} T_i b_i.
$$
Then, by duality, we have
\begin{align*}
II_2 &\lesssim \frac{1}{\alpha^{q}} \int_{\R^n\setminus \bigcup_{i}16B_i} \Big|\sum_{i} T_i b_i(x)\Big|^{q}v(x)dw(x)
\\
&
\leq \frac{1}{\alpha^{q}} \Bigg(\sup_{\norm{u}_{L^{q'}(vdw)}=1} \sum_{i}\int_{\R^n \setminus \bigcup_{i}16B_i}  \left|T_i b_i(x)\right|\,  \left|u(x)\right|\, v(x)dw(x) \Bigg)^{q} 
\\
&
\leq \frac{1}{\alpha^{q}} \Bigg(\sup_{\norm{u}_{L^{q'}(vdw)}=1} \sum_{i}\sum_{j\geq 4}\int_{C_j(B_i)}  \left|T_i b_i(x)\right|\,  \left|u(x)\right|\, v(x)dw(x) \Bigg)^{q} 
\\ &\lesssim
\frac1{\alpha^{q}} \Bigg(\sup_{\norm{u}_{L^{q'}(vdw)}=1} \sum_{i} \sum_{j\ge 4} \norm{T_i b_i}_{L^{q}(C_j(B_i),vdw)} \norm{u}_{L^{q'}(C_j(B_i),vdw)} \Bigg)^{q}.
\end{align*}

{ To  estimate $ \norm{T_i b_i}_{L^{q}(C_j(B_i),vdw)} $, we pick $p_0$ close enough to $q_-$, and $q_0$ close enough to $q_+$  such that 
\begin{equation}\label{p0q0non-homG}
q_-<p_0<2<q_0<q_{+},  \quad \text{and} \quad v\in A_{\frac{q}{p_0}}(w)\cap RH_{\left(\frac{q_0}{q}\right)'}(w).
\end{equation}
Note that  $\mathcal{W}_v^w(q_-,q_+)\neq \emptyset$ since by assumption $\mathcal{W}_v^w(\max\{r_w,q_-\},q_+)\neq \emptyset$ and $\mathcal{W}_v^w(\max\{r_w,q_-\},q_+)\subset \mathcal{W}_v^w(q_-,q_+)$. 
 Notice also that applying Remark \ref{rem:stable} with $v\equiv 1$,  we have $t^{\frac{3}{2}} \nabla L_w e^{-t L_w} \in \Ocal(L^{p_0}(w)-L^{q_0}(w))$.
Then, by Minkowski's integral inequality, Lemma \ref{ARHsinpesoconpeso} $(a)$ and $(b)$ (see \eqref{p0q0non-homG}), \eqref{CZ:PS} (see \eqref{pickq}), and recalling that $r_i\approx r_{B_i}$, for $j\geq 2$,
\begin{align*}
\left\|T_i\right.& b_i\left.\right\|_{L^{q}(C_j(B_i),vdw)}  
 =vw(2^{j+1}B)^{\frac{1}{q}}
\Bigg(\dashint_{C_j(B_i)} \left(\int_{0}^{r_i} \Big|t^3 \nabla L_w e^{-t^2 L_w}b_i\Big|^2 \frac{dt}{t^3}\right)^{\frac{q}{2}}  d(vw)\Bigg)^{\frac{1}{q}}
\\
& \lesssim vw(2^{j+1}B_i)^{\frac{1}{q}}
\Bigg(\dashint_{C_j(B_i)} \left(\int_{0}^{r_i} \abs{t^3 \nabla L_w e^{-t^2 L_w}b_i}^2 \frac{dt}{t^3}\right)^{\frac{q_0}{2}}  dw\Bigg)^{\frac{1}{q_0}}
\\
& \leq vw(2^{j+1}B_i)^{\frac{1}{q}}
\left(\int_{0}^{r_i}\Bigg(\dashint_{C_j(B_i)}  \abs{t^3 \nabla L_w e^{-t^2 L_w}b_i}^{q_0}  dw\Bigg)^{\frac{2}{q_0}}\frac{dt}{t^3}\right)^{\frac{1}{2}} 
\\ & \lesssim 
2^{j\theta_1}  vw(2^{j+1}B_i)^{\frac{1}{q}} \br{\fint_{B_i} |b_i|^{p_0} dw}^{\frac{1}{p_0}} \Bigg(\int_{0}^{r_i}  \br{\frac{2^jr_{B_i}}{t}}^{2\theta_2} e^{-c\frac{4^j r_{B_i}^2}{t^2}}  \frac{dt}{t^3}\Bigg)^{\frac{1}{2}}
\\ & \lesssim
e^{-c4^j}\,vw(2^{j+1}B_i)^{\frac{1}{q}}\br{\fint_{B_i} \left|\frac{b_i}{r_{B_i}}\right|^{q} d(vw)}^{\frac{1}{q}}
 \lesssim
e^{-c4^j}\alpha  \,vw(2^{j+1}B_i)^{\frac{1}{q}}.
\end{align*}
%

 Now we use Lemma \ref{lemma:lastestimate} with $p_1=q$, $\mathcal{I}_{ij}=\|T_ib_i\|_{L^q(C_j(B_i),vdw)}$, $\{B_i\}_{i}$ the collection of balls given by Lemma \ref{lem:CZweighted}, and with $e^{-c4^j}$ replacing $2^{-j(2M-\widetilde{C})}$ (consequently $M$ and $\widetilde{C}$ do not play any role here).
Therefore, Lemma \ref{lemma:lastestimate} and \eqref{CZ:sum} imply
\begin{align*}
II_2 \lesssim
 vw\Big(\bigcup_{i}B_i\Big)
\lesssim \frac{1}{\alpha^p} \int_{\R^n} |\nabla f|^p vdw.
\end{align*}
Collecting the above estimates, we get the desired result.
\qed

\subsection{Non-homogeneous conical square function }\label{section:nono-homo-conical}
In this section, we shall prove weighted boundedness in Sobolev spaces for the inhomogeneous conical square function  $\widetilde\Scal_{\hh}^{w}$ defined in \eqref{conicalheat}.
The analogous result  for elliptic operators was studied in \cite{HMay09} for the Riesz transform characterization of Hardy spaces. See also \cite{PA17} for the the Riesz transform characterization of weighted Hardy spaces.
Our result is stated as follows.
\begin{theorem}\label{thm:boundednesswidetildeS}
Given $w\in A_2(dx)$, $v\in A_{\infty}(w)$,
assume that
\begin{align}\label{notempty}
\mathcal{W}_v^w\left(\max\left\{r_w,q_-(L_w)\right\},q_+(L_w)\right)\neq \emptyset.
\end{align}
Then, for every $h\in \mathcal S$ and $p\in\mathcal{W}_v^w\left(\max\left\{r_w,(p_-(L_w))_{w,*}\right\},p_+(L_w)\right)$, it holds
\begin{align}\label{boundednesswidetilde}
\big\|\widetilde{\Scal}_{\hh}^wh\big\|_{L^p(w)}\lesssim \|\nabla h\|_{L^p(w)}.
\end{align}
\end{theorem}
In order to prove this theorem, we shall use Lemma \ref{lemma:arbi}  and Proposition \ref{prop:widetildeS-heatS}. Lemma \ref{lemma:arbi} will be also useful in the proof of Proposition \ref{prop:timeD} (all these results are stated below). 
 \begin{lemma}\label{lemma:arbi}
Let $w\in A_2(dx)$ and $v\in A_{\infty}(w)$ be such that $\mathcal{W}_v^w(q_-(L_w),q_+(L_w))\neq \emptyset$, and let  
$$
p\in (\mathfrak{r}_v(w)\max\{r_w,(q_-(L_w))_{w,*}\},\mathfrak{r}_v(w)\max\{r_w,q_-(L_w)\}).
$$
Given $\alpha >0$ and $f\in \mathcal{S}$ 
such that $\|\nabla f\|_{L^p(vdw)}<\infty$, let $\{b_i\}_i$ be the collection of smooth functions from Lemma \ref{lem:CZweighted} (applied to $f$, $p$, $\alpha$, and  ${\varpi}=vw$). Write $\widetilde{b}:=\sum_{i=1}^{\infty}A_{r_{B_i}}b_i$, where $A_{r_{B_i}}:=I-(I-e^{-r_{B_i}^2L_w})^M$ and 
$M\in \N$ is arbitrarily large. 
 Then, for $p_1\in \mathcal{W}_v^w(q_-(L_w),q_+(L_w))$ such that
$1\leq p_1<p_{vw}^*$ (note that following \eqref{lateruse} we get that $\mathfrak{r}_v(w)q_-(L_w)<p_{vw}^*$), there holds
\begin{align*}
\|\nabla\widetilde{b}\,\|_{L^{p_1}(vdw)}^{p_1}\lesssim 
\alpha^{p_1-p}\|\nabla f\|_{L^p(vdw)}^p.
\end{align*}
\end{lemma}
\begin{proof}   First of all denote $q_-:=q_-(L_w)$ and $q_+:=q_+(L_w)$. By duality and expanding $A_{r_{B_i}}$, we have
\begin{multline*}
\|\nabla\widetilde{b}\|_{L^{p_1}(vdw)}^{p_1}
=\int_{\R^n} \Big|\nabla\Big(\sum_{i}\sum_{k=1}^{M}{C_{k,M}}e^{-kr_{B_i}^2L_w}b_i\Big)\Big|^{p_1}vdw
\\
\lesssim\sup_{\|u\|_{L^{p_1'}(vdw)}=1}\left(\sum_{k=1}^{M}\sum_{i}\int_{\R^n} \Big|\sqrt{k}r_{B_i}\nabla e^{-kr_{B_i}^2L_w}\Big(\frac{b_i}{r_{B_i}}\Big)\Big|\,|u|vdw\right)^{p_1}.
\end{multline*}
Besides, recall that by hypothesis $v\in A_{\frac{p_1}{q_-}}(w)\cap RH_{\left(\frac{q_+}{p_1}\right)'}(w)$ (see \eqref{intervalrsw}) and hence $\sqrt{\tau}\nabla e^{-\tau L_w}\in \mathcal{O}(L^{p_1}(vdw)-L^{p_1}(vdw))$. Using this, \eqref{CZ:PS}, and also \eqref{doublingcondition}, 
\begin{align*}
\int_{\R^n}& \left|\sqrt{k}r_{B_i}\nabla e^{-kr_{B_i}^2L_w}\left(\frac{b_i}{r_{B_i}}\right)\right|\,|u|vdw
\\\nonumber
&\lesssim \sum_{j\geq 1}vw(2^{j+1}B_i)\!\left(\dashint_{C_j(B_i)}\!\left|\!\sqrt{k}r_{B_i}\nabla e^{-kr_{B_i}^2L_w}\!\!\left(\frac{b_i}{r_{B_i}}\right)\!\right|^{p_1}\!\!\!\!d(vw)\!\right)^{\!\!\frac{1}{p_1}}\!\!\left(\dashint_{C_j(B_i)}\!\!|u|^{p_1'}d(vw)\right)^{\!\!\frac{1}{p_1'}}
\\\nonumber
&\lesssim
\sum_{j\geq 1}e^{-c4^j}vw(B_i)\left(\dashint_{B_i} \left|\frac{b_i}{r_{B_i}}\right|^{p_1}d(vw)\right)^{\frac{1}{p_1}}\,\inf_{x\in B_i}\left(\mathcal{M}^{vw}(|u|^{p_1'})(x)\right)^{\frac{1}{p_1'}}
\\
&
\lesssim \alpha \int_{B_i}\left(\mathcal{M}^{vw}(|u|^{p_1'})\right)^{\frac{1}{p_1'}}vdw.
\end{align*}
Consequently,  \eqref{maximal-u} with $\widetilde{p}=p_1$ and \eqref{CZ:sum} imply
\begin{multline*}
\|\nabla\widetilde{b}\|_{L^{p_1}(vdw)}^{p_1}
\lesssim 
\alpha^{p_1}\sup_{\|u\|_{L^{p_1'}(vdw)}=1}\left(\sum_{i}\int_{B_i}\left(\mathcal{M}^{vw}(|u|^{p_1'})\right)^{\frac{1}{p_1'}}vdw\right)^{p_1}
\\
\lesssim 
\alpha^{p_1} vw\Big(\bigcup_{i}B_i\Big)
\lesssim \alpha^{p_1-p}\int_{\R^n}|\nabla f|^{p}vdw.
\end{multline*}
\end{proof}
To formulate  our next result (proceeding similarly as in
 \cite{HMay09,PA17}), we introduce the following  conical square function
$$
\Scal_{1/2,\hh}^wf(x):=\left(\int_{B(x,t)}\int_0^{\infty}\big|t\sqrt{L_w}e^{-t^2L_w}f(y)\big|^2\frac{dw(y)\, dt}{tw(B(y,t))}\right)^{\frac{1}{2}}.
$$
Observe that $\widetilde\Scal_{\hh}^{w}f=\Scal_{1/2,\hh}^w\sqrt{L_w}f$. Our goal is to see that $\Scal_{1/2,\hh}^wf$ compares with $\Scal_{\hh}^wf$ (defined in \eqref{eq:conicalheat}) in some weighted spaces (see \cite[Proposition 4.5]{PAII18} for a general version of this result).
For the following statement we recall that $p_+(L_w)_w^{k,*}$ was defined in \eqref{p_w^*}. 

\begin{proposition}\label{prop:widetildeS-heatS}
Given $w\in A_2(dx)$, $v\in A_{\infty}(w)$, and $f\in L^2(w)$, there hold
\begin{list}{$(\theenumi)$}{\usecounter{enumi}\leftmargin=1cm \labelwidth=1cm\itemsep=0.2cm\topsep=.2cm \renewcommand{\theenumi}{\alph{enumi}}}

\item $\|\Scal_{\hh}^wf\|_{L^p(vdw)}\lesssim
\|\Scal_{1/2,\hh}^wf\|_{L^p(vdw)}$,\, for all\,  $p\in \mathcal{W}_v^w(0,p_+(L_w)^{2,*}_w)$;

\item $\|\Scal_{1/2,\hh}^wf\|_{L^p(vdw)}\lesssim\|\Scal_{\hh}^wf\|_{L^p(vdw)}$,\, for all\, $p\in \mathcal{W}_v^w(0,p_+(L_w)^{*}_w)$.
\end{list}

{  In particular, if $p\in \mathcal{W}_v^w(0,p_+(L_w)^{*}_w)$, we have
\begin{align*}
\|\Scal_{1/2,\hh}^wf\|_{L^p(vdw)}\approx\|\Scal_{\hh}^wf\|_{L^p(vdw)}.
\end{align*}}
\end{proposition}
\begin{proof}
We shall use extrapolation to prove both inequalities. Indeed, \cite[Theorem A.1, ($b$)]{ChMPA16} (or \cite[Theorem A.1, ($c$)]{ChMPA16}, if $p_+(L_w)^{*}_w=\infty$) allows us to obtain  $(a)$ from
\begin{align}\label{extrapol-1}
\|\Scal_{\hh}^wf\|_{L^2(v_0dw)}^2\lesssim\|{\Scal^w_{1/2,\hh}}f\|_{L^2(v_0dw)}^2,\quad \forall v_0\in  RH_{\left(\frac{p_+(L_w)^{2,*}_w}{2}\right)'}(w)
\end{align}
and $(b)$ from 
\begin{align}\label{extrapol2}
\|\Scal_{1/2,\hh}^wf\|_{L^2(v_0dw)}^2\lesssim\|\Scal_{\hh}^wf\|_{L^2(v_0dw)}^2, \quad \forall v_0\in  RH_{\left(\frac{p_+(L_w)^{*}_w}{2}\right)'}(w).
\end{align}

 To set the stage, fix $w\in A_2(dx)$ and  $v_0\in RH_{\left(\frac{p_+(L_w)^{k,*}_w}{2}\right)'}(w)$. 
 Here and below $k$ is either $1$ or $2$, depending whether we are proving \eqref{extrapol2} or \eqref{extrapol-1} respectively.

  Note that we can find $\widehat{r}$, $q_0$, $r$, and $M\in \N$ such that  $r_w<\widehat{r}<2$, $2<q_0<p_+(L_w)$, $q_0/2\leq r<\infty$, $v_0\in RH_{r'}(w)$, and
\begin{align}\label{positive-III-1}
k+\frac{n\,\widehat{r}}{2r}-\frac{n\,\widehat{r}}{q_0}-\frac{1}{M}
> 0.
\end{align}

Indeed, if $n\,r_w>kp_+(L_w)$, note that we can take $r_w<\widehat{r}<2$ close enough to $r_w$, $\varepsilon_0>0$ small enough, and $2<q_0<p_+(L_w)$, close enough to $p_+(L_w)$
so that for $r:=\frac{q_0n\widehat{r}}{2(1+\varepsilon_0)(n\widehat{r}-kq_0)}$ we have that
 $q_0/2\leq r<\infty$, $v_0\in RH_{r'}(w)$, and
\begin{align*}
k+\frac{n\widehat{r}}{2r}-\frac{n\widehat{r}}{q_0}=\frac{\varepsilon_0}{q_0}\left(n\widehat{r}-kq_0\right)>\frac{\varepsilon_0}{q_0}\left(nr_w-kp_+(L_w)\right)>0.
\end{align*}
Then, taking $M\in \N$ large enough we obtain \eqref{positive-III-1}.

If now  $n\,r_w\leq kp_+(L_w)$, our condition on the weight $v_0$ becomes $v_0\in A_{\infty}(w)$. Then, we take $r>\mathfrak{s}_{v_0}(w)$ and $q_0$ satisfying
$\max\left\{2,\frac{2rp_+(L_w)}{p_+(L_w)+2r}\right\}<q_0<\min\left\{p_+(L_w),2r\right\}$ if $p_+(L_w)<\infty$, or $q_0=2r$ if $p_+(L_w)=\infty$. Therefore, we have that $2<q_0<p_+(L_w)$, $q_0/2\leq r<\infty$, $v_0\in RH_{r'}(w)$, and
\begin{align*}
k+\frac{n\,r_w}{2r}-\frac{n\,r_w}{q_0}
>
k-\frac{n\,r_w}{p_+(L_w)}\geq 0.
\end{align*}
Taking further $r_w<\widehat{r}<2$ close enough to $r_w$ and $M\in \N$ large enough, we obtain that
\begin{align*}
k+\frac{n\,\widehat{r}}{2r}-\frac{n\,\widehat{r}}{q_0}-\frac{1}{M}
>
0.
\end{align*}
After this observation we show the desired estimate.

\medskip

We first prove \eqref{extrapol-1}.  
Use \eqref{representationsquarerootofL}, Minkowski's integral inequality, and also \eqref{doublingcondition} after noticing that $B(x,t)\subset B(y,2t)$, for all $y\in B(x,t)$. Thus,
\begin{align*}
&\Scal_{\hh}^wf(x)
\\
&\lesssim
\left(\int_0^{\infty}\left(\int_{0}^{t}\left(\int_{B(x,t)}|sL_we^{-s^2L_w}t^2\sqrt{L_w}e^{-t^2L_w}f(y)|^2dw(y)\right)^{\frac{1}{2}}\frac{ds}{s}\right)^2\frac{dt}{tw(B(x,t))}\right)^{\!\!\frac{1}{2}}
\\&
\,+
\left(\int_0^{\infty}\left(\int_{t}^{\infty}\left(\int_{B(x,t)}\left|sL_we^{-s^2L_w}t^2\sqrt{L_w}e^{-t^2L_w}f(y)\right|^2dw(y)\right)^{\frac{1}{2}}\frac{ds}{s}\right)^2\frac{dt}{tw(B(x,t))}\right)^{\!\!\frac{1}{2}}
\\
&
=:I+II.
\end{align*}

For $I$, consider $F(y,t):=t\sqrt{L_w}e^{-\frac{t^2}{2}L_w}f(y)$. Using the fact that $\tau L_we^{-\tau L_w}\in \mathcal{F}(L^2(w)-L^2(w))$, we get
\begin{align*}
&I
\\
&
\leq \!\!
\Bigg(\int_0^{\infty}\!\!\Bigg(\int_{0}^{t}\!\frac{st}{\frac{t^2}{2}+s^2}\!\Bigg(\mathop{\int}\limits_{B(x,t)}\!\Big|\Big(s^2+\frac{t^2}{2}\Big)L_we^{-\left(s^2+\frac{t^2}{2}\right)L_w}F(y,t)\Big|^{\!2}\!\!dw(y)\!\Bigg)^{\!\! \!\!\frac{1}{2}}\!\frac{ds}{s}\!\!\Bigg)^{\!\!\!\!2}\!\frac{dt}{tw(B(x,t))}\!\!\Bigg)^{\!\!\frac{1}{2}}
\\ 
&
\lesssim
\sum_{j\geq 1}e^{-c4^{j}}\left(
\int_0^{\infty}\left(\int_{0}^{t}\frac{s}{t}\frac{ds}{s}\right)^{2}\int_{B(x,2^{j+1}t)}|F(y,t)|^2\frac{dw(y)\,dt}{tw(B(x,t))}\right)^{\frac{1}{2}}
\\
&
\lesssim
\sum_{j\geq 1}e^{-c4^{j}}\left(
\int_0^{\infty}\int_{B(x,2^{j+2}t)}|F(y,\sqrt{2}t)|^2\frac{dw(y)\,dt}{tw(B(y,t))}\right)^{\frac{1}{2}},
\end{align*}
where in the last inequality we have changed the variable $t$ into $\sqrt{2}t$ and used \eqref{doublingcondition}.
Then, applying change of angles (Proposition \ref{prop:alpha}), we conclude that
\begin{align*}
\|I\|_{L^2(v_0dw)}
\lesssim
\sum_{j\geq 1}e^{-c4^j}2^{j\theta_{v_0,w}}
\|\Scal_{1/2,\hh}^wf\|_{L^2(v_0dw)}
\lesssim\|\Scal_{1/2,\hh}^wf\|_{L^2(v_0dw)}.
\end{align*}

For the estimate of $II$, consider $\widetilde{F}(y,s):=(s\sqrt{L_w})^3e^{-s^2L_w}f(y)$.  We apply Cauchy-Schwartz's inequality in the integral in $s$, the fact that $e^{-\tau L_w}\in \mathcal{F}(L^2(w)- L^2(w))$,     Jensen's inequality in the integral in $y$, Fubini's theorem, and \eqref{doublingcondition}.  
Hence,
\begin{align*}
II
&
\lesssim
\left(\int_0^{\infty}\left(\int_{t}^{\infty}\left(\frac{t}{s}\right)^{\frac{1}{M}}
\left(\frac{t}{s}\right)^{2-\frac{1}{M}}\left(\dashint_{B(x,t)}\big|e^{-t^2L_w}\widetilde{F}(y,s)\big|^2dw(y)\right)^{\frac{1}{2}}\frac{ds}{s}\right)^2\frac{dt}{t}\right)^{\frac{1}{2}}
\\
&
\lesssim
\left(\int_0^{\infty}\left(\int_{t}^{\infty}\left(\frac{t}{s}\right)^{\frac{2}{M}}\frac{ds}{s}\right)\left(\int_{t}^{\infty}
\left(\frac{t}{s}\right)^{4-\frac{2}{M}}\!\dashint_{B(x,t)}\big|e^{-t^2L_w}\widetilde{F}(y,s)\big|^2dw(y)\frac{ds}{s}\right)\!\frac{dt}{t}\right)^{\!\!\frac{1}{2}}
\\
&
\lesssim
\sum_{j\geq 1}e^{-c4^j}
\left(\int_0^{\infty}\int_{t}^{\infty}\left(\frac{t}{s}\right)^{4-\frac{2}{M}}\dashint_{B(x,2^{j+1}t)}\big|\widetilde{F}(y,s)\big|^2dw(y)\frac{ds}{s}\frac{dt}{t}\right)^{\frac{1}{2}}
\\&
\lesssim
\sum_{j\geq 1}e^{-c4^j}\left(\int_0^{\infty}\int_{t}^{\infty}\left(\frac{t}{s}\right)^{4-\frac{2}{M}}\left(\dashint_{B(x,2^{j+1}t)}\big|\widetilde{F}(y,s)\big|^{q_0}dw(y)\right)^{\frac{2}{q_0}}\frac{ds}{s}\frac{dt}{t}\right)^{\frac{1}{2}}
\\&
\lesssim
\sum_{j\geq 1}e^{-c4^j}\Bigg(\int_0^{\infty}\!\!\!\int_{0}^{s}\left(\frac{t}{s}\right)^{4-\frac{2}{M}}\Bigg(\mathop{\int}\limits_{B(x,2^{j+1}st/s)}\!\!\big|\widetilde{F}(y,s)\big|^{q_0}\frac{dw(y)}{w(B(y,2^{j+1}st/s))}\Bigg)^{\!\!\frac{2}{q_0}}\!\frac{dt}{t}\frac{ds}{s}\Bigg)^{\!\!\frac{1}{2}}\!\!.
\end{align*}
Then, for  $t<s$, applying Propositions \ref{prop:Q} and \eqref{doublingcondition},
\begin{multline*}
\int_{\R^n}\left(\int_{B(x,2^{j+1}st/s)}\big|\widetilde{F}(y,s)\big|^{q_0}\frac{dw(y)}{w(B(y,2^{j+1}st/s))}\right)^{\frac{2}{q_0}}v_0(x)dw(x)
\\
\lesssim \left(\frac{t}{s}\right)^{n\widehat{r}\left(\frac{1}{r}-\frac{2}{q_0}\right)}
\int_{\R^n}\left(\dashint_{B(x,2^{j+1}s)}\big|\widetilde{F}(y,s)\big|^{q_0}dw(y)\right)^{\frac{2}{q_0}}v_0(x)dw(x).
\end{multline*}
Denote $\widehat{F}(y,s):=s\sqrt{L_w}e^{-\frac{s^2}{2}L_w}$. Since  $\tau L_we^{-\tau L_w}\in \Ocal(L^{2}(w)-L^{q_0}(w))$, then we have
\begin{multline*}
\left(\dashint_{B(x,2^{j+1}t)}\big|\widetilde{F}(y,s)\big|^{q_0}dw(y)\right)^{\frac{2}{q_0}}
\lesssim
\left(\dashint_{B(x,2^{j+1}s)}\left|\frac{s^2}{2}L_we^{-\frac{s^2}{2}L_w}\widehat{F}(y,s)\right|^{q_0}dw(y)\right)^{\frac{2}{q_0}}
\\
\lesssim
\sum_{l\geq 1}e^{-c4^l}2^{j\theta_2}
\dashint_{B(x,2^{l+j+2}s)}\big|\widehat{F}(y,s)\big|^{2}dw(y).
\end{multline*}
Consequently, applying Fubini's theorem, \eqref{positive-III-1} with $k=2$, changing the variable $s$ into $\sqrt{2}s$, and by \eqref{doublingcondition} and change of angles (Proposition \ref{prop:alpha}), we get, for  $\widetilde{C}:=4+n\widehat{r}\left(\frac{1}{r}-\frac{2}{q_0}\right)-\frac{2}{M}$,
\begin{align*}
&\|II\|_{L^2(v_0dw)}
\\
&
\lesssim
\sum_{j\geq 1}e^{-c4^j}\!\sum_{l\geq 1}e^{-c4^l}2^{j\theta_2}\!\Bigg(\!\int_{\R^n}\!\int_0^{\infty}\!\!\int_{0}^{s}\!\left(\frac{t}{s}\right)^{\!\!\widetilde{C}}\!\frac{dt}{t}
\mathop{\dashint}_{\!B(x,2^{l+j+2}s)}\!\big|\widehat{F}(y,s)\big|^{\!2}dw(y)\frac{ds}{s}v_0(x)dw(x)\!\!\Bigg)^{\!\!\!\frac{1}{2}}
\\&
\lesssim
\sum_{j\geq 1}e^{-c4^j}
\sum_{l\geq 1}e^{-c4^l}\left(\int_{\R^n}\int_0^{\infty}
\int_{B(x,2^{l+j+3}s)}\big|\widehat{F}(y,\sqrt{2}s)\big|^{2}\frac{dw(y)ds}{sw(B(y,s))}v_0(x)dw(x)\right)^{\frac{1}{2}}
\\&
\lesssim
\sum_{j\geq 1}e^{-c4^j}\sum_{l\geq 1}e^{-c4^l}2^{(l+j)\theta_{v_0,w}}\|\Scal_{1/2,\hh}^wf\|_{L^2(v_0dw)}
\\&
\lesssim
\|\Scal_{1/2,\hh}^wf\|_{L^2(v_0dw)}.
\end{align*}

\medskip

As for proving \eqref{extrapol2}, using again \eqref{doublingcondition}, \eqref{representationsquarerootofL}, and Minkowski's integral inequality, we obtain
\begin{align*}
&\Scal_{1/2,\hh}^wf(x)
\\&\,\lesssim
\left(\int_0^{\infty}\left(\int_0^{t}\left(\int_{B(x,t)}|tsL_we^{-s^2L_w}e^{-t^2L_w}f(y)|^2dw(y)\right)^{\frac{1}{2}}\frac{ds}{s}\right)^{{2}}\frac{dt}{tw(B(x,t))}\right)^{\frac{1}{2}}
\\
&\quad+
\left(\int_0^{\infty}\left(\int_{t}^{\infty}\left(\int_{B(x,t)}|tsL_we^{-s^2L_w}e^{-t^2L_w}f(y)|^2dw(y)\right)^{\frac{1}{2}}\frac{ds}{s}\right)^{{2}}\frac{dt}{tw(B(x,t))}\right)^{\frac{1}{2}}
\\&\,=:\widetilde{I}+\widetilde{II}.
\end{align*}

We first estimate $\widetilde{I}$. Using that $s<t$ and applying the fact that $e^{-\tau L_w}\in \mathcal{F}(L^2(w)-L^2(w))$, and \eqref{doublingcondition}, we have
\begin{align*}
\widetilde{I} &=
\left(\int_0^{\infty}\left(\int_0^{t}\frac{s}{t}\left(\int_{B(x,t)}|e^{-s^2L_w}t^2L_we^{-t^2L_w}f(y)|^2dw(y)\right)^{\frac{1}{2}}\frac{ds}{s}\right)^{{2}}\frac{dt}{tw(B(x,t))}\right)^{\frac{1}{2}}
\\
&\lesssim
\sum_{j\geq 1}e^{-c4^j}\left(\int_0^{\infty}\left(\int_0^{t}\frac{s}{t}\frac{ds}{s}\right)^2\int_{B(x,2^{j+1}t)}|t^2L_we^{-t^2L_w}f(y)|^2\frac{dw(y)\,dt}{tw(B(y,t))}\right)^{\frac{1}{2}}
\\
&\lesssim
\sum_{j\geq 1}e^{-c4^j}
\left(\int_0^{\infty}\int_{B(x,2^{j+1}t)}|t^2L_we^{-t^2L_w}f(y)|^2\frac{dw(y)\,dt}{tw(B(y,t))}\right)^{\frac{1}{2}}.
\end{align*}
Therefore, applying change of angles (Proposition \ref{prop:alpha}), we get
$$
\|\widetilde{I}\|_{L^2(v_0dw)}
\lesssim
\sum_{j\geq 1}e^{-c4^j}2^{j\theta_{v_0,w}}\|\Scal_{\hh}^wf\|_{L^2(v_0dw)}\lesssim \|\Scal_{\hh}^wf\|_{L^2(v_0dw)}.
$$

The estimate of $\widetilde{II}$ is very similar to that of $II$ (in the proof of \eqref{extrapol-1}), so we skip some details. We  apply again the fact that $e^{-\tau L_w}\in \mathcal{F}(L^2(w)-L^2(w))$, Cauchy-Schwartz's inequality in the integral in $s$,   Jensen's inequality in the integral in $y$,   Fubini's theorem, and
\eqref{doublingcondition}. Hence, we have
\begin{align*}
&\!\widetilde{II}
\\
&\lesssim
\sum_{j\geq 1}e^{-c4^j}\!
\Bigg(\int_0^{\infty}\!\Bigg(\int_{t}^{\infty}\!\!\left(\frac{t}{s}\right)^{\frac{1}{M}}\!\!\left(\frac{t}{s}\right)^{1-\frac{1}{M}}\!\Bigg(\!
\mathop{\dashint}_{B(x,2^{j+1}t)}\!\!\!\!\!|s^2L_we^{-s^2L_w}f(y)|^{2}dw(y)\Bigg)^{\!\!\!\frac{1}{2}}\!\frac{ds}{s}\Bigg)^{\!\!2}\frac{dt}{t}\Bigg)^{\!\!\!\frac{1}{2}}
\\
&\lesssim
\sum_{j\geq 1}e^{-c4^j}
\left(\int_0^{\infty}\int_{t}^{\infty}\left(\frac{t}{s}\right)^{2-\frac{2}{M}}\left(
\dashint_{B(x,2^{j+1}t)}|s^2L_we^{-s^2L_w}f(y)|^{q_0}dw(y)\right)^{\frac{2}{q_0}}\frac{ds}{s}\frac{dt}{t}\right)^{\frac{1}{2}}
\\
&\lesssim
\sum_{j\geq 1}e^{-c4^j}\!\!
\Bigg(\!\!\int_0^{\infty}\!\!\int_{0}^{s}\!\!\left(\frac{t}{s}\right)^{2-\frac{2}{M}}\!\!\Bigg(
\mathop{\int}_{B(x,2^{j+1}st/s)}\!\!\!\!\!\!|s^2L_we^{-s^2L_w}f(y)|^{q_0}\frac{dw(y)}{w(B(y,2^{j+1}st/s))}\!\Bigg)^{\!\!\!\frac{2}{q_0}}\!\frac{dt}{t}\!\frac{ds}{s}\!\Bigg)^{\!\!\!\frac{1}{2}}\!\!.
\end{align*}
Note that, for $t<s$,   Proposition \ref{prop:Q}, and  \eqref{doublingcondition} imply
\begin{multline*}
\int_{\R^n}
\left(\int_{B(x,2^{j+1}st/s)}|s^2L_we^{-s^2L_w}f(y)|^{q_0}\frac{dw(y)}{B(y,2^{j+1}st/s)}\right)^{\frac{2}{q_0}}v_0(x)dw(x)
\\
\lesssim
\left(\frac{t}{s}\right)^{n\,\widehat{r}\left(\frac{1}{r}-\frac{2}{q_0}\right)}\int_{\R^n}
\left(\dashint_{B(x,2^{j+1}s)}|s^2L_we^{-s^2L_w}f(y)|^{q_0}dw(y)\right)^{\frac{2}{q_0}}v_0(x)dw(x).
\end{multline*}
Besides, since $e^{-\tau L_w}\in \mathcal{O}(L^{2}(w)\rightarrow L^{q_0}(w))$
\begin{align*}
\Bigg(\mathop{\dashint}_{B(x,2^{j+1}s)}\!|s^2L_we^{-s^2L_w}f(y)|^{q_0}dw(y)\Bigg)^{\!\!\!\frac{2}{q_0}}
\!\!\!\lesssim\!
\sum_{l\geq 1}e^{-c4^l}2^{j\theta_2}\!\!\!\!\!\!\!\mathop{\dashint}_{B(x,2^{j+l+2}s)}\!\!\!|s^2L_we^{-\frac{s^2}{2}L_w}f(y)|^{2}dw(y).
\end{align*}
Hence, applying Fubini's theorem, \eqref{positive-III-1} with $k=1$, changing the variable $s$ into $\sqrt{2}s$, by \eqref{doublingcondition} and  Proposition \ref{prop:alpha} we have
\begin{align*}
&\|\widetilde{II}\|_{L^2(v_0dw)}
\\&
\lesssim
\sum_{j\geq 1}e^{-c4^j} \sum_{l\geq 1}e^{-c4^l}
\left(\int_{\R^n}\int_0^{\infty}\!\!\!\int_{B(x,2^{j+l+3}s)}|s^2L_we^{-s^2L_w}f(y)|^2\frac{
dw(y)\,ds}{sw(B(y,s))}v_0(x)dw(x)\!\right)^{\!\!\!\frac{1}{2}}
\\&
\lesssim
\sum_{j\geq 1}e^{-c4^j} \sum_{l\geq 1}e^{-c4^l}2^{(l+j)\theta_{w,v_0}}\|\Scal_{\hh}^wf\|_{L^2(v_0dw)}
\lesssim\|\Scal_{\hh}^wf\|_{L^2(v_0dw)}.
\end{align*} 
This and the estimate obtained for $\|\widetilde{I}\|_{L^2(v_0dw)}$ give us \eqref{extrapol2}.
\end{proof}

\begin{proof}[Proof of Theorem \ref{thm:boundednesswidetildeS}]
First of all fix $w\in A_2(dx)$, and denote $q_-:=p_-(L_w)=q_-(L_w)$ (see Lemma \ref{lem:ODweighted}), $q_+:=q_+(L_w)$, and $p_+:=p_+(L_w)$.

We claim that for all $p\in \mathcal{W}_v^w\left(\max\{r_w,q_-\},p_+\right)$
and $h\in \Scal $,
 \begin{align}\label{Sobolev}
 \|\widetilde{\Scal}_{\hh}^wh\|_{L^p(vdw)}\lesssim \|\nabla h\|_{L^p(vdw)}.
 \end{align}
Indeed,  applying Proposition \ref{prop:widetildeS-heatS},  \cite[Theorem 3.1]{ChMPA16}, and Proposition \ref{prop:wRR}, we have that
\begin{multline*}
\|\widetilde{\Scal}_{\hh}^wh\|_{L^p(vdw)} 
= \|\Scal_{1/2,\hh}^w\sqrt{L_w}h\|_{L^p(vdw)}
\approx \|\Scal_{\hh}^w\sqrt{L_w}h\|_{L^p(vdw)}
\\
\lesssim \|\sqrt{L_w}h\|_{L^p(vdw)}\lesssim\|\nabla h\|_{L^p(vdw)}.
\end{multline*}

Note that $\mathcal{W}_v^w\left(\max\{r_w,q_-\},p_+\right)=\big(\mathfrak{r}_v(w)\max\{r_w,q_-\},p_+/\mathfrak{s}_v(w)\big)$. Therefore,  for every $p$ satisfying 
\begin{align}\label{choiceofp2}
\mathfrak{r}_v(w)\max\left\{r_w,(q_-)_{w,*}\right\}<p<\mathfrak{r}_v(w)\max\{r_w,q_-\},
\end{align}
if we show that
 \begin{align}\label{weakinterpolation}
 \|\widetilde{\Scal}_{\hh}^wh\|_{L^{p,\infty}(vdw)}\lesssim \|\nabla h\|_{L^p(vdw)}, \quad \forall h\in \Scal,
 \end{align}
then, by interpolation (see \cite{Ba09} and Remark \ref{remark:product-weight}) we would conclude \eqref{boundednesswidetilde}.

Now fix $p$ as in \eqref{choiceofp2}, and note that $vw\in A_p(dx)$, since $r_{vw}\leq r_w\mathfrak{r}_v(w)<p$ (see Remark \ref{remark:product-weight}). 
Given $\alpha>0$, we apply Lemma \ref{lem:CZweighted} to $h\in \Scal$, $\alpha$, the product weight $\varpi=vw$ and $p$.
Let $\{B_i\}_{i}$ be the collection of balls given by Lemma \ref{lem:CZweighted}. Consider for $M\in \N$ arbitrarily large, 
\[
B_{r_{B_i}}:=(I-e^{-r_{B_i}^2L_w})^M,\qquad A_{r_{B_i}}:=I-B_{r_{B_i}}=\sum_{k=1}^MC_{k,M}e^{-kr_{B_i}^2L_w}.
\]
Then
$$
h= g+\sum_{i}A_{r_{B_i}}b_i+\sum_{i}B_{r_{B_i}}b_i=:g+\widetilde{b}+\widehat{b}.
$$ 
It follows that
\begin{multline}\label{spplitingriesz}
vw\left(\left\{x\in \R^n: \widetilde{\Scal}_{\hh}^w h(x)>\alpha\right\}\right)
\leq
vw\left(\left\{x\in \R^n: \widetilde{\Scal}_{\hh}^w g(x)>\frac{\alpha}{3}\right\}\right)
\\
+
vw\left(\left\{x\in \R^n: \widetilde{\Scal}_{\hh}^w \widetilde{b}(x)>\frac{\alpha}{3}\right\}\right)
\\
+
vw\left(\left\{x\in \R^n: \widetilde{\Scal}_{\hh}^w \widehat{b}(x)>\frac{\alpha}{3}\right\}\right)
=:I+II+III.
\end{multline}

Now, since $\mathcal{W}_v^w(\max\{r_w,q_-\},q_+)\!\neq \!\emptyset$ by assumption and $p\!>\!\mathfrak{r}_v(w)\max\{r_w,(q_-)_{w,*}\}$ (see \eqref{choiceofp2}), by Remark \ref{remark:intervalnotempty} 
we can pick
$p_1$ such that 
\begin{align}\label{choicep_1}
\mathfrak{r}_v(w)\max\{r_w,q_-\}<p_1<\min\left\{\frac{q_+}{\mathfrak{s}_v(w)},p_{vw}^*\right\}.
\end{align}
Observe that 
\[
p_1>p, \quad \text{and} \quad p_1\in \mathcal{W}_v^w\left(\max\{r_w,q_-\},q_+\right).
\]
Note also that in particular $\mathfrak{r}_v(w)\,q_-<p_1<\frac{q_+}{\mathfrak{s}_v(w)}$, that is,
\begin{align}\label{choicep_1-2}
v\in A_{\frac{p_1}{q_-}}(w)\cap RH_{\left(\frac{q_+}{p_1}\right)'}(w).
\end{align}

Now we are ready to estimate $I$.
Applying Chebyshev's inequality with $p_1$, \eqref{Sobolev}, and properties \eqref{CZ:g}-\eqref{CZ:overlap}, we obtain
\begin{equation}\label{Itermriesz}
I\lesssim \frac{1}{\alpha^{p_1}}\int_{\R^n}|\widetilde{\Scal}_{\hh}^w g|^{p_1}vdw
\lesssim \frac{1}{\alpha^{p_1}}
\int_{\R^n}|\nabla g|^{p_1}vdw
\lesssim
\frac{1}{\alpha^{p}}
\int_{\R^n}|\nabla h|^{p}vdw.
\end{equation}

\medskip

In order to estimate $II$, apply Chebyshev's inequality, \eqref{Sobolev}, and Lemma \ref{lemma:arbi} (with $f=h$). Then
\begin{align}\label{IItermrieszII}
II \lesssim\frac{1}{\alpha^{p_1}}\int_{\R^n} \left|\widetilde{\Scal}_{\hh}^w\widetilde{b}\right|^{p_1}vdw
\lesssim\frac{1}{\alpha^{p_1}}\int_{\R^n} \big|\nabla \widetilde{b}\big|^{p_1}vdw
\lesssim
\frac{1}{\alpha^{p}}\int_{\R^n}|\nabla h|^{p}vdw.
\end{align}

Next, we estimate $III$.  Note that, by \eqref{CZ:sum}
\begin{align}\label{III-Riesz}
III &\lesssim 
vw\Big(\bigcup_{i}16B_i\Big)
+
vw\Big(\Big\{x\in \R^n\setminus \bigcup_{i}16B_i: \widetilde{\Scal}_{\hh}^w \widehat{b}(x)>\frac{\alpha}{3}\Big\}\Big)
\\\nonumber
&\lesssim
\frac{1}{\alpha^{p}}\int_{\R^n}|\nabla h|^{p}vdw+
III_1,
\end{align}
where 
$$
III_1:=
vw\Big\{\Big(x\in \R^n\setminus \bigcup_{i}16B_i: \widetilde{\Scal}_{\hh}^w \widehat{b}(x)>\frac{\alpha}{3}\Big\}\Big).
$$
By Chebyshev's inequality, duality, splitting the integral in $x$,  and applying H\"older's inequality:
\begin{align}\label{III2-Riesz}
 III_1&\lesssim
 \frac{1}{\alpha^{p_1}}\int_{\R^n\setminus 
 \cup_{i}16B_i}\big|\widetilde{\Scal}_{\hh}^w\widehat{b}\big|^{p_1}vdw
 \\ \nonumber
 &\lesssim
 \frac{1}{\alpha^{p_1}}
\Bigg(\sup_{\|u\|_{L^{p_1'}(vdw)}=1} \sum_{i}\sum_{j\geq 4}\Bigg(\int_{C_j(B_i)}\abs{\widetilde{\Scal}_{\hh}^w \br{B_{r_{B_i}}b_i}}^{p_1}vdw\Bigg)^{\!\!\!\frac{1}{p_1}}\!\|u\chi_{C_j(B_i)}\|_{L^{p_1'}(vdw)}\Bigg)^{p_1}
\\ \nonumber
&=:\frac{1}{\alpha^{p_1}}\sup_{\|u\|_{L^{p_1'}(vdw)}=1}\left(\sum_{i}\sum_{j\geq 4}\mathcal{III}_{ij}\,\|u\chi_{C_j(B_i)}\|_{L^{p_1'}(vdw)}\right)^{p_1}.
\end{align}
Splitting the integral in $t$ (recall that $j\geq 4$), we have
\begin{align}\label{defIIIij2}
&\mathcal{III}_{ij}
\lesssim\!\!
\Bigg(\!\int_{C_j(B_i)}\!\!\Bigg(\!\int_0^{2^{j-2}r_{B_i}}\!\!\int_{B(x,t)}\!\left|tL_we^{-t^2L_w}\!\left(B_{r_{B_i}}b_i\right)\!(y)\right|^{\!2}\!\!\frac{dw(y)\,dt}{tw(B(y,t))}\!\Bigg)^{\!\!\frac{p_1}{2}}\!\!v(x)dw(x)\!\Bigg)^{\!\!\frac{1}{p_1}}
\\\nonumber
&\,
+\!\!
\Bigg(\!\int_{C_j(B_i)}\!\!\Bigg(\!\int_{2^{j-2}r_{B_i}}^{\infty}\!\!\int_{B(x,t)}\!\Big|t^2L_we^{-t^2L_w}\!\Big(B_{r_{B_i}}\!\Big(\frac{b_i}{r_{B_i}}\Big)\!\Big)\!(y)\Big|^{2}\frac{dw(y)\,dt}{tw(B(y,t))}\!\Bigg)^{\!\!\frac{p_1}{2}}\!\!v(x)dw(x)\!\Bigg)^{\!\!\frac{1}{p_1}}
\\\nonumber
&
=:
\mathcal{III}_{ij}^1+\mathcal{III}_{ij}^2.
\end{align}

Before estimating $\mathcal{III}_{ij}^1$ and $\mathcal{III}_{ij}^2$, we take $p_0$ close enough to $q_-$, and $q_0$  close enough to $q_+$ so that 
\begin{equation} \label{p0q0choice}
q_-<p_0<\min\{2,p_1\},\quad \max\{2,p_1\}<q_0<q_+,
\quad
v\in A_{\frac{p_1}{p_0}}(w) \cap RH_{\left(\frac{q_0}{p_1}\right)'}(w).
\end{equation}

Hence, by  Lemma \ref{ARHsinpesoconpeso} (b),
\begin{multline}\label{plug2}
\mathcal{III}_{ij}^1 
\\
\!
\lesssim \!
vw(2^{j+1}B_i)^{\!\frac{1}{p_1}}\!\!
\Bigg(\!\dashint_{C_j(B_i)}\!\!\!\Bigg(\!\int_0^{2^{j-2}r_{B_i}}\!\!\!\!\int_{B(x,t)}\!\left|tL_we^{-t^2L_w}\!\left(B_{r_{B_i}}b_i\right)\!\!(y)\right|^{2}\!\!\!\frac{dw(y)\,dt}{tw(B(y,t))}\Bigg)^{\!\!\!\!\frac{q_0}{2}}\!\!\!dw(x)\!\Bigg)^{\!\!\!\frac{1}{q_0}}\!\!\!.
\end{multline}

Besides, note that for $x\in C_j(B_i)$ and $0<t\leq 2^{j-2}r_{B_i}$ we have that $B(x,t)\subset 2^{j+2}B_i\setminus 2^{j-1}B_i$. Then, by \eqref{doublingcondition}, recalling that $q_0>2$, applying Jensen's  inequality with respect to $dw(y)\,dt$, and Fubini's theorem, we get
\begin{align}\label{plug1}
&\Bigg(\!\dashint_{C_j(B_i)}\!\!\Bigg(\int_0^{2^{j-2}r_{B_i}}\!\!\!\!\int_{B(x,t)}\!\left|tL_we^{-t^2L_w}\left(B_{r_{B_i}}b_i\right)(y)\right|^{2}\!\frac{dw(y)\,dt}{tw(B(y,t))}\Bigg)^{\frac{q_0}{2}}dw(x)\Bigg)^{\frac{1}{q_0}}
\\\nonumber
&\quad\lesssim(2^{j}r_{B_i})^{\frac{1}{2}}\!\!\Bigg(\!\dashint_{C_j(B_i)}\!\!\Bigg(\dashint_0^{2^{j-2}r_{B_i}}\!\!\!\!\dashint_{B(x,t)}\!\frac{1}{t}\left|tL_we^{-t^2L_w}\left(B_{r_{B_i}}b_i\right)\!(y)\right|^{2}\!dw(y)\,dt\Bigg)^{\!\!\frac{q_0}{2}}\!\!dw(x)\!\Bigg)^{\!\!\frac{1}{q_0}}
\\\nonumber
&\quad
\lesssim 
\left(\!\dashint_{C_j(B_i)}\!\int_0^{2^{j-2}r_{B_i}}\!\!\left(\frac{2^{j}r_{B_i}}{t}\right)^{\!\!\frac{q_0}{2}-1}\!\!\dashint_{B(x,t)}\!\left|tL_we^{-t^2L_w}\!\!\left(B_{r_{B_i}}b_i\right)\!(y)\right|^{q_0}\!\frac{dw(y)\,dt}{t}dw(x)\right)^{\!\!\frac{1}{q_0}}
\\\nonumber
&\quad
\lesssim 
\left(\!\dashint_{C_j(B_i)}\!\int_0^{2^{j-2}r_{B_i}}\!\!\left(\frac{2^{j}r_{B_i}}{t}\right)^{\!\frac{q_0}{2}-1}\!\!\!\int_{B(x,t)}\!\left|tL_we^{-t^2L_w}\!\left(B_{r_{B_i}}b_i\right)\!(y)\right|^{q_0}\!\!\!\frac{dw(y)\,dt}{tw(B(y,t))}dw(x)\!\right)^{\!\!\!\frac{1}{q_0}}
\\\nonumber
&\quad
\lesssim 
\left(\!\int_0^{2^{j-2}r_{B_i}}\!\!\left(\frac{2^{j}r_{B_i}}{t}\right)^{\frac{q_0}{2}-1}t^{-q_0}\dashint_{2^{j+2}B_i\setminus 2^{j-1}B_i}\!\left|t^2L_we^{-t^2L_w}\left(B_{r_{B_i}}b_i\right)(y)\right|^{q_0}\!\frac{dw(y)\,dt}{t}\right)^{\!\!\frac{1}{q_0}}\!\!.
\end{align}
We estimate the integral in $y$ by using functional calculus.  The notation is taken from  \cite{Au07} and \cite[Section 7]{AMIII06}.
We write $\vartheta\in[0,\pi/2)$ for the supremum of $|{\rm arg}(\langle L_wf,f\rangle_{L^2(w)})|$ over all $f$ in the domain of $L_w$.
Let $0<\vartheta <\theta<\nu<\mu<\pi /2$ and note that, for a fixed $t>0$, $\phi(z,t):=e^{-t^2 z}(1-e^{-r_{B_i}^2 z})^M$ is holomorphic in the open sector $\Sigma_\mu=\{z\in\mathbb{C}\setminus\{0\}:|{\rm arg} (z)|<\mu\}$ and satisfies $|\phi(z,t)|\lesssim |z|^M\,(1+|z|)^{-2M}$ (with implicit constant depending on $\mu$, $t>0$, $r_{B_i}$, and $M$) for every $z\in\Sigma_\mu$. 
Hence, we can write
\begin{equation}\label{afaerverv}
	\phi(L_w,t)=\int_{\Gamma } e^{-zL_w }\eta (z,t)dz,
\qquad \text{where} \quad 
\eta(z,t) = \int_{\gamma} e^{\zeta z} \phi(\zeta,t) d\zeta.
\end{equation}
Here $\Gamma=\partial\Sigma_{\frac\pi2-\theta}$ with positive orientation (although orientation is irrelevant for our computations) and 
$\gamma=\R_+e^{i\,{\rm sign}({\rm Im} (z))\,\nu}$. It is not difficult to see that for every $z\in \Gamma$,
$$
|\eta(z,t)| \lesssim \frac{r_{B_i}^{2M}}{(|z|+t^2)^{M+1}}.
$$
Moreover, observe that $2^{j+2}B_i\setminus 2^{j-1}B_i= \bigcup_{l=1}^{3}C_{l+j-2}(B_i)$, $\forall j\ge 4$. Also, our choices of $p_0$ and $q_0$ in \eqref{p0q0choice} yield that $zL_we^{-zL_w}\in \mathcal O(L^{p_0}(w)- L^{q_0}(w))$. Thus, by these facts and Minkowski's integral inequality, we obtain 
\begin{align*}
&\left(\dashint_{2^{j+2}B_i\setminus 2^{j-1}B_i}\right.\left.\left|t^2L_we^{-t^2L_w}\left(B_{r_{B_i}}b_i\right)\right|^{q_0}dw\right)^{\frac{1}{q_0}}
\\&\qquad
\lesssim
\int_{\Gamma}
\left(\dashint_{2^{j+2}B_i\setminus 2^{j-1}B_i}\left|zL_we^{-zL_w} b_i\right|^{q_0}dw\right)^{\frac{1}{q_0}}\frac{t^2}{|z|}\frac{r_{B_i}^{2M}}{(|z|+t^2)^{M+1}}|dz|
\\&\qquad
\lesssim 2^{j\theta_1} 
\left(\dashint_{B_i}\left|b_i\right|^{p_0}dw\right)^{\frac{1}{p_0}}\int_{\Gamma}\Upsilon\left(\frac{2^jr_{B_i}}{|z|^{\frac{1}{2}}}\right)^{\theta_2}e^{-c\frac{4^jr_{B_i}^2}{|z|}}
\frac{t^2}{|z|}\frac{r_{B_i}^{2M}}{(|z|+t^2)^{M+1}}|dz|
\\&\qquad
\lesssim 2^{j\theta_1} t^2
\left(\dashint_{B_i}\left|b_i\right|^{p_1}d(vw)\right)^{\frac{1}{p_1}}\int_{0}^{\infty}
\Upsilon\left(\frac{2^jr_{B_i}}{s^{\frac{1}{2}}}\right)^{\theta_2}e^{-c\frac{4^jr_{B_i}^2}{s}}
\frac{r_{B_i}^{2M}}{s^{M+1}}\frac{ds}{s}
\\&\qquad
\lesssim \alpha\,  r_{B_i}^{-1}2^{-j(2M+2-\theta_1)}t^2
\int_{0}^{\infty}\Upsilon\left(s\right)^{\theta_2}e^{-cs^2}
s^{2M+2}\frac{ds}{s}
\\&\qquad
\lesssim \alpha\,  r_{B_i}^{-1}\,2^{-j(2M+2-\theta_1)}t^2,
\end{align*}
where, to obtain the last inequality, we need to take $M\in \N$ large enough so that $2M+2>\theta_2$. Besides, we use Lemma 2.16 (a) in the third inequality; and the forth inequality  follows from (3.7) (see (3.36)) and the change of variable $s$ into $4^{j}r_{B_i}^2/s^2$.

Plugging the above estimate into \eqref{plug1} and changing the variable $t$ into $2^j r_{B_i}t$, allows us to obtain
\begin{multline*}
\left(\!\dashint_{C_j(B_i)}\!\!\left(\int_0^{2^{j-2}r_{B_i}}\!\!\!\!\int_{B(x,t)}\!\left|tL_we^{-t^2L_w}\left(B_{r_{B_i}}b_i\right)(y)\right|^{2}\!\frac{dw(y)\,dt}{tw(B(y,t))}\right)^{\frac{q_0}{2}}dw(x)\right)^{\frac{1}{q_0}}
\\
\lesssim
\alpha \, r_{B_i}^{-1}\,2^{-j(2M+2-\theta_1)}
\left(\int_0^{2^{j-2}r_{B_i}}
\left(\frac{2^{j}r_{B_i}}{t}\right)^{\frac{q_0}{2}-1}t^{q_0}
\frac{dt}{t}\right)^{\frac{1}{q_0}}
\lesssim
\alpha\,  2^{-j(2M+1-\theta_1)}.
\end{multline*}
This and \eqref{plug2} yield, for $M\in \N$ such that $2M+2>\theta_2$, 
\begin{align}\label{termcalIII}
\mathcal{III}_{ij}^1
\lesssim \alpha\, vw(2^{j+1}B_i)^{\frac{1}{p_1}}2^{-j(2M+1-\theta_1)}.
\end{align}

In order to estimate $\mathcal{III}_{ij}^2$, we first change the variable $t$ into $t\theta_M:=t\sqrt{M+1}$. Then
\begin{align}\label{plug3}
&
\left(\int_{2^{j-2}r_{B_i}}^{\infty}\int_{B(x,t)}\left|t^2L_we^{-t^2L_w}\left(B_{r_{B_i}}\left(\frac{b_i}{ r_{B_i}}\right)\right)(y)\right|^{2}\frac{dw(y)\,dt}{tw(B(y,t))}\right)^{\frac{1}{2}}
\\\nonumber&
\lesssim
\left(\int_{2^{j-2}r_{B_i}}^{\infty}\dashint_{B(x,t)}\left|t^2L_we^{-t^2L_w}\left(B_{r_{B_i}}\left(\frac{b_i}{ r_{B_i}}\right)\right)(y)\right|^{2}\frac{dw(y)\,dt}{t}\right)^{\frac{1}{2}}
\\\nonumber&
\lesssim
\left(\int_{\frac{2^{j-2}r_{B_i}}{\theta_M}}^{\infty}\dashint_{B(x,\theta_Mt)}\left|\mathcal{T}_{t,r_{B_i}}t^2L_we^{-t^2L_w}\left(\frac{b_i}{ r_{B_i}}\right)(y)\right|^{2}\frac{dw(y)\,dt}{t}\right)^{\frac{1}{2}},
\end{align}
where we recall that $\mathcal{T}_{t,r_{B_i}}:=(e^{-t^2L_w}-e^{-(t^2+r_{B_i}^2)L_w})^M$. 

 {  Next, fix $x\in C_{j}(B_i)$ and $t>\frac{2^{j-2} r_{B_i}}{\theta_M}$}. In this case, $B_i\subset B(x,12\theta_{M}t)$.
Thus,  by \eqref{boundednesstsr} and the fact that $\tau L_we^{-\tau L_w}\in \mathcal{O}(L^{p_0}(w)-L^2(w))$,  for $\tau >0$, we get
\begin{align*}
&\dashint_{B(x,\theta_Mt)}\left|\mathcal{T}_{t,r_{B_i}}t^2L_we^{-t^2L_w}\left(\frac{b_i}{ r_{B_i}}\right)(y)\right|^{2} dw(y),
\\&  \qquad\quad
\lesssim
 \left(\frac{r_{B_i}^2}{t^2}\right)^{2M}
\frac{1}{w(B(x,\theta_Mt))}\int_{\R^n}\left|t^2L_we^{-t^2L_w}\left(\chi_{B(x,12\theta_Mt)}\frac{b_i}{ r_{B_i}}\right)(y)\right|^{2}dw(y)
\\& \qquad\quad
\lesssim 
\left(\frac{r_{B_i}^2}{t^2}\right)^{2M}
\sum_{l\geq 1}2^{2nl}\dashint_{C_l(B(x,12\theta_M t))}\left|t^2L_we^{-t^2L_w}\left(\chi_{B(x,12\theta_Mt)}\frac{b_i}{ r_{B_i}}\right)(y)\right|^{2}dw(y)
\\& \qquad\quad
\lesssim 
\left(\frac{r_{B_i}^2}{t^2}\right)^{2M}
\sum_{l\geq 1}e^{-c4^l}\left(\dashint_{B(x,12\theta_M t)}\left|\frac{b_i(y)}{ r_{B_i}}\right|^{p_0}dw(y)\right)^{\frac{2}{p_0}}
\\& \qquad\quad
\lesssim 
\left(\frac{r_{B_i}^2}{t^2}\right)^{2M}
\left(\frac{w(B_i)}{w(B(x,12\theta_M t))}\right)^{\frac{2}{p_0}}
\left(\dashint_{B_i}\left|\frac{b_i}{ r_{B_i}}\right|^{p_0}dw\right)^{\frac{2}{p_0}}
\\& \qquad\quad
\lesssim 
\left(\frac{r_{B_i}^2}{t^2}\right)^{2M}
\left(\dashint_{B_i}\left|\frac{b_i}{ r_{B_i}}\right|^{p_1}d(vw)\right)^{\frac{2}{p_1}}
\lesssim 
\left(\frac{r_{B_i}^2}{t^2}\right)^{2M}
\alpha^2.
\end{align*}
Here the next-to-last inequality is due to Lemma \ref{ARHsinpesoconpeso} (a) and the fact that 
$B_i\subset B(x,12\theta_M t)$,
and 
the last inequality follows from  \eqref{CZ:PS}.

Plugging the above estimate into \eqref{plug3} and recalling the definition of $\mathcal{III}_{ij}^2$ in \eqref{defIIIij2} allows us to obtain
\begin{align*}
\mathcal{III}_{ij}^2
\lesssim \alpha\Bigg(\int_{C_j(B_i)}
\Bigg(\int_{\frac{2^{j-2}r_{B_i}}{\theta_M}}^{\infty}
\left(\frac{r_{B_i}^2}{t^2}\right)^{2M}\frac{dt}{t}\Bigg)^{\frac{p_1}{2}}vdw\Bigg)^{\frac{1}{p_1}}
\lesssim
\alpha vw(2^{j+1}B_i)^{\frac{1}{p_1}}
2^{-j2M}.
\end{align*}
By this  and \eqref{termcalIII},  for $M\in \N$ such that $2M>\theta_2-2$, we have
\[
\mathcal{III}_{ij}\lesssim \alpha vw(2^{j+1}B_i)^{\frac{1}{p_1}}2^{-j(2M-\theta_1)}.
\] 

Then, by Lemma \ref{lemma:lastestimate} with $\mathcal{I}_{ij}=\mathcal{III}_{ij}$, $\widetilde{C}=\theta_1$, and $\{B_i\}_{i}$ the collection of balls given by Lemma \ref{lem:CZweighted}, and by \eqref{CZ:sum} and \eqref{III2-Riesz},
 for $M\in \N$ so that $2M>\max\{\theta_2-2,\theta_1+r_w\mathfrak{r}_v(w)n\}$, we conclude that
\begin{align*}
III_1
\lesssim vw\Big(\bigcup_{i}B_i\Big)
\lesssim \frac{1}{\alpha^{p}}\int_{\R^n}|\nabla h|^{p}vdw.
\end{align*}

This, together with \eqref{Itermriesz}-\eqref{III-Riesz} and \eqref{spplitingriesz}, yields
\eqref{weakinterpolation}.
\end{proof}


\section{Proof of Theorem \ref{thm:maindegenerate}}\label{sec:proofmain}
In this section, we prove Theorem \ref{thm:maindegenerate} for functions $f\in \Scal$, and  conclude the result by a density argument.
Fix $w\in A_2(dx)$, $v\in A_{\infty}(w)$ and $f\in \Scal$, and
note that for every $(x,t)\in \R^{n+1}_+$  and $u(x,t):=\nabla_{x,t}e^{-t\sqrt{L_w}}f(x)$,
\begin{align*}
|u(x,t)|^2=|\nabla e^{-t\sqrt{L_w}}f(x)|^2+|\partial_te^{-t\sqrt{L_w}}f(x)|^2,
\end{align*}
where we define the  Poisson semigroup $\{e^{-t\sqrt{L_w}}\}_{t>0}$ using the classical subordination formula, or the functional calculus for $L_w$ (see \cite{Au07,CMR15}):
\begin{align}\label{subordinationformula}
 e^{-t\sqrt{L_w}}=C\int_{0}^{\infty}e^{-\lambda}\lambda^{\frac{1}{2}}e^{-\frac{t^2}{4\lambda}L_w}\frac{d\lambda}{\lambda}.
\end{align}
Therefore, it suffices to see that if $\mathcal{W}_v^w\left(\max\left\{r_w,q_-(L_w)\right\},q_+(L_w)\right)\neq \emptyset$, then 
\begin{align*}
\|\Ncal_w(\nabla e^{-t\sqrt{L_w}}f)\|_{L^p(vdw)}\!\lesssim\! \|\nabla f\|_{L^p(vdw)}
\quad
\!\textrm{and}
 \!\quad
\|\Ncal_w(\partial_te^{-t\sqrt{L_w}}f)\|_{L^p(vdw)}\!\lesssim\! \|\nabla f\|_{L^p(vdw)},
\end{align*}
for all $p\in \mathcal{W}_v^w\left(\max\left\{r_w,(q_-(L_w))_{w,*}\right\},q_+(L_w)\right)$.  We shall see this in Propositions \ref{prop:spaceD} and \ref{prop:timeD} below.
%
\subsection{Non-tangential maximal function estimate for the spatial derivatives }

\begin{proposition} \label{prop:spaceD}
Given $w\in A_2(dx)$ and $v\in A_{\infty}(w)$ such that  
$$
\mathcal{W}_v^w\left(\max\left\{r_w,q_-(L_w)\right\},q_+(L_w)\right)\neq \emptyset.
$$ 
Then, for all $f\in \Scal$ and $p\in \mathcal{W}_v^w\left(\max\left\{r_w,(q_-(L_w))_{w,*}\right\},q_+(L_w)\right)$, we have
\begin{equation}\label{eq:Gweighted}
\norm{\Ncal_w (\nabla e^{-t\sqrt{L_w}}f)}_{L^p(vdw)} \lesssim \norm{\nabla f}_{L^p(vdw)}.
\end{equation}
\end{proposition}
\begin{proof}First of all, fix $w\in A_2(dx)$
 and   define $q_-:=q_-(L_w)$ and $q_+:=q_+(L_w)$.  

 In the context of \eqref{eq:NTdeg}
we set $\alpha:=c_0c_1$. We claim that
\begin{align}\label{Nwspacecontrol}
\Ncal_w(\nabla e^{-t\sqrt{L_w}}f)(x)  
\lesssim
\sup_{t>0} \br{\,\fint_{B(x,\alpha t)} \abs{\nabla e^{-t\sqrt{L_w}} f(z)}^2 dw(z)}^{\frac{1}{2}}.
\end{align}
Indeed,  by \eqref{doublingcondition},
\begin{align*}
\Ncal_w (\nabla e^{-t\sqrt{L_w}}f)(x) 
&= \sup_{t>0} \br{\,\fint_{c_0^{-1}t<s<c_0t} \fint_{B(x,c_1t)}  \abs{\nabla e^{-s\sqrt{L_w}}  f(z)}^2 dw(z) ds}^{\frac{1}{2}}
\\ &\lesssim
\sup_{t>0} \sup_{c_0^{-1}t<s<c_0t} \br{\,\fint_{B(x,\alpha s)} \abs{\nabla e^{-s\sqrt{L_w}} f(z) }^2 dw(z) }^{\frac{1}{2}}
\\ &\lesssim
\sup_{t>0} \br{\,\dashint_{B(x,\alpha t)} \abs{\nabla e^{-t\sqrt{L_w}}  f(z)}^2 dw(z)}^{\frac{1}{2}}.
\end{align*}

Besides, by the subordination formula \eqref{subordinationformula} and Minkowski's integral inequality, 
\begin{multline*}
\br{\,\fint_{B(x,\alpha t)} \abs{\nabla e^{-t\sqrt{L_w}} f(z)}^2 dw(z)}^{\frac{1}{2}}
\\
\lesssim
 \int_{0}^{\frac{1}{4}} e^{-{\lambda}} {\lambda}^{\frac{1}{2}} \br{\,\fint_{B(x,\alpha t)} \Big|\nabla e^{-\frac{t^2}{4{\lambda}} L_w}f(z)\Big|^2 dw(z)}^{\frac{1}{2}} \frac{d{\lambda}}{{\lambda}}
\\  +
\int_{\frac{1}{4}}^{\infty} e^{-{\lambda}} {\lambda}^{\frac{1}{2}} \br{\,\fint_{B(x,\alpha t)} \Big|\nabla e^{-\frac{t^2}{4{\lambda}} L_w}f(z)\Big|^2 dw(z)}^{\frac{1}{2}} \frac{d{\lambda}}{{\lambda}}
=:I+II.
\end{multline*}

Dealing first with term $I$, note that
\begin{align*}
I \le &
\int_0^{\frac{1}{4}}  {\lambda}^{\frac{1}{2}} \br{\,\fint_{B(x,\alpha t)} \Big|\nabla e^{-t^2L_w}f(z)\Big|^2 dw(z)}^{\frac{1}{2}} \frac{d{\lambda}}{{\lambda}}
\\ &+
\int_{0}^{\frac{1}{4}} {\lambda}^{\frac{1}{2}} \br{\,\fint_{B(x,\alpha t)} \Big|\Big(\nabla e^{-\frac{t^2}{4{\lambda}} L_w}-\nabla e^{-t^2L_w}\Big)f(z)\Big|^2 dw(z)}^{\frac{1}{2}} \frac{d{\lambda}}{{\lambda}} 
=:I_1+I_2.
\end{align*}
In order to estimate term $I_1$, for any $p\in\mathcal{W}_v^w\left(\max\left\{r_w, (q_-)_{w,*}\right\},q_+\right)$, we pick $p_0$ in the interval $\left(\max\left\{r_w, (q_-)_{w,*}\right\},\min\{2,p\}\right)$, close enough to $\max\left\{r_w, (q_-)_{w,*}\right\}$ so that $v\in A_{\frac{p}{p_0}}(w)$ (see \eqref{eq:defi:rvw} and \eqref{intervalrsw}). Therefore $\Mcal_{p_0}^w(f):=\left(\Mcal_{p_0}^w(|f|^{p_0})\right)^{\frac{1}{p_0}}$ is bounded on $L^{p}(vdw)$. This and Lemma \ref{lem:Gsum} yield
\begin{align*}
 \Bigg\|\sup_{t>0} &\left(\fint_{B(\cdot,\alpha t)} \abs{\nabla e^{-t^2L_w}f(z)}^2 dw(z)\right)^{\frac{1}{2}}\Bigg\|_{L^p(vdw)}
\\ &\lesssim
\Bigg\|\sup_{t>0}\sum_{j\ge1} e^{-c4^j} \!\br{\,\fint_{B(\cdot,2^{j+1}\alpha t)} |\nabla f(z)|^{p_0} dw(z)}^{\frac{1}{p_0}}\Bigg\|_{L^p(vdw)}
\\ &\lesssim
\sum_{j\ge1}e^{-c4^j} \norm{\Mcal_{p_0}^w(\nabla f)}_{L^p(vdw)}
 \lesssim
\norm{\nabla f}_{L^p(vdw)}.
\end{align*}
Consequently, by Minkowski's integral inequality,
$$
\Big\|\sup_{t>0}I_1\Big\|_{L^p(vdw)} \!\!\!
\le\!\! \int_0^{\frac{1}{4}} \! {\lambda}^{\frac{1}{2}} \Bigg\|\sup_{t>0}\!\br{\, \fint_{B(\cdot,\alpha t)} \Big|\nabla e^{-t^2L_w}f(z)\Big|^2\! dw(z)\!}^{\!\!\frac{1}{2}}\Bigg\|_{\!L^p(vdw)} \!\!\frac{d{\lambda}}{{\lambda}} 
\lesssim \norm{\nabla f}_{L^p(vdw)}.
$$

Now we turn to the estimate of term $I_2$. Write 
\[
\nabla e^{-\frac{t^2}{4{\lambda}} L_w}-\nabla e^{-t^2L_w}=\nabla e^{-\frac{t^2}{2} L_w}\br{e^{-(\frac{1}{4{\lambda}}-\frac{1}{2}) t^2 L_w}-e^{-\frac{t^2}{2} L_w}}
\] 
and use again Lemma \ref{lem:Gsum} and \eqref{doublingcondition}. Then,
\begin{align*}
I_2 &=
\int_{0}^{\frac{1}{4}} {\lambda}^{\frac{1}{2}} \br{\fint_{B(x,\alpha t)} \abs{\nabla e^{-\frac{t^2}{2} L_w} \Big(e^{-(\frac{1}{4{\lambda}}-\frac{1}{2}) t^2 L_w}-e^{-\frac{t^2}{2} L_w}\Big)f(z)}^2 dw(z)}^{\frac{1}{2}} \frac{d{\lambda}}{{\lambda}}
\\ &\lesssim
 \sum_{j=1}^{\infty} e^{-c4^j} \int_{0}^{\frac{1}{4}} {\lambda}^{\frac{1}{2}} \br{\,\fint_{B(x,2^{j+2}\alpha t)} \abs{\nabla\Big(e^{-(\frac{1}{4{\lambda}}-\frac{1}{2}) t^2 L_w}-e^{-\frac{t^2}{2} L_w}\Big)f(z)}^{p_0} dw(z)}^{\frac{1}{p_0}} \frac{d{\lambda}}{{\lambda}}.
\end{align*}
Since $0<{\lambda}\leq1/4$, there holds 
\begin{align*}
&\abs{\Big(\nabla e^{-(\frac{1}{4{\lambda}}-\frac{1}{2}) t^2 L_w}-\nabla e^{-\frac{t^2}{2} L_w}\Big)f(z)} 
\\ &\leq 
\int_{\frac{t}{\sqrt 2}}^{t\sqrt{\frac{1}{4{\lambda}}-\frac{1}{2}}} \abs{\partial_s \nabla e^{-s^2 L_w} f(z)} ds
 \lesssim 
\int_{\frac{t}{\sqrt 2}}^{t\sqrt{\frac{1}{4{\lambda}}-\frac{1}{2}}} \abs{s^2\nabla L_w e^{-s^2 L_w} f(z)} \frac{ds}{s}
\\
&\lesssim
\br{\int_{0}^{\infty} \abs{s^2 \nabla L_w e^{-s^2 L_w} f(z)}^2 \frac{ds}{s}}^{\frac{1}{2}} \br{\log(2{\lambda})^{-\frac{1}{2}}}^{\frac{1}{2}}
 \lesssim
 \left(\log {\lambda}^{-1}\right)^{\frac{1}{2}}  \widetilde\Grm_{\hh}^{w} f(z), 
\end{align*}
where $ \widetilde\Grm_{\hh}^{w}$ is the vertical square function defined in \eqref{verticalnabaheat}. Then, we get
\begin{align*}
\sup_{t>0}I_2 &\lesssim
 \sum_{j\geq 1} e^{-c4^j}\int_{0}^{\frac{1}{4}} {\lambda}^{\frac{1}{2}} \left(\log {\lambda}^{-1}\right)^{\frac{1}{2}} \sup_{t>0} \br{\,\fint_{B(x,2^{j+2}\alpha t)} \abs{\widetilde\Grm_{\hh}^{w}f(z)}^{p_0} dw(z)}^{\frac{1}{p_0}} \frac{d{\lambda}}{{\lambda}}
\\ 
& \lesssim  \Mcal_{p_0}^w \br{\widetilde\Grm_{\hh}^{w}f}(x).
\end{align*}
Then, since  $\Mcal_{p_0}^w$ is bounded on $L^p(vdw)$, the above computation and  Theorem \ref{thm:wtG} imply
$$
\Big\|\sup_{t>0}I_2\Big\|_{L^p(vdw)} 
\le  \big\| \Mcal_{p_0}^w \big(\widetilde\Grm_{\hh}^{w} f\big)\big\|_{L^p(vdw)} 
\lesssim \big\|\widetilde\Grm_{\hh}^{w} f\big\|_{L^p(vdw)} \lesssim \norm{\nabla f}_{L^p(vdw)}.
$$

We finally estimate term $II$. Applying Lemma \ref{lem:Gsum},  we have that, for every $t>0$,
\begin{multline*}
\br{\fint_{B(x,\alpha 2\sqrt{{\lambda}}t)}\! |\nabla e^{-t^2L_w}f(z)|^2 dw(z)}^{\frac{1}{2}}
\\
\lesssim \Upsilon(\sqrt{{\lambda}})^{\theta}\!\sum_{j=1}^{\infty} e^{-c4^j}\! \br{\,\fint_{B(x,2^{j+2}\alpha  \sqrt{{\lambda}}t)}|\nabla f(z)|^{p_0} dw(z)}^{\frac{1}{p_0}}
\lesssim \Upsilon(\sqrt{{\lambda}})^{\theta}
\Mcal_{p_0}^w (\nabla f)(x).
\end{multline*}
Hence,
\begin{align*}
\Big\|\sup_{t>0}II\Big\|_{L^p(vdw)} 
 &\lesssim 
\int_{\frac{1}{4}}^{\infty} e^{-{\lambda}} \Upsilon(\sqrt{{\lambda}})^{\theta} \frac{d{\lambda}}{{\lambda}} \norm{\Mcal_{p_0}^w (\nabla f)}_{L^p(vdw)}
\lesssim \norm{\nabla f}_{L^p(vdw)}.
\end{align*}

Collecting the above estimates, we conclude \eqref{eq:Gweighted}.
\end{proof}

\subsection{Non-tangential maximal function estimate for the time derivative}
\begin{proposition} \label{prop:timeD}
Given $w\in A_2(dx)$ and $v\in A_{\infty}(w)$, assume that
$$
\mathcal{W}_v^w\left(\max\left\{r_w,q_-(L_w)\right\},q_+(L_w)\right)\neq \emptyset.
$$ 
Then, for all $p\in \mathcal{W}_v^w\left(\max\left\{r_w,(p_-(L_w))_{w,*}\right\},p_+(L_w)\right)$ and  $f\in \mathcal{S}$, we have
\begin{equation}\label{eq:timeD}
\norm{\Ncal_w (\partial_t e^{-t\sqrt{L_w}}f)}_{L^p(vdw)} \lesssim \norm{\nabla f}_{L^p(vdw)}.
\end{equation}
\end{proposition}
To prove this result, we need
 Theorem \ref{thm:boundednesswidetildeS}, a change of angles result in $L^p(vdw)$ for the operator defined in \eqref{eq:NTkappa}, and the boundedness of the non-tangential maximal square functions defined in \eqref{nontangheat} and \eqref{nontangpoisson}. We obtain these results in Lemma \ref{lema:changeofangelN},   and Proposition \ref{prop:boundedpoisson} below.

Our next result is an extension of \cite[Lemma 6.2]{HMay09} (see also \cite[Lemma 7.3]{MaPAII17}).
\begin{lemma}\label{lema:changeofangelN}
 Given $w\in A_{r}(dx)$ and $v\in A_{\widehat{r}}(w)$, $1\leq r,\widehat{r}<\infty$, let $0<p<\infty$ and $\kappa\geq 1$. There hold
\begin{equation}\label{changeangleNweakLp}
\|\mathcal{N}^{\kappa,w}F\|_{L^{p,\infty}(vdw)}\lesssim \kappa^{n\left(\frac{r+1}{2}+\frac{r\,\widehat{r}}{p}\right)} \|\mathcal{N}^{w}F\|_{L^{p,\infty}(vdw)},
\end{equation}
and 
\begin{equation}\label{changeangleNLp}
\|\mathcal{N}^{\kappa,w}F\|_{L^{p}(vdw)}
\lesssim
  \kappa^{n\left(\frac{r+1}{2}+\frac{r\,\widehat{r}}{p}\right)} 
\|\mathcal{N}^{w}F\|_{L^{p}(vdw)}.
\end{equation}
\end{lemma}
\begin{proof}
We will just prove \eqref{changeangleNweakLp}, the proof of \eqref{changeangleNLp} follows analogously by writing the $L^p(vdw)$-norm as an integral of the level sets. Details are left to the interested reader. 

Consider, for any $\lambda>0$,
$$
O_{\lambda}:=\{x\in \R^n: \mathcal{N}^wF(x)>\lambda\},\quad
 E_{\lambda}:=\R^n\backslash O_{\lambda},
$$
 and, for $\gamma=1-\frac{1}{[w]_{A_r(dx)}(11\kappa)^{rn}}$, the set of $\gamma$-density 
$$
E_{\lambda}^*:=
\Big\{x\in \R^n: \forall\, r>0,\,\frac{w(E_{\lambda}\cap B(x,r))}{w(B(x,r))}\geq\gamma\Big\}.
$$
Note that $O_{\lambda}^*:=\R^n\setminus E_{\lambda}^*=\left\{x\in \R^n:\mathcal{M}^w(\chi_{O_{\lambda}})(x)>\frac{1}{[w]_{A_r(dx)}(11\kappa)^{rn}}\right\}$.

We claim that
for every $\lambda>0$,
\begin{align}\label{claim:changeofanglesN}
\mathcal{N}^{\kappa,w}F(x)\leq  [w]_{A_r(dx)} 2^{\frac{nr}{2}}(9\kappa)^{\frac{n(r+1)}{2}}\lambda, \quad \quad \forall x\in E^{*}_{\lambda}.
\end{align}
Assuming this momentarily, let $0<p<\infty$. Since  $\mathcal{M}^w:L^{\widehat{r}}(vdw)\rightarrow L^{\widehat{r},\infty}(vdw)$, as we are assuming that $v\in A_{\widehat{r}}(w)$,   we get
\begin{align*}
&\|\mathcal{N}^{\kappa,w}F\|^p_{L^{p,\infty}(vdw)}
=
\sup_{\lambda>0}\lambda^{p}vw\br{\bbr{x\in \R^n:\mathcal{N}^{\kappa,w}F(x)>\lambda}}
\\
&=\!
\sup_{\lambda>0}\!\Big([w]_{A_r(dx)}  2^{\frac{nr}{2}} (9\kappa)^{\frac{n(r+1)}{2}}\lambda\Big)^{\!p} 
vw\Big(\!\Big\{x\in \R^n:\mathcal{N}^{\kappa,w}F(x)>[w]_{A_r(dx)}  2^{\frac{nr}{2}} (9\kappa)^{\frac{n(r+1)}{2}}\lambda\Big\}\!\Big)
\\
&\leq
[w]_{A_r(dx)}^{p}  2^{\frac{pnr}{2}} (9\kappa)^{\frac{n(r+1)p}{2}}\sup_{\lambda>0}\lambda^{p}vw(O_\lambda^*)
\lesssim
\kappa^{n\left(\frac{(r+1)p}{2}+r\,\widehat{r}\right)}\sup_{\lambda>0}\lambda^{p}vw(O_\lambda)
\\
&=
\kappa^{n\left(\frac{(r+1)p}{2}+r\,\widehat{r}\right)}
\|\mathcal{N}^wF\|_{L^{p,\infty}(vdw)}^p,
\end{align*}
which would finish the proof. 

It remains to show \eqref{claim:changeofanglesN}. First, note that if  $x\in E_{\lambda}^*$ and $t>0$, for every $y\in B(x,2\kappa t)$ we have $B(y,t/2)\cap E_{\lambda}\neq \emptyset$. To prove this, suppose by way of contradiction that $B(y,t/2)\subset O_{\lambda}$. Then, by \eqref{pesosineqw:Ap}, since $B(y,t/2)\subset B(x,3\kappa t)$ and $B(x,3\kappa t)\subset B(y, 5\kappa t)$, 
$$
\mathcal{M}^w(\chi_{O_{\lambda}})(x)\geq\frac{w(B(y,t/2))}{w(B(x,3\kappa t))}\geq \frac{w(B(y,t/2))}{w(B(y,5\kappa t))}\geq \frac{1}{[w]_{A_r(dx)}(10\kappa)^{rn}}>\frac{1}{[w]_{A_r(dx)}(11\kappa)^{rn}}.
$$ 
This implies that $x\in O_{\lambda}^*$, which contradicts our assumption.

Fix now $x\in E_{\lambda}^*$ and $t>0$, and note that if $y\in B(x,2\kappa t)$ there exists $x_0\in B(y,t/2)\cap E_{\lambda}$, hence $\mathcal{N}^wF(x_0)\leq \lambda$. Besides, since $B(y,t/2)\subset B(x_0,t)$ and by \eqref{pesosineqw:Ap}, for every $y\in B(x,2\kappa t)$,
\begin{multline}\label{menor-lambda}
\left(\int_{B(y,t/2)}|F(z,t)|^2\frac{dw(z)}{w(B(y,t/2))}\right)^{\frac{1}{2}}
\\
\leq [w]_{A_r(dx)}^{\frac{1}{2}}2^{\frac{nr}{2}}\sup_{s>0}
\left(\int_{B(x_0,s)}|F(z,s)|^2\frac{dw(z)}{w(B(x_0,s))}\right)^{\frac{1}{2}}
\\
= [w]_{A_r(dx)}^{\frac{1}{2}}2^{\frac{nr}{2}}\mathcal{N}^wF(x_0)\leq [w]_{A_r(dx)}^{\frac{1}{2}}2^{\frac{nr}{2}}\lambda.
\end{multline}

On the other hand, for every $x\in \R^n$ and $t>0$, we have that 
$B(x,\kappa t)\subset\bigcup_iB(x_i,t/2)$, where $\{B(x_i,t/2)\}_i$ is a collection of at most $(9\kappa)^n$ balls such that, for every $i$, we have that $x_i\in B(x,2\kappa t)$. In particular, $B(x_i, t/2),\,B(x,t)\subset B(x_i,3\kappa t)$. 

Therefore, by the above observations and  \eqref{pesosineqw:Ap}, we conclude that
\begin{multline*}
\int_{B(x,\kappa t)}|F(y,t)|^2\frac{dw(y)}{w(B(x,t))}
\leq
[w]_{A_r(dx)}(3\kappa)^{nr}
 \sum_i
\int_{B(x_i,t/2)}|F(y,t)|^2\frac{dw(y)}{w(B(x_i,t/2))}
\\ \leq 
(9\kappa)^{n(r+1)} [w]_{A_r(dx)}^2 2^{{nr}} \lambda^2,
\end{multline*}
where we have used \eqref{menor-lambda}, since  $x_i\in B(x, 2\kappa t)$.
Finally, taking the supremum over all $t>0$, we obtain
$$
\mathcal{N}^{\kappa,w}F(x)^2\leq [w]_{A_r(dx)}^2 2^{nr}(9\kappa)^{n(r+1)} \lambda^2,\quad \forall\,x\in E_{\lambda}^*.
$$
This readily gives
\eqref{claim:changeofanglesN} and the proof is complete.
\end{proof}

 \begin{proposition}\label{prop:boundedpoisson}
Let $L_w$ be a degenerate elliptic operator with $w\in A_2(dx)$ and let $v\in A_{\infty}(w)$. Then
 \begin{list}{$(\theenumi)$}{\usecounter{enumi}\leftmargin=1cm \labelwidth=1cm\itemsep=0.2cm\topsep=.0cm \renewcommand{\theenumi}{\alph{enumi}}}
 
\item $\Ncal_{\hh}^{w} $ is bounded on $L^p(vdw)$  for all $p\in \mathcal{W}_v^w\left(p_-(L_w),\infty\right)$.
 
\item $\Ncal_{\pp}^{w} $ is bounded on $L^p(vdw)$ for all $p\in \mathcal{W}_v^w\left(p_-(L_w),p_+(L_w)\right)$.

\end{list}
 \end{proposition}
\begin{proof}
We first prove part $(a)$. Fix $p\in \mathcal{W}_v^w(p_-(L_w),\infty)$ and choose $p_0$ close enough to $p_-(L_w)$ so that 
\begin{equation}\label{p0comparison}
p_-(L_w)<p_0<\min\{2,p\} \quad \text{and}\quad v\in A_{\frac{p}{p_0}}(w).
\end{equation}
Then $e^{-\tau L_w}\in \mathcal O(L^{p_0}(w)-L^2(w))$. 
This fact and  \eqref{doublingcondition} yield
\begin{multline*}
\Ncal_{\hh}^{w}f(x)
 \lesssim  \sup_{t>0}
\sum_{j\geq 1} \br{\fint_{B(x, t)} \abs{e^{-t^2L_w} \left(\chi_{C_j(B(x,t))}f\right)(z)}^2 dw(z)}^{\frac{1}{2}}
\\ 
\lesssim \sup_{t>0}
\sum_{j\geq 1} 2^{j\theta_1} \Upsilon\br{2^{j+1}}^{\theta_2} e^{-c 4^j} \br{\fint_{C_j(B(x,t))} |f(z)|^{p_0} dw(z)}^{\frac{1}{p_0}}
\lesssim 
\Mcal_{p_0}^wf(x).
\end{multline*}
Consequently,  
$$
\|\Ncal_{\hh}^{w}f\|_{L^p(vdw)}\lesssim \|\mathcal{M}_{p_0}^wf\|_{L^p(vdw)}\lesssim \|f\|_{L^p(vdw)},
$$
since  $\Mcal_{p_0}^w$ is bounded on $L^p(vdw)$ by our choice of $p_0$.

\medskip

We now prove part $(b)$. Note that 
\begin{align*}
\!\!\!\!\Ncal_{\pp}^{w}f(x)\! \leq \!\Ncal_{\hh}^{w}f(x)
\!+\!
\sup_{t>0}\! \br{\fint_{B(x, t)} \!\abs{\!\br{\!e^{-t\sqrt {L_w}}\!-\!e^{-t^2L_w}\!\!}\!f(z)}^2\!\! dw(z)\!\!}^{\!\!\!\frac{1}{2}}\!\!\!=:\!\Ncal_{\hh}^wf(x)\!+ \!\sup_{t>0}I\!.
\end{align*}
By the subordination formula \eqref{subordinationformula} and Minkowski's integral inequality, 
\begin{align*}
I&
=
\br{\fint_{B(x,t)} \abs{\int_{0}^{\infty} e^{-{\lambda}} {\lambda}^{\frac{1}{2}} \br{e^{-\frac{t^2}{4{\lambda}} L_w}-e^{-t^2L_w}}f(z) \frac{d{\lambda}}{{\lambda}} }^2 dw(z)}^{\frac{1}{2}}
\\ &\leq
\int_{0}^{\infty} e^{-{\lambda}} {\lambda}^{\frac{1}{2}} \br{\fint_{B(x,t)} \abs{\br{e^{-\frac{t^2}{4{\lambda}} L_w}-e^{-t^2L_w}}f(z)}^2 dw(z)}^{\frac{1}{2}} \frac{d{\lambda}}{{\lambda}}
\\& =
\int_{0}^{\frac{1}{4}} e^{-{\lambda}} {\lambda}^{\frac{1}{2}} \br{\fint_{B(x,t)} \abs{\br{e^{-\frac{t^2}{4{\lambda}} L_w}-e^{-t^2L_w}}f(z)}^2 dw(z)}^{\frac{1}{2}} \frac{d{\lambda}}{{\lambda}}
\\&\qquad\quad +
\int_{\frac{1}{4}}^{\infty} e^{-{\lambda}} {\lambda}^{\frac{1}{2}} \br{\fint_{B(x,t)} \abs{\br{e^{-\frac{t^2}{4{\lambda}} L_w}-e^{-t^2L_w}}f(z)}^2 dw(z)}^{\frac{1}{2}} \frac{d{\lambda}}{{\lambda}}
\\
&=:I_1+I_2.
\end{align*}
In order to estimate $I_1$ and $I_2$, for each $t>0$ consider $h_t:=\br{e^{-(\frac{1}{4{\lambda}}-\frac{1}{2}) t^2 L_w}-e^{-\frac{t^2}{2} L_w}}f$. Next fix $p\in \mathcal{W}_v^w\left(p_-(L_w),p_+(L_w)\right)$, and choose $p_0$ as  in \eqref{p0comparison}. Then, applying the fact that $e^{-\tau L_w} \in \mathcal O(L^{p_0}(w)-L^2(w))$, we have 
\begin{align*}
I_1 
&
\leq
\int_{0}^{\frac{1}{4}} {\lambda}^{\frac{1}{2}} \br{\fint_{B(x,t)} \abs{e^{-\frac{t^2}{2} L_w} h_t(z)}^2 dw(z)}^{\frac{1}{2}} \frac{d{\lambda}}{{\lambda}}
\\
&
\leq
\int_{0}^{\frac{1}{4}} {\lambda}^{\frac{1}{2}} \sum_{j\geq 1} \br{\fint_{B(x, t)} \abs{e^{-\frac{t^2}{2} L_w} \left(\chi_{C_j(B)}h_t\right)(z)}^2 dw(z)}^{\frac{1}{2}} \frac{d{\lambda}}{{\lambda}}
\\
&
\leq
\int_{0}^{\frac{1}{4}} {\lambda}^{\frac{1}{2}} \sum_{j\geq 1} 2^{j\theta_1} \Upsilon\br{2^{j+1}}^{\theta_2} e^{-c4^j} \br{\fint_{B(x,2^{j+1} t)} |h_t(z)|^{p_0} dw(z)}^{\frac{1}{p_0}} \frac{d{\lambda}}{{\lambda}}
\\
&
\lesssim \sum_{j\geq 1} e^{-c4^j} \int_{0}^{\frac{1}{4}} {\lambda}^{\frac{1}{2}} \br{\fint_{B(x,2^{j+1} t)} |h_t(z)|^{p_0} dw(z)}^{\frac{1}{p_0}} \frac{d{\lambda}}{{\lambda}}.
\end{align*}
When $0<{\lambda}\leq1/4$,   Cauchy-Schwartz inequality implies
\begin{multline*}
\abs{\br{e^{-(\frac{1}{4{\lambda}}-\frac{1}{2}) t^2 L_w}-e^{-\frac{t^2}{2} L_w}}f(z)} 
\leq 
\int_{\frac{t}{\sqrt 2}}^{t\sqrt{\frac{1}{4{\lambda}}-\frac{1}{2}}} \abs{\partial_s e^{-s^2 L_w} f(z)} ds
\\
\lesssim
\int_{\frac{t}{\sqrt 2}}^{t\sqrt{\frac{1}{4{\lambda}}-\frac{1}{2}}} \abs{s^2 L_w e^{-s^2 L_w} f(z)} \frac{ds}{s}
\lesssim
\mathrm{g}_{\hh,t}^{w} f(z)\, (\log {\lambda}^{-1})^{\frac{1}{2}},
\end{multline*}
where 
\begin{align}\label{ghht}
\mathrm{g}_{\hh,t}^{w} f(z):=\br{\int_{\frac{t}{2}}^{\infty} \abs{s^2 L_w e^{-s^2 L_w} f(z)}^2 \frac{ds}{s}}^{\frac{1}{2}}.
\end{align}
Therefore,
\begin{multline*}
I_1
\lesssim  
\sum_{j\geq 1} e^{-c4^j} \int_{0}^{\frac{1}{4}} {\lambda}^{\frac{1}{2}} (\log {\lambda}^{-1})^{\frac{1}{2}} \frac{d{\lambda}}{{\lambda}}  \br{\fint_{B(x,2^{j+1} t)} |\mathrm{g}_{\hh,t}^{w} f(z)|^{p_0} dw(z)}^{\frac{1}{p_0}} 
\\
\lesssim  
\sum_{j\geq 1} e^{-c4^j} \br{\fint_{B(x,2^{j+1} t)} |\mathrm{g}_{\hh,t}^{w} f(z)|^{p_0} dw(z)}^{\frac{1}{p_0}}.
\end{multline*}

As for term $I_2$, note that if ${\lambda}>1/4$, we have that
\begin{multline*}
\abs{\br{e^{-\frac{t^2}{4{\lambda}} L_w}-e^{-t^2L_w}}f(z)}
\lesssim
\int_{\frac{t}{2 \sqrt {\lambda}}}^{t} \abs{s^2L_w e^{-s^2 L_w} f(z)} \frac{ds}{s}
\\
\lesssim
\br{\int_{\frac{t}{2 \sqrt {\lambda}}}^{t} \abs{s^2L_w e^{-s^2 L_w} f(z)}^2 \frac{ds}{s}}^{\frac{1}{2}} \br{\log (4 {\lambda})}^{\frac{1}{2}}.
\end{multline*}
Then, by \eqref{doublingcondition},
\begin{align*}
I_2 
&\lesssim
\int_{\frac{1}{4}}^{\infty} e^{-c{\lambda}}  \br{\int_{\frac{t}{2 \sqrt {\lambda}}}^{t}\int_{B(x, t)}  \abs{s^2L_w e^{-s^2 L_w} f(z)}^2 \frac{dw(z)ds}{sw(B(z,t))} }^{\frac{1}{2}} \frac{d{\lambda}}{{\lambda}} 
\\ 
&\lesssim
\int_{\frac{1}{4}}^{\infty} e^{-c{\lambda}}
\br{\int_{\frac{t}{2 \sqrt {\lambda}}}^{t} \int_{B(x,2\sqrt {\lambda} s)} \abs{s^2L_w e^{-s^2 L_w} f(z)}^2 \frac{ dw(z)\,ds}{s\,w(B(z,s))}}^{\frac{1}{2}} \frac{d{\lambda}}{{\lambda}} 
\\ 
&\lesssim
\int_{\frac{1}{4}}^{\infty}  e^{-c{\lambda}}\Scal_{\hh}^{2\sqrt {\lambda},w}f(x) \frac{d{\lambda}}{{\lambda}},
\end{align*}
where $\Scal_{\hh}^{2\sqrt {\lambda},w}$ is defined in \eqref{eq:conicalheat}.

From the above estimates for $I_1$ and $I_2$, we obtain that, for all $x\in \R^n$,
\begin{align}\label{comparisonpoisson}
\Ncal_{\pp}^{w}f(x)
&\lesssim
\Ncal_{\hh}^{w} f(x)+
\sup_{t>0}
\sum_{j\geq 1} e^{-c4^j} \br{\fint_{B(x,2^{j+1} t)} |\mathrm{g}_{\hh,t}^{w} f(z)|^{p_0} dw(z)}^{\frac{1}{p_0}}
\\\nonumber&\qquad \quad
+\int_{\frac{1}{4}}^{\infty}  e^{-c{\lambda}} \Scal_{\hh}^{2\sqrt {\lambda},w}f(x) \frac{d{\lambda}}{{\lambda}}
\\\nonumber&
\lesssim \Ncal_{\hh}^{w} f(x)+\Mcal_{p_0}^w (\mathrm{g}_{\hh}^w f)(x)+\int_{\frac{1}{4}}^{\infty}  e^{-c{\lambda}} \Scal_{\hh}^{2\sqrt {\lambda},w}f(x) \frac{d{\lambda}}{{\lambda}},
\end{align}
where $\mathrm{g}_{\hh}^w$ is defined in \eqref{verticalheat-1}.

To proceed, we first note that $\mathcal{M}_{p_0}^w$ is bounded on $L^p(vdw)$, since $v\in A_{\frac{p}{p_0}}(w)$, and so is $\mathrm{g}_{\hh}^w$ (see \cite{CMR15}) since $p\in \mathcal{W}_v^w(p_-(L_w),p_+(L_w))$. Using this and invoking Proposition \ref{prop:alpha} and  \cite[Theorem 3.1]{ChMPA16},  for some $\theta>0$ depending on $v$, $w$, and $n$, we conclude that
\begin{multline*}
\norm{\Ncal_{\pp}^{w}f}_{L^p(vdw)} 
\lesssim 
\norm{\Ncal_{\hh}^{w} f}_{L^p(vdw)}
+
\norm{\Mcal_{p_0}^w \br{\mathrm{g}_{\hh}^{w} f}}_{L^p(vdw)}
+
\int_{\frac{1}{4}}^{\infty}  e^{-c{\lambda}} \norm{\Scal_{\hh}^{2\sqrt {\lambda},w}f}_{L^p(vdw)} \frac{d{\lambda}}{{\lambda}}
\\ 
\lesssim 
\norm{f}_{L^p(vdw)}
+
\norm{\mathrm{g}_{\hh}^{w} f}_{L^p(vdw)}
+
\norm{f}_{L^p(vdw)} \int_{\frac{1}{4}}^{\infty} {\lambda}^{\theta}e^{-c{\lambda}} \frac{d{\lambda}}{{\lambda}}
\lesssim 
\norm{f}_{L^p(vdw)}.
\end{multline*} 
This completes the proof.
\end{proof}

\medskip

\begin{proof}[Proof of Proposition \ref{prop:timeD}]
{  First of all, fix $w\in A_2(dx)$ and denote 
\[
q_-:=p_-(L_w)=q_-(L_w), \quad q_+:=q_+(L_w), \quad p_+:=p_+(L_ w),
\]
 and 
\[
u(x,t):=\partial_t e^{-t\sqrt{L_w}}f(x)=-\sqrt{L_w}e^{-t\sqrt{L_w}}f(x).
\]}
  
From the definitions of $\Ncal_w$ and $\Ncal_{\pp}^{w}$ (see  \eqref{eq:NTdeg} and \eqref{nontangpoisson}), proceeding as in the proof of \eqref{Nwspacecontrol} we have that
\begin{equation}\label{eq:NTpoisson-w}
\Ncal_w u(x) \lesssim \Ncal_{\pp}^{\alpha,w}(\!\sqrt{L_w} f)(x),\quad \forall x\in \R^n,
\end{equation}
with $\alpha=c_0c_1$.
Consequently, Lemma \ref{lema:changeofangelN}, and Propositions  \ref{prop:boundedpoisson} $(b)$ and  \ref{prop:wRR} imply
\begin{equation*}
\|\Ncal_wu\|_{L^p(vdw)}\lesssim \big\|\Ncal_{\pp}^{w} (\!\sqrt{L_w} f)\big\|_{L^p(vdw)}\lesssim \big\|\!\sqrt{L_w}f\big\|_{L^p(vdw)} \lesssim \norm{\nabla f}_{L^p(vdw)},
\end{equation*}
  for all  $p\in \mathcal{W}_v^w(\max\{r_w,q_-\},p_+)=\big(\mathfrak{r}_v(w)\max\{r_w,q_-\},p_+/\mathfrak{s}_v(w)\big)$ and $f\in \Scal$.
  
Our goal is to obtain \eqref{eq:timeD} for  all 
$$
p\in \mathcal{W}_v^w(\max\{r_w,(q_-)_{w,*}\},p_+)=\big(\mathfrak{r}_v(w)\max\{r_w,(q_-)_{w,*}\},p_+/\mathfrak{s}_v(w)\big).
$$ 
Recall that $(q_-)_{w,*}<q_-$ (see \eqref{p_{w,*}}).
Hence,
 fix $p$  such that
 \begin{align}\label{pinterval}
\mathfrak{r}_v(w)\max\left\{r_w,(q_-)_{w,*}\right\}<p<\mathfrak{r}_v(w)\max\left\{r_w,q_-\right\}.
 \end{align} 
   Then, in view of
 inequality \eqref{eq:NTpoisson-w} and Lemma \ref{lema:changeofangelN}, if we show that, for all $f\in \Scal$,
\begin{align}\label{goal}
\big\|\mathcal{N}_{\pp}^w\big(\!\sqrt{L_w}f\big)\big\|_{L^{p,\infty}(vdw)}\lesssim \|\nabla f\|_{L^p(vdw)} ,
\end{align}
by interpolation, see \cite{Ba09} and Remark \ref{remark:product-weight}, we would conclude the desired estimate.

Given $\alpha>0$, take a function $f\in \Scal$. We apply Lemma \ref{lem:CZweighted} to $f$, $\alpha$, and the product weight ${\varpi}=vw$ (note that $vw\in A_p(dx)$ since $r_{vw}\leq r_w\mathfrak{r}_v(w)<p$, see Remark \ref{remark:product-weight}). 
Let $\{B_i\}_{i}$ be the collection of balls given by Lemma \ref{lem:CZweighted}. Consider for $M\in \N$ arbitrarily large, 
\[
B_{r_{B_i}}:=(I-e^{-r_{B_i}^2L_w})^M,\qquad A_{r_{B_i}}:=I-B_{r_{B_i}}=\sum_{k=1}^MC_{k,M}e^{-kr_{B_i}^2L_w}.
\]
 Hence,
\begin{align}\label{decompositionf}
f=g+\sum_{i}A_{r_{B_i}}b_i+\sum_{i}B_{r_{B_i}}b_i=:g+\widetilde{b}+\widehat{b}.
\end{align}

   To prove the weak-type estimates for $g,\widetilde{b}$, and $\widehat{b}$, we  need some preparations. On the one hand, since we assume that $\mathcal{W}_v^w(\max\{r_w,q_-\},q_+)\neq \emptyset$, by \eqref{pinterval} and \eqref{intervalnotempty2} we can take $p_1$ satisfying
 \begin{align}\label{p_1interval}
 \mathfrak{r}_v(w)\max\left\{r_w,q_-\right\}<p_1<\min\Big\{\frac{q_+}{\mathfrak{s}_v(w)},p_{vw}^*\Big\}.
 \end{align}
In particular,  $\mathfrak{r}_v(w)q_-<p_1<\frac{q_+}{\mathfrak{s}_v(w)}$, that is, $p_1\in \mathcal{W}_v^w(q_-,q_+)$. This can be written as
  \begin{align}\label{p_1clasesARH}
 v\in A_{\frac{p_1}{q_-}}(w)\cap RH_{\left(\frac{q_+}{p_1}\right)'}(w).
 \end{align}
  
  On the other hand,   take $p_0$ satisfying $q_-<p_0<\min\{2,p_1\}$ close enough to $q_-$, and $q_0$ satisfying $\max\{2,p_1\}<q_0<q_+$ close enough to $q_+$, so that 
\begin{align}\label{p_0choice}
v\in A_{\frac{p_1}{p_0}}(w) \cap RH_{\left(\frac{q_0}{p_1}\right)'}(w).
\end{align}

Next, {  by \eqref{comparisonpoisson} and  \eqref{doublingcondition},  we have that for any function $h\in L^2(w)$
\begin{align*}
\mathcal{N}_{\pp}^wh(x)
\lesssim\,&
 \Ncal_{\hh}^wh(x)
+
\sum_{l\geq 1}e^{-c4^l}
\sup_{t>0}\left(\dashint_{B(x,2^{l+1}t)}
\left|\mathrm{g}_{\hh,t}^wh(y)\right|^{{p_0}}dw(y)\right)^{\frac{1}{p_0}}
\\
&+ 
\int_{\frac{1}{4}}^{\infty}e^{-c\lambda}\Scal_{\hh}^{\,2\sqrt{\lambda},w}h(x)\frac{d\lambda}{\lambda}
=: \Ncal_{\hh}^wh(x)+\sum_{l\geq 1}e^{-c4^{l}}\mathfrak{O}_{2,l}h(x)+\mathfrak{O}_3h(x),
\end{align*}
where we recall that  $\mathrm{g}_{\hh,t}^w$ is defined in \eqref{ghht}.
Besides, note that, the fact that $e^{-\tau L_w} \in \Ocal(L^{p_0}(w)-L^2(w))$, and  \eqref{doublingcondition} yield
\begin{align*}
 \Ncal_{\hh}^wh(x)&=\sup_{t>0}\left(\dashint_{B(x,t)}\left|e^{-t^2L_w}h(y)\right|^2dw(y)\right)^{\frac{1}{2}}
 \\&
 \lesssim
 \sum_{l\geq 1}
 e^{-c4^l}
 \sup_{t>0}\left(\dashint_{B(x,2^{l+1}t)}\left|e^{-\frac{t^2}{2}L_w}h(y)\right|^{p_0}dw(y)\right)^{\frac{1}{p_0}}
  \\&
\lesssim
 \sum_{l\geq 1}
 e^{-c4^l}
 \sup_{t>0}\left(\dashint_{B(x,2^{l+2}t)}\left|e^{-t^2L_w}h(y)\right|^{p_0}dw(y)\right)^{\frac{1}{p_0}}
 \\
 &=:\sum_{l\geq 1}e^{-c4^l}\mathfrak{O}_{1,l}h(x).
\end{align*}
Therefore, for any function $h\in L^2(w)$, we have that
\begin{align*}
\Ncal_{\pp}^wh(x)
\leq C\left(
\sum_{l\geq 1}e^{-c4^l}{\mathfrak{O}}_{1,l}h(x)
+
\sum_{l\geq 1}e^{-c4^l}\mathfrak{O}_{2,l}h(x)
+\mathfrak{O}_3h(x)\right),\quad \forall x\in \R^n.
\end{align*}}
Using this and \eqref{decompositionf}, we get
\begin{align}\label{maintimederivative}
&vw\Big(\Big\{x\in \R^n:  \, \Ncal_{\pp}^w\left(\sqrt{L_w}f\right)(x)>\alpha\Big\}\Big)
\\\nonumber
&\qquad\leq
vw\left(\left\{x\in \R^n:\Ncal_{\pp}^w\left(\sqrt{L_w}g\right)(x)>\frac{\alpha}{5}\right\}\right)
\\\nonumber&
\qquad\qquad+
vw\left(\left\{x\in \R^n:\Ncal_{\pp}^w\left(\sqrt{L_w}\,\widetilde{b}\,\right)(x)>\frac{\alpha}{5}\right\}\right)
\\\nonumber&
\qquad\qquad+\sum_{m=1}^{2}
vw\Big(\Big\{x\in \R^n:C\sum_{l\geq 1}e^{-c4^l}\mathfrak{O}_{m,l}\Big(\sqrt{L_w}\, \widehat{b}\Big)(x)>\frac{\alpha}{5}\Big\}\Big)
\\\nonumber&
\qquad\qquad+
vw\left(\left\{x\in \R^n:C\mathfrak{O}_{3}\left(\sqrt{L_w}\,\widehat{b}\right)(x)>\frac{\alpha}{5}\right\}\right)
\\\nonumber&
\qquad=:I+II+\sum_{m=1}^2III_{m}+IV.
\end{align}

In order to estimate $I$, first note that $p<p_1$ (see \eqref{pinterval} and \eqref{p_1interval}). Then, 
apply Chebyshev's inequality, Propositions \ref{prop:boundedpoisson} $(b)$ and  \ref{prop:wRR}, and properties \eqref{CZ:decomp}-\eqref{CZ:overlap} to get
\begin{align}\label{firstermsum}
I
&\lesssim
 \frac{1}{\alpha^{p_1}}\int_{\R^n}\big|\mathcal{N}_{\pp}^w\big(\sqrt{L_w}g\big)\big|^{p_1}vdw
 \lesssim 
 \frac{1}{\alpha^{p_1}}\int_{\R^n}|\nabla g|^{p_1}vdw
\lesssim
   \frac{1}{\alpha^{p}}\int_{\R^n}|\nabla f|^{p}vdw.
\end{align}

Now we estimate
$II$. 
To this end, apply Chebyshev's inequality, Propositions \ref{prop:boundedpoisson} $(b)$ and \ref{prop:wRR}, and Lemma \ref{lemma:arbi}. Then,
\begin{multline}\label{IItermtimederivative}
II \lesssim\frac{1}{\alpha^{p_1}}\int_{\R^n} \left|{\Ncal}_{\pp}^w\left(\sqrt{L_w}\,\widetilde{b}\,\right)\right|^{p_1}vdw
\lesssim\frac{1}{\alpha^{p_1}}\int_{\R^n} \big|\nabla\widetilde{b}\big|^{p_1}vdw
\\
\lesssim \frac{1}{\alpha^{p}}\int_{\R^n}|\nabla f|^{p}vdw.
\end{multline}

We next estimate $IV$.
With this aim, we write  $b=\sum_i b_i$  so that $\widehat{b}=b-\widetilde{b}$, and note that 
\begin{multline*}
IV
\leq  vw\left(\left\{x\in \R^n:C\mathfrak{O}_{3}\left(
\sqrt{L_w}b\right)(x)>\frac{\alpha}{10}\right\}\right)
\\
+
 vw\left(\left\{x\in \R^n:C\mathfrak{O}_{3}\left(\sqrt{L_w}\,\widetilde{b}\right)(x)>\frac{\alpha}{10}\right\}\right)
=: IV_1+IV_2.
\end{multline*}
In order to estimate $IV_1$ apply Chebyshev's inequality, Minkowski's integral inequality, and Proposition \ref{prop:alpha}, then 
 \begin{align*}
IV_1
&
\lesssim
\frac{1}{\alpha^{p}}\left(\int_{\frac{1}{4}}^{\infty}
e^{-cu}\left\|\Scal_{\hh}^{2\sqrt{u},w}\big(\sqrt{L_w}\,b\big)\right\|_{L^{p}(vdw)}
\frac{du}{u}\right)^{p}
\lesssim
\frac{1}{\alpha^{p}}\left\|\Scal_{\hh}^w\big(\sqrt{L_w}\, {b}\big)\right\|_{L^{p}(vdw)}^{p}
\\
&
\lesssim
\frac{1}{\alpha^{p}}\left\|\Scal_{1/2,\hh}^w\big(\sqrt{L_w}\, {b}\big)\right\|_{L^{p}(vdw)}^{p}
= \frac{1}{\alpha^{p}}\left\|\wt\Scal_{\hh}^w {b} \right\|_{L^{p}(vdw)}^{p}
\lesssim
\frac{1}{\alpha^{p}} \left\|\nabla b\right\|_{L^{p}(vdw)}^p
\\
&\lesssim
\frac{1}{\alpha^{p}} \sum_{i}\int_{B_i}|\nabla b_i|^p vdw
\lesssim 
\sum_{i}vw(B_i)
\lesssim \frac{1}{\alpha^{p}}\int_{\R^n}|\nabla f|^p vdw,
\end{align*}
where we have used Proposition \ref{prop:widetildeS-heatS} in the third inequality, Theorem \ref{thm:boundednesswidetildeS} in the fourth inequality, and  the last two inequalities follow from \eqref{CZ:b} and \eqref{CZ:sum}.

As for the estimate of $IV_2$  apply again Chebyshev's inequality, Minkowski's integral inequality and Proposition \ref{prop:alpha}.
Then, \cite[Theorem 3.1]{ChMPA16},  Proposition \ref{prop:wRR}, and Lemma \ref{lemma:arbi} readily give 
\begin{multline*}
IV_2
\lesssim
\frac{1}{\alpha^{p_1}}\left(\int_{\frac{1}{4}}^{\infty}
e^{-cu}\left\|\Scal_{\hh}^{2\sqrt{u},w}\big(\sqrt{L_w}\, \widetilde{b}\,\big)\right\|_{L^{p_1}(vdw)}
\frac{du}{u}\right)^{p_1}
\\
\lesssim
\frac{1}{\alpha^{p_1}}
\left\|\Scal_{\hh}^w\big(\sqrt{L_w}\, \widetilde{b}\,\big)\right\|_{L^{p_1}(vdw)}^{p_1}
\lesssim
\frac{1}{\alpha^{p_1}}
\big\|\nabla \widetilde{b}\,\big\|_{L^{p_1}(vdw)}^{p_1}
\lesssim 
\frac{1}{\alpha^{p}}\int_{\R^n}|\nabla f|^p vdw.
\end{multline*}
Therefore,  we conclude that
\begin{align}\label{termIIItimederivative}
IV
\lesssim
\frac{1}{\alpha^{p}}\int_{\R^n}|\nabla f|^{p}vdw.
\end{align}

 Now, it remains to estimate
$
III_m$, for $m=1,2$.
Note that by  \eqref{CZ:sum},
\begin{align}\label{termII_2}
III_{m}
\leq &
\,vw\!\Big(\!\bigcup_{i} 16B_i\!\Big)\!\!+\!\!
vw\Big(\Big\{x\in \R^n\setminus \cup_{i} 16B_i:C\!\sum_{l\geq 1}e^{-c4^l}\!\mathfrak{O}_{m,l}\big(\sqrt{L_w}\,\widehat{b}\,\big)(x)\!>\!\frac{\alpha}{5}\Big\}\Big)
\\\nonumber
 \lesssim &
\frac{1}{\alpha^{p}}\!\int_{\R^n}\!|\nabla f|^{p}vdw\!+\!\sum_{l\geq 1}
\!vw\Big(
\Big\{x\in \R^n\setminus \cup_{i} 16B_i:\mathfrak{O}_{m,l}\big(\sqrt{L_w}\,\widehat{b}\,\big)(x)\!>\!\frac{e^{c4^l}\alpha}{C2^{l}}
\Big\}\Big)
\\\nonumber
=: &
\frac{1}{\alpha^{p}}\int_{\R^n}|\nabla f|^{p}vdw+\sum_{l\geq 1}III_{m,l}.
\end{align}
Applying Chebyshev's inequality,  duality, and H\"older's inequality, it follows that
\begin{align}\label{sum:termII_2}
&III_{m,l}\lesssim
\frac{e^{-c4^l}}{\alpha^{p_1}}
\int_{\R^n\setminus \cup_{i} 16B_i}
\left|\mathfrak{O}_{m,l}\big(\sqrt{L_w}\,\widehat{b}\,\big)\right|^{p_1}vdw
 \\\nonumber&\,\,
\lesssim
\frac{e^{-c4^l}}{\alpha^{p_1}}\!\Bigg(\sup_{\|u\|_{\!L^{\!p_1'}(vdw)}\!=1}\!\sum_{i}\sum_{j\geq 4}\!\br{\int_{C_j(B_i)}\!
\abs{\mathfrak{O}_{m,l}\!\left(\!\sqrt{L_w}\br{\!B_{r_{B_i}}b_i\!}\!\right)\!}^{\!p_1}\!vdw\!\!}^{\!\!\!\frac{1}{p_1}}\!\!\|u\chi_{C_j(B_i)}\|_{\!L^{p_1'}(vdw)}\Bigg)^{\!\!\!p_1}
\\\nonumber &\,\,
=:
\frac{e^{-c4^l}}{\alpha^{p_1}}\Bigg(\sup_{\|u\|_{L^{p_1'}(vdw)}=1}\sum_{i}\sum_{j\geq 4}I_{m,l}^{ij}\,\|u\chi_{C_j(B_i)}\|_{L^{p_1'}(vdw)}\Bigg)^{p_1}.
\end{align}
Then, for $m=1$, we have that 
\begin{align*}
&{I}_{1,l}^{ij}
\\&
\lesssim\!\!
\left(\!\int_{C_j(B_i)}\!\!\left(\sup_{0<t<2^{j-l-3}r_{B_i}}\!\!\left(\dashint_{B(x,2^{l+2}t)}\!
\left|e^{-t^2L_w}\!\sqrt{L_w}\!\left(\!B_{r_{B_i}}b_i\!\right)\!(y)\!\right|^{p_0}\!\!dw(y)\!\right)^{\!\!\frac{1}{p_0}}\!\right)^{\!\!p_1}\!\!\!v(x)dw(x)\!\right)^{\!\!\!\frac{1}{p_1}}
\\
&\,+\!
\left(\!\int_{C_j(B_i)}\!\!\left(\sup_{t\geq 2^{j-l-3}r_{B_i}}\!\!\left(\dashint_{B(x,2^{l+2}t)}\!
\left|e^{-t^2L_w}\!\sqrt{L_w}\!\left(B_{r_{B_i}}b_i\right)\!(y)\right|^{p_0}\!dw(y)\!\right)^{\!\!\frac{1}{p_0}}\!\right)^{\!p_1}\!\!v(x)dw(x)\!\right)^{\!\!\!\frac{1}{p_1}}
\\&=:\mathfrak{C}_1+\mathfrak{C}_2.
\end{align*}

In order to estimate $\mathfrak{C}_1$, we use functional calculus as in the proof of Theorem \ref{thm:boundednesswidetildeS}. Recall \eqref{afaerverv} and take $\phi(z,t):=tz^{\frac{1}{2}}e^{-t^2 z}(1-e^{-r_{B_i}^2 z})^M$. Then $\phi(z,t)$ is holomorphic in the open sector $\Sigma_\mu=\{z\in\mathbb{C}\setminus\{0\}:|{\rm arg} (z)|<\mu\}$ and satisfies $|\phi(z,t)|\lesssim |z|^M\,(1+|z|)^{-2M}$ (with implicit constant depending on $\mu$, $t>0$, $r_{B_i}$, and $M$) for every $z\in\Sigma_\mu$. We can check that  for every $z\in \Gamma=\partial\Sigma_{\frac{\pi}{2}-\theta}$,
$$
|\eta(z,t)| \lesssim \frac{tr_{B_i}^{2M}}{(|z|+t^2)^{M+\frac{3}{2}}}.
$$
Now fix $x\in C_{j}(B_i)$, $j\geq 4$, and $0<t<2^{j-l-3}r_{B_i}$, 
then $B(x,2^{l+2}t)\subset 2^{j+2}B_i\setminus 2^{j-1}B_i$. This and Minkowski's integral inequality imply
\begin{align*}
&\left(\dashint_{B(x,2^{l+2}t)}
\left|e^{-t^2L_w}\sqrt{L_w}\left(B_{r_{B_i}}b_i\right)(y)\right|^{p_0}dw(y)\right)^{\frac{1}{p_0}}
\\& \quad
=
\left(\dashint_{B(x,2^{l+2}t)}
\left|\phi(L_w,t)\left(\frac{b_i}{t}\right)(y)\right|^{p_0}dw(y)\right)^{\frac{1}{p_0}}
\\& \quad
\lesssim
\int_{\Gamma}\left(\dashint_{B(x,2^{l+2}t)}
\left|e^{-zL_w}\left(\frac{b_i}{t}\right)(y)\right|^{p_0}dw(y)\right)^{\frac{1}{p_0}}\frac{tr_{B_i}^{2M}}{(|z|+t^2)^{M+\frac{3}{2}}}|dz|
\\& \quad
\lesssim
\int_{\Gamma}\left(\dashint_{B(x,2^{l+2}t)}
\left|\chi_{2^{j+2}B_i\setminus 2^{j-1}B_i}e^{-zL_w}b_i(y)\right|^{p_0}dw(y)\right)^{\frac{1}{p_0}}\frac{r_{B_i}^{2M}}{|z|^{M+\frac{3}{2}}}|dz|
\\& \quad
\lesssim
\int_{\Gamma}\mathcal{M}_{p_0}^{w}\left(\chi_{2^{j+2}B_i\setminus 2^{j-1}B_i}e^{-zL_w}b_i\right)(x)\frac{r_{B_i}^{2M}}{|z|^{M+\frac{3}{2}}}|dz|.
\end{align*}
Recalling that $\mathcal{M}_{p_0}^w$ on $L^{p_1}(vdw)$ since $v\in A_{\frac{p_1}{p_0}}(w)$, and applying again  Minkowski's integral inequality, we get
\begin{align*}
\mathfrak{C}_1
&\lesssim
\int_{\Gamma}
\left(
\int_{C_j(B_i)}
\left|\mathcal{M}_{p_0}^w\left(
\chi_{2^{j+2}B_i\setminus 2^{j-1}B_i}e^{-zL_w}b_i\right)(x)\right|^{p_1}v(x)
dw(x)\right)^{\frac{1}{p_1}}
\frac{r_{B_i}^{2M}}{|z|^{M+\frac{3}{2}}}|dz|
\\
&\lesssim
\int_{\Gamma}\left(\int_{2^{j+2}B_i\setminus 2^{j-1}B_i}\left|e^{-zL_w}b_i(y)\right|^{p_1}v(y)dw(y)\right)^{\frac{1}{p_1}}
\frac{r_{B_i}^{2M}}{|z|^{M+\frac{3}{2}}}|dz|.
\end{align*}
Observe that  $2^{j+2}B_i\setminus 2^{j-1}B_i=\cup_{l=1}^3C_{l+j-2}(B_i)$.
Then by the fact that $e^{-zL_w}\in\mathcal{O}(L^{p_1}(vdw)-L^{p_1}(vdw))$,  \eqref{CZ:PS}, and changing the variable $s$ into $\frac{4^jr_{B_i}^2}{s^2}$, 
\begin{align*}
\mathfrak{C}_1
&\lesssim
vw(2^{j+1}B_i)^{\frac{1}{p_1}}
2^{j\theta_1}
\left(\dashint_{B_i}\left|b_i\right|^{p_1}d(vw)\right)^{\frac{1}{p_1}}
\int_{0}^{\infty}\Upsilon\left(\frac{2^jr_{B_i}}{s^{\frac{1}{2}}}\right)^{\theta_2}
e^{-c\frac{4^jr_{B_i}^2}{s}}
\frac{sr_{B_i}^{2M}}{s^{M+\frac{3}{2}}}\frac{ds}{s}
\\
&\lesssim
\alpha
vw(2^{j+1}B_i)^{\frac{1}{p_1}}
2^{-j(2M+1-\theta_1)}
\int_{0}^{\infty}\Upsilon\left(s\right)^{\theta_2}
e^{-cs^2}
s^{2M+1}
\frac{ds}{s}
\\ &\lesssim
\alpha
vw(2^{j+1}B_i)^{\frac{1}{p_1}}
2^{-j(2M+1-\theta_1)},
\end{align*}
 provided $2M+1>\theta_2$.

We continue by estimating $\mathfrak{C}_2$. To this end, first change the variable $t$ into $t\sqrt{M+1}=:t\theta_M$. Next, for any $x\in C_j(B_i)$ and $t\geq\frac{2^{j-1}r_{B_i}}{2^{l+2}\theta_M}$,  note that $B_i\subset B(x_{B_i},\theta_M 2^{l+2}t)=:B_{i}^l\subset B(x,\theta_M2^{l+2}5t)$ ($x_{B_i}$ denotes the center of $B_i$). Then,
\begin{multline*}
\!\!\!\mathfrak{C}_2
\lesssim\!\!
\Bigg(\!\!\int_{C_j(B_i)}\!\!\Bigg(\sup_{t\geq\frac{2^{j-l-3}r_{B_i}}{\theta_M}}\!\dashint_{B(x,\theta_M 2^{l+2}t)}\!
\abs{\!\mathcal{T}_{t,r_{B_i}}\!\sqrt{L_w}e^{-t^2L_w}\!\br{\!\chi_{B_{i}^l}b_i\!}\!(y)}^{\!p_0}\!dw(y)\!\!\Bigg)^{\!\!\!\frac{p_1}{p_0}}\!\!d(vw)(x)\!\!\Bigg)^{\!\!\!\frac{1}{p_1}}
\\
\lesssim\!\!\!
\Bigg(\!\!\int_{\!C_j(B_i)}\!\!\!\Bigg(\!\sup_{t\geq\frac{2^{j-l-3}r_{B_i}}{\theta_M}}\!\!\!\!\!\!w(B(x,\theta_M 2^{l+2}t))^{\!-1}\!\!\!\!\int_{\!\R^n}\!\!
\abs{\!\mathcal{T}_{t,r_{B_i}}\!\sqrt{L_w}e^{-t^2L_w}\!\br{\!\chi_{B_{i}^l}b_i\!}\!(y)}^{\!p_0}\!\!\!\!dw(y)\!\!\Bigg)^{\!\!\!\!\frac{p_1}{p_0}}\!\!\!\!d(vw)(x)\!\!\Bigg)^{\!\!\!\frac{1}{p_1}}\!\!\!\!,
\end{multline*}
where  $\mathcal{T}_{t,r_{B_i}}:=\left(e^{-t^2L_w}-e^{-(t^2+r_{B_i}^2)L_w}\right)^M$.

In the above setting, \eqref{boundednesstsr}, Proposition \ref{prop:wRR},  the fact that $\sqrt{\tau}\nabla e^{-\tau L_w}\!\!\in\! \mathcal{O}(L^{p_0}(w)-L^{p_0}(w))$, \eqref{doublingcondition}, and 
Lemma \ref{ARHsinpesoconpeso} (a) (see \eqref{p_0choice}), imply
\begin{align}\label{termT}
&\left(\int_{\R^n}
\abs{\mathcal{T}_{t,r_{B_i}}\sqrt{L_w}e^{-t^2L_w}\br{\chi_{B_{i}^l}b_i}}^{p_0}dw\right)^{\frac{1}{p_0}}
\\\nonumber
&\quad
\lesssim
\left(\frac{r_{B_i}^2}{t^2}\right)^{M}\!\!\left(\int_{\R^n}
\left|\nabla e^{-t^2L_w}\left(\chi_{B_{i}^l}b_i\right)\!\right|^{p_0}\!dw\right)^{\!\!\frac{1}{p_0}}
\\\nonumber
&\quad
\lesssim 2^{l}
\sum_{N\geq 1}w(C_N(B_i^l))^{\frac{1}{p_0}}
\left(\frac{r_{B_i}^2}{t^2}\right)^{M}\!\!\left(\dashint_{C_N(B_{i}^l)}
\left|t\nabla e^{-t^2L_w}\left(\chi_{B_{i}^l}\frac{b_i}{r_{B_i}}\right)\!\right|^{p_0}\!dw\right)^{\!\!\frac{1}{p_0}}
\\\nonumber
&\quad
\lesssim 2^{l\theta}w(B_i^l)^{\frac{1}{p_0}}
\sum_{N\geq 1}e^{-c4^N}
\left(\frac{r_{B_i}^2}{t^2}\right)^{M}\left(\dashint_{B_{i}^l}
\left|\frac{b_i}{r_{B_i}}\right|^{p_0}dw\right)^{\frac{1}{p_0}}
\\\nonumber
&\quad
\lesssim 2^{l\theta}w(B_i^l)^{\frac{1}{p_0}}
\left(\frac{r_{B_i}^2}{t^2}\right)^{M}\left(\dashint_{B_i}
\left|\frac{b_i}{r_{B_i}}\right|^{p_1}d(vw)\right)^{\frac{1}{p_1}}
\\\nonumber
&\quad
\lesssim 2^{l\theta}w(B_i^l)^{\frac{1}{p_0}}
\alpha\left(\frac{r_{B_i}^2}{t^2}\right)^{M}
.
\end{align}
Consequently, 
\begin{align*}
\mathfrak{C}_2
\lesssim  2^{l\theta}\alpha
\Bigg(\int_{C_j(B_i)}\Bigg(\sup_{t\geq \frac{2^{j-l-3}r_{B_i}}{\theta_M}}\left(\frac{r_{B_i}^2}{t^2}\right)^{M}\Bigg)^{p_1}vdw\Bigg)^{\frac{1}{p_1}}
\lesssim
\alpha
vw(2^{j+1}B_i)^{\frac{1}{p_1}}2^{-j2M}2^{l(2M+\theta)}
,
\end{align*}
where in the first inequality, we have used that $w(B(x,\theta_M2^{l+2}t))^{-1}w(B_i^l)\leq C$,  since $B_{i}^l\subset B(x,\theta_M2^{l+2}5t)$.

Collecting the estimates obtained for $\mathfrak{C}_1$ and  $\mathfrak{C}_2$, we conclude that, for $M\in \N$ such that $2M+1>\theta_2$,
\begin{align}\label{term:I_1^ij}
I_{1,l}^{ij}\lesssim \alpha\, vw(2^{j+1}B_i)^{\frac{1}{p_1}}  2^{-j(2M-\theta_1)} 2^{l(2M+\theta)}.
\end{align}

\medskip

Next, let us  estimate term $I_{2,l}^{ij}$. Splitting the supremum in $t$, we have
\begin{align*}
I_{2,l}^{ij}&
\lesssim 
\left(
\int_{C_j(B_i)}
\sup_{0<t<2^{j-l-2}r_{B_i}}\!\!\left(\dashint_{B(x,2^{l+1}t)}\!\!
\left(\mathrm{g}_{\hh}^w\sqrt{L_w}\left(B_{r_{B_i}}b_i\right)\!\!(y)\right)^{p_0}dw(y)\right)^{\!\!\frac{p_1}{p_0}}v(x)dw(x)
\right)^{\!\!\frac{1}{p_1}}
\\
&\,\,+
\left(
\int_{C_j(B_i)}
\sup_{t\geq 2^{j-l-2}r_{B_i}}\!\!\left(\dashint_{B(x,2^{l+1}t)}
\!\!\left(\mathrm{g}_{\hh,t}^{w}\sqrt{L_w}\left(B_{r_{B_i}}b_i\right)(y)\right)^{p_0}dw(y)\right)^{\!\!\frac{p_1}{p_0}}v(x)dw(x)\!
\right)^{\!\!\frac{1}{p_1}}
\\&=:D_1^{ij}+D_2^{ij}.
\end{align*}

Regarding $D_{1}^{ij}$, we claim that 
\begin{align}\label{D1ijfinal}
D_{1}^{ij}
\lesssim
\alpha vw(2^{j+1}B_i)^{\frac{1}{p_1}}2^{-j(2M+1-\widetilde\theta_1)}.
\end{align}
To  this end, first note that for $0<t<2^{j-l-2}r_{B_i}$ and $x\in C_j(B_i)$, then $B(x,2^{l+1}t)\subset 2^{j+2}B_i\setminus 2^{j-1}B_i$. Next recall that  $\mathcal{M}_{p_0}^w$ is $L^{p_1}(vdw)$ bounded since $v\in A_{\frac{p_1}{p_0}}(w)$ (see \eqref{p_0choice}). Hence
\begin{multline*}
D_1^{ij}
\lesssim
\left(\int_{C_j(B_i)} \left| \mathcal{M}_{p_0}^w \left(\chi_{2^{j+2}B_i\setminus 2^{j-1}B_i}
\mathrm{g}_{\hh}^w\sqrt{L_w}\left(B_{r_{B_i}}b_i\right)\right)\right|^{p_1}vdw \right)^{\frac{1}{p_1}}
\\ \lesssim 
vw(2^{j+1}B_i)^{\frac{1}{p_1}} \left(
\dashint_{2^{j+2}B_i\setminus 2^{j-1}B_i}
\left|
\mathrm{g}_{\hh}^w\sqrt{L_w}\left(B_{r_{B_i}}b_i\right)
\right|^{p_1}d(vw)
\right)^{\frac{1}{p_1}}.
\end{multline*}
 In view of \eqref{p_0choice}, we apply Lemma \ref{ARHsinpesoconpeso} (b) and Minkowski's integral inequality to get
\begin{align}\label{D1ij}
D_1^{ij}
\lesssim
vw(2^{j+1}\!B_i)^{\!\frac{1}{p_1}}
\!\!
\left(\int_0^{\infty}\!\!\!\left(
\dashint_{2^{j+2}B_i\setminus 2^{j-1}B_i}\!
\left|
rL_wr\sqrt{L_w}e^{-r^2L_w}\!\left(B_{r_{B_i}}b_i(x)\right)\!
\right|^{\!q_0}\!\!dw(x)\!\!\right)^{\!\!\!\frac{2}{q_0}}\!\frac{dr}{r}
\right)^{\!\!\!\frac{1}{2}}\!\!.
\end{align}
In order to estimate the integral 
in $x$, we use functional calculus  as in the estimate of $\mathfrak{C}_1$. Apply the fact that $zL_we^{-zL_w}\in \mathcal{O}(L^{p_0}(w)-L^{q_0}(w))$,  Lemma \ref{ARHsinpesoconpeso} (a), and \eqref{CZ:PS}. 
Thus,
\begin{align*}
&\left(\dashint_{2^{j+2}B_i\setminus 2^{j-1}B_i}\right.\left.\left|rL_wr\sqrt{L_w}e^{-r^2L_w}\left(B_{r_{B_i}}b_i\right)\right|^{q_0}dw\right)^{\frac{1}{q_0}}
\\&\qquad
\lesssim
\int_{\Gamma}\left(
\dashint_{2^{j+2}B_i\setminus 2^{j-1}B_i}
|zL_w e^{-zL_w} b_i|^{q_0}dw\right)^{\frac{1}{q_0}}
\frac{r^2r_{B_i}^{2M}}{(|z|+r^2)^{M+\frac{3}{2}}}\frac{|dz|}{|z|}
\\&\qquad
\lesssim
2^{j\widetilde\theta_1}
\int_{0}^{\infty}\Upsilon(s)^{\wt\theta_2}
e^{-cs^2}
\frac{r^2r_{B_i}^{2M}}{(4^jr_{B_i}^2/s^2+r^2)^{M+\frac{3}{2}}}\frac{ds}{s}\left(\dashint_{B_i}|b_i|^{p_0}dw\right)^{\frac{1}{p_0}}
\\&\qquad
\lesssim
2^{j\widetilde\theta_1}
\int_{0}^{\infty}\Upsilon(s)^{\wt\theta_2}
e^{-cs^2}
r^2\frac{r_{B_i}^{2M}}{(4^jr_{B_i}^2/s^2+r^2)^{M+\frac{3}{2}}}\frac{ds}{s}\left(\dashint_{B_i}|b_i|^{p_1}d(vw)\right)^{\frac{1}{p_1}}
\\&\qquad
\lesssim
\alpha r_{B_i}
2^{j\widetilde\theta_1}
\int_{0}^{\infty}\Upsilon\left(s\right)^{\widetilde\theta_2}
e^{-cs^2}
r^2\frac{r_{B_i}^{2M}}{(4^jr_{B_i}^2/s^2+r^2)^{M+\frac{3}{2}}}\frac{ds}{s}.
\end{align*} 
Plugging this into \eqref{D1ij} and
 changing the variable $r$ into $2^jr_{B_i}r$, we obtain, for $M\in \N$ such that $2M>\widetilde\theta_2$,
\begin{align*}
D_{1}^{ij}
&
\lesssim \alpha r_{B_i}2^{j\widetilde\theta_1}vw(2^{j+1}B_i)^{\frac{1}{p_1}}\!
\left(\!\int_{0}^{\infty}\!\!\left(\int_{0}^{\infty}\!\Upsilon\left(s\right)^{\widetilde\theta_2}
e^{-cs^2}
r^2\frac{r_{B_i}^{2M}}{(4^jr_{B_i}^2/s^2+r^2)^{M+\frac{3}{2}}}\frac{ds}{s}\!\right)^{\!\!\!2}\!\!\frac{dr}{r}\!\right)^{\!\!\!\frac{1}{2}}
\\\nonumber&
\lesssim
\alpha vw(2^{j+1}B_i)^{\frac{1}{p_1}}
2^{-j(2M+1-\widetilde\theta_1)}\!
\left(\!\int_{0}^{\infty}\!\!r^{4}\!\!\left(\int_{0}^{\infty}\!\Upsilon\left(s\right)^{\widetilde\theta_2}
e^{-cs^2}\!
\frac{1}{(1/s^2+r^2)^{M+\frac{3}{2}}}\frac{ds}{s}\!\right)^{\!\!\!2}\!\!\frac{dr}{r}\!\!\right)^{\!\!\!\frac{1}{2}}
\\\nonumber&
\lesssim
\alpha vw(2^{j+1}B_i)^{\frac{1}{p_1}}
2^{-j(2M+1-\widetilde\theta_1)}\left(\left(
\int_{0}^{1}r^{4}\frac{dr}{r}\right)^{\frac{1}{2}}\int_{0}^{\infty}\Upsilon\left(s\right)^{\widetilde\theta_2}
e^{-cs^2}
s^{2M+3}\frac{ds}{s}\right.
\\\nonumber&\hspace*{6cm}
+
\left.
\left(\int_{1}^{\infty}r^{-2}\frac{dr}{r}\right)^{\frac{1}{2}}\int_{0}^{\infty}\Upsilon\left(s\right)^{\widetilde\theta_2}
e^{-cs^2}
s^{2M}\frac{ds}{s}\right)
\\\nonumber&
\lesssim
\alpha vw(2^{j+1}B_i)^{\frac{1}{p_1}}
2^{-j(2M+1-\widetilde\theta_1)}.
\end{align*}

Now turning to the estimate of $D_2^{ij}$, we claim
\begin{align}\label{D2ij}
D_2^{ij}
\lesssim 
2^{l\left(2M+\widetilde{\theta}\right)}\alpha vw(2^{j+1}B_i)^{\frac{1}{p_1}} 2^{-2jM}.
\end{align}
For any $t\geq 2^{j-l-2}r_{B_i}$ and $f\in L^2(w)$, we have that
$$
\mathrm{g}_{\hh,t}^wf(x)=\left(\int_{\frac{t}{2}}^{\infty}|r^2L_we^{-r^2L_w}f(x)|^2\frac{dr}{r}\right)^{\frac{1}{2}}\leq 
\left(\int_{2^{j-l-3}r_{B_i}}^{\infty}|r^2L_we^{-r^2L_w}f(x)|^2\frac{dr}{r}\right)^{\frac{1}{2}}.
$$
Moreover, recall that $p_0<q_0$ (see \eqref{p_0choice}), this implies the boundedness of the maximal operator $\Mcal_{p_0}^w$ on $L^{q_0}(w)$.
This,
together with Lemma \ref{ARHsinpesoconpeso} (b) and Minkowski's integral inequality, allows us to obtain
\begin{align}\label{D2ijplug}
&D_2^{ij}
\lesssim
vw(2^{j+1}B_i)^{\frac{1}{p_1}}
\left(
\dashint_{C_j(B_i)}
\left|\mathcal{M}^w_{p_0}
\br{\mathrm{g}_{\hh,2^{j-l-2}r_{B_i}}^{w}
\br{\sqrt{L_w}\left(B_{r_{B_i}}b_i\right)}}\right|^{q_0}dw
\right)^{\frac{1}{q_0}}
\\\nonumber
&
\lesssim\!
vw(2^{j+1}\!B_i)^{\!\frac{1}{p_1}}
w(2^{j+1}\!B_i)^{\!-\frac{1}{q_0}}
\!\!\left(\!
\int_{\R^n}\!\!\left(
\!\int_{2^{j-l-3}r_{B_i}}^{\infty}
\!\!\left|
r^2L_we^{-r^2L_w}\!\!
\left(
\sqrt{L_w}\left(
B_{r_{B_i}}b_i
\right)\!
\right)\!
\right|^{\!2}\!\frac{dr}{r}\!\right)^{\!\!\!\frac{q_0}{2}}\!\!\!dw\!\right)^{\!\!\!\frac{1}{q_0}}\!\!
\\\nonumber
&
\lesssim\!
vw(2^{j+1}\!B_i)^{\!\frac{1}{p_1}}
w(2^{j+1}\!B_i)^{\!-\frac{1}{q_0}}\!
\!\left(\!
\int_{\frac{2^{j-l-3}r_{B_i}}{\theta_M}}^{\infty}\!\!\left(\int_{\R^n}\!
\left|\mathcal{T}_{r,r_{B_i}}\!\sqrt{L_w}r^2L_we^{-r^2L_w}\!\left(\chi_{B_{i}^l} b_i\right)\!
\right|^{\!q_0}\!\!dw\!\right)^{\!\!\!\frac{2}{q_0}}\!\!\frac{dr}{r}\!\right)^{\!\!\!\frac{1}{2}}\!\!
,
\end{align}
where in the last inequality we have changed the variable $r$ into $r\theta_M:=r\sqrt{M+1}$, used that  $B_i\subset B(x_{B_i}, \theta_M 2^{l+1}r)=:B_{i}^l$, for $r>\frac{2^{j-l-3}r_{B_i}}{\theta_M}$ and $j\geq 4$ ($x_{B_i}$ denotes the center of $B_i$), and we recall that  
$\mathcal{T}_{r,r_{B_i}}:=(e^{-r^2L_w}-e^{-(r^2+r_{B_i}^2)L_w})^M$. 

Proceeding as in the estimate of \eqref{termT}, but
using  now the fact  that $\sqrt{\tau}\nabla \tau\! L_we^{-\tau L_w}\!\in\mathcal{O}(L^{p_0}(w)-L^{q_0}(w))$ instead of
 $\sqrt{\tau}\nabla e^{-\tau L_w}\in\mathcal{O}(L^{p_0}(w)-L^{p_0}(w))$, we get
\begin{multline*}
\br{\int_{\R^n}\abs{\mathcal{T}_{r,r_{B_i}}\sqrt{L_w}r^2L_we^{-r^2L_w}\br{\chi_{B_{i}^l}b_i}}^{q_0}dw}^{\frac{1}{q_0}}
\lesssim 2^{l\widetilde{\theta}}\alpha w(B_i^l)^{\frac{1}{q_0}}
\left(\frac{r_{B_i}^2}{r^2}\right)^{M}
\\
\lesssim 2^{l\left(\widetilde{\theta}+\frac{2n}{q_0}\right)}\alpha w(2^{j+1}B_i)^{\frac{1}{q_0}}2^{-\frac{2jn}{q_0}}
\left(\frac{r_{B_i}^2}{r^2}\right)^{M-\frac{n}{q_0}},
\end{multline*}
where in the last inequality we have used that  
for $r>\frac{2^{j-l-3}r_{B_i}}{\theta_M}$ and $j\geq 4$,  $2^{j+1}B_i\subset 2^{3}B_{i}^l$, and \eqref{pesosineqw:Ap}.
Plugging this into \eqref{D2ijplug} leads  to
\begin{align*}
D_2^{ij}
& \lesssim 
2^{l\left(\widetilde{\theta}+\frac{2n}{q_0}\right)}\alpha 2^{-\frac{2jn}{q_0}}
vw(2^{j+1}\!B_i)^{\frac{1}{p_1}}
\left(
\int_{\frac{2^{j-l-3}r_{B_i}}{\theta_M}}^{\infty}
\left(\frac{r_{B_i}^2}{r^2}\right)^{2M-\frac{2n}{q_0}}
\frac{dr}{r}\!\right)^{\!\!\frac{1}{2}}
\\ &\lesssim 
2^{l\left(2M+\widetilde{\theta}\right)}\alpha vw(2^{j+1}B_i)^{\frac{1}{p_1}} 2^{-2jM},
\end{align*}
 provided $2M>\frac{2n}{q_0}$.

Gather  \eqref{D1ijfinal} and \eqref{D2ij}, then for $M\in \N$ such that $2M>\max\{\widetilde\theta_2,2n/q_0\}$,
\begin{align*}
I_{2,l}^{ij}\lesssim 2^{l(2M+\widetilde{\theta})}\alpha vw(2^{j+1}B_i)^{\frac{1}{p_1}}2^{-j(2M-\widetilde\theta_1)}.
\end{align*}
This and  \eqref{term:I_1^ij} yield, for  $2M>\max\{\widetilde\theta_2,2n/q_0,\theta_2-1\}$, 
$$
I_{m,l}^{ij}\leq C_1\alpha vw(2^{j+1}B_i)^{\frac{1}{p_1}}2^{-j(2M-C_2)}, \quad m=1,2,
$$
with $C_2:=\max\{\theta_1,\widetilde{\theta}_1\}$ and $C_1:=C 2^{lC_M}$. 
Then, in view of  \eqref{sum:termII_2}, applying Lemma \ref{lemma:lastestimate} with $\mathcal{I}_{ij}=I_{m,l}^{ij}$ and $\{B_i\}_{i}$ the collection of balls given by Lemma \ref{lem:CZweighted}, and \eqref{CZ:sum}, for $2M\!\!>\!\max\!\big\{\!C_2\!+\!nr_{w}\mathfrak{r}_v(w),\widetilde\theta_2,\frac{2n}{q_0},\theta_2\!-\!1\!\big\}$\!, we get 
\begin{align*}
III_{m,l}
\lesssim e^{-c4^l}
vw\Big(\bigcup_{i}B_i\Big)
\lesssim e^{-c4^l}
\frac{1}{\alpha^{p}}\int_{\R^n}|\nabla f|^{p}vdw, \quad m=1,2.
\end{align*}
Therefore, by  \eqref{termII_2}
\begin{align*}
III_{m}
\lesssim \sum_{l\geq 1} e^{-c4^l}
\frac{1}{\alpha^{p}}\int_{\R^n}|\nabla f|^{p}vdw
\lesssim
\frac{1}{\alpha^{p}}\int_{\R^n}|\nabla f|^{p}vdw.
\end{align*}
Collecting this estimate and \eqref{maintimederivative}-\eqref{termIIItimederivative}, the proof is complete.
\end{proof}

\section{The regularity problem in unweighted Lebesgue spaces}\label{section:unweighted}

Our main result, Theorem \ref{thm:maindegenerate}, establishes the solvability of the regularity problem in $L^p(v\,dw)$ of the block operator $\L_w$. Recall that $w\in A_2(dx)$ is fixed and controls the degeneracy of the operator and that $v\in A_\infty(w)$. This means that we can establish the solvability of the regularity problem in unweighted Lebesgue spaces by taking $v=w^{-1}$. In this section our goal is to explore this idea and study ranges for which we can solve the regularity problem in terms of the weight $w$. A particular case of interest, where we can be more explicit, is that of power weights.

To start fix $w\in A_2(dx)$ and recall the definitions of $r_w$ and $s_w$ in  \eqref{eq:defi:rw}. As just mentioned, we let $v=w^{-1}$ and observe that from the definitions it is clear that for every $1\le r<\infty$ one has $w^{-1}\in A_r(w)$ if and only if $w\in RH_{r'}(dx)$, and $w^{-1}\in RH_{r'}(w)$ if and only if $w\in A_r(dx)$. Hence, according to \eqref{eq:defi:rvw} we have
$\mathfrak{r}_{w^{-1}}(w)=s_w$ and $\mathfrak{s}_{w^{-1}}(w)=r_w$. Then looking at Theorem \ref{thm:maindegenerate} and using \eqref{intervalrsw}, we see that \eqref{main:compa} is equivalent to 
\begin{equation}\label{cond1:apps}
\max\{r_w,q_-(L_w)\}\, s_w< \frac{q_+(L_w)}{r_w},
\end{equation}
and if that holds we have  $(R^{\L_w})_{L^p(dx)}$ solvability for $p$ so that
\begin{equation}\label{cond2:apps}
\max\Big\{r_w,\frac{n\,r_w\, q_-(L_w)}{n\,r_w+q_-(L_w)}\Big\}\, s_w< p<\frac{q_+(L_w)}{r_w}.
\end{equation}
It is important to note that $q_-(L_w)$ and $q_+(L_w)$ are defined in an abstract way and depend intrinsically on $w$. From \cite[Propositions 3.1 and 7.1]{CMR15} and recalling that $n\ge 2$,  we know that $q_-(L_w)=p_-(L_w)\le \frac{2\,n\,r_w}{n\,r_w+2}$, hence we have an estimate for $q_-(L_w)$ in terms of $n$ and $r_w$. On the other hand $q_+(L_w)>2$ and can be arbitrarily close to $2$ (even in the case $w\equiv 1$), and we do not have an explicit bound in terms of $w$ (see \cite[Proof of Theorem 11.8]{CMR15} in this regard). Taking this into account and in order to check that \eqref{cond1:apps} holds we will replace its right-hand side with $\frac{2}{r_w}$.

Our first result for general weights is as follows:

\begin{corollary}\label{corol:weights}
	Let $w\in A_2(dx)$ and let $\L_w$ be a block  degenerate elliptic operator in $\R_+^{n+1}$ as in \eqref{def:L-block}. Associated with $\L_w$ consider 
	the regularity problem $(R^{\L_w})_{L^p(dx)}$ as in Section \ref{section:intro}. Given $f\in C_c^\infty(\re^n)$ if one sets $u(x,t)=e^{-t\sqrt L_w}f(x)$, $(x,t)\in\R_+^{n+1}$, then 
	\begin{equation}\label{corol:main-est}
	\norm{\Ncal_w (\nabla_{x,t} u)}_{L^p(dx)} \leq C\norm{\nabla f}_{L^p(dx)}
	\end{equation}
	in any of the following scenarios:
	\begin{list}{$(\theenumi)$}{\usecounter{enumi}\leftmargin=1cm \labelwidth=1cm\itemsep=0.2cm\topsep=.0cm \renewcommand{\theenumi}{\alph{enumi}}}
		
		\item If $w\in A_1(dx)\cap RH_{1+\frac{n}2}(dx)$ and 
		\[
		\max\Big\{1,\frac{n\, q_-(L_w)}{n+q_-(L_w)}\Big\}\, s_w< p<q_+(L_w),
		\]
		in particular, in the range 
		\[
		\max\Big\{1,\frac{2\,n}{n+4}\Big\}\, s_w< p\le 2.
		\]
		
		\item If $w\in A_{r_0}(dx)\cap RH_{\infty}(dx)$ with $r_0:=\min\Big\{\sqrt{2}, \frac{1+\sqrt{1+\frac8n}}{2}\Big\}$ and
		\[
		\max\Big\{r_w,\frac{n\,r_w\, q_-(L_w)}{n\,r_w+q_-(L_w)}\Big\}< p<\frac{q_+(L_w)}{r_w},
		\]
		in particular, in the range 
		\[
		\max\Big\{r_w,\frac{2\,n\,r_w}{n\,r_w+4}\Big\}< p\le \frac{2}{r_w}.
		\]
		
		\item If $w\in A_r(dx)\cap RH_{s(r)'}(dx)$ with $1<r< r_0$ and $s(r)=\min\{\frac2{r^2}, \frac{n\,r+2}{n\,r^2}\}$, and 
		\[
		\max\Big\{r_w,\frac{n\,r_w\, q_-(L_w)}{n\,r_w+q_-(L_w)}\Big\}\, s_w< p<\frac{q_+(L_w)}{r_w},
		\]
		in particular, in the range
		
		\[
		\max\Big\{r_w,\frac{2\,n\,r_w}{n\,r_w+4}\Big\}\, s_w< p\le \frac{2}{r_w}.
		\]
		
		\item Given $\Theta\ge 1$ there exists $\epsilon_0=\epsilon_0(\Theta,n,\Lambda/\lambda)\in (0,\frac1{2n}]$, such that for every $w\in A_{1+\epsilon}(dx)\cap  RH_{\max\{\frac2{1-\epsilon}, 1+(1+\epsilon)\,\frac{n}2\}}(dx)$ with  $0\le\epsilon<\epsilon_0$ and $[w]_{A_2(dx)}\le \Theta$, then \eqref{corol:main-est} holds with $p=2$, or equivalently $(R^{\L_w})_{L^2(dx)}$ is solvable.
			\end{list}
	\end{corollary}

\begin{proof}

We first consider $(a)$. Let $w\in A_1(dx)\cap RH_{1+\frac{n}2}(dx)$ then $r_w=1$ and $s_w<(1+\frac{n}2)'=1+\frac2{n}$. Using that $q_-(L_w)\le \frac{2\,n}{n+2}$ (since $n\ge 2$) we have
\[
\max\{r_w,q_-(L_w)\}\, s_w
<
\max\Big\{1,\frac{2\,n}{n+2}\Big\}\, \Big(1+\frac2{n}\Big)
=
2
<
q_+(L_w)
=
\frac{q_+(L_w)}{r_w}.
\]
That is, \eqref{cond1:apps} holds and according to \eqref{cond2:apps} we have $(R^{\L_w})_{L^p(dx)}$-solvability for $p$ so that
\[
\max\Big\{1,\frac{n\, q_-(L_w)}{n+q_-(L_w)}\Big\}\, s_w< p<q_+(L_w)
\]
and, in particular, in the range 
\[
\max\Big\{1,\frac{2\,n}{n+4}\Big\}\, s_w< p\le 2.
\]

To prove $(b)$ and $(c)$ assume that $w\in A_r(dx)\cap RH_{s(r)'}(dx)$ with $1<r\le \min\Big\{\sqrt{2}, \frac{1+\sqrt{1+\frac8n}}{2}\Big\}$ and $s(r)=\min\{\frac2{r^2}, \frac{n\,r+2}{n\,r^2}\}$ and note that the restriction on $r$ gives $s(r)\in [1,\infty)$. In particular, $r_w<r$, $s_w\le s(r)$, and
\begin{multline*}
r_w\,\max\{r_w,q_-(L_w)\}\, s_w
\le
r_w\,\max\Big\{r_w,\frac{2\,n\,r_w}{n\,r_w+2} \Big\}\, s_w
\\
<
\max\Big\{r^2,\frac{2\,n\,r^2}{n\,r+2} \Big\}\, s(r)= 2
< q_+(L_w).
\end{multline*}
This implies \eqref{cond1:apps} and we have $(R^{\L_w})_{L^p(dx)}$-solvability for $p$ in the range given by \eqref{cond2:apps}, and in particular for those $p$'s satisfying
\[
\max\Big\{r_w,\frac{2\,n\,r_w}{n\,r_w+4}\Big\}\, s_w< p\le \frac{2}{r_w}.
\]
All these show $(b)$ by taking $r=r_0$ so that $s(r)=1$ and hence $s_w=1$. Also $(c)$ follows from the case 
$1<r<r_0$.

To deal with $(d)$ we proceed as in \cite[pp.~654--655]{CMR15}. There, it is shown that given $\Theta\ge 1$ there exists $\epsilon_0=\epsilon_0(\Theta,n,\Lambda/\lambda)\in (0,\frac1{2n}]$ such that if $w\in A_{1+\epsilon}(dx)$ with $0\le\epsilon<\epsilon_0$ so that $[w]_{A_2(dx)}\le \Theta$ then $2\,r_w<q_+(L_w)$. That is, $2< \frac{q_+(L_w)}{r_w}$. On the other hand, if we additionally assume that $w\in RH_{\max\{\frac2{1-\epsilon}, 1+(1+\epsilon)\,\frac{n}2\}}(dx)$ then
\begin{multline*}
s_w
<\Big(\max\Big\{\frac2{1-\epsilon}, 1+(1+\epsilon)\,\frac{n}2\Big\}\Big)'
=
\min\Big\{\frac2{1+\epsilon}, 1+\frac2{n\,(1+\epsilon)}\Big\}
\\
<
\min\Big\{\frac2{r_w}, 1+\frac2{n\,r_w}\Big\}
=
\frac2{\max\Big\{r_w,\frac{2\,n\,r_w}{n\,r_w+2}\Big\}},
\end{multline*}
that is, 
\[
\max\Big\{r_w,\frac{2\,n\,r_w}{n\,r_w+2} \Big\}\, s_w<2.
\]
Altogether we have obtained that $\max\Big\{r_w,\frac{2\,n\,r_w}{n\,r_w+2} \Big\}\, s_w<2< \frac{q_+(L_w)}{r_w}$. This implies \eqref{cond1:apps} and also that $p=2$ satisfies \eqref{cond2:apps}. Consequently,  $(R^{\L_w})_{L^2(dx)}$ is solvable as desired. 
\end{proof}

Concerning power weights  we have the following result:

\begin{corollary}\label{corol:weights:power}
	Consider the power weight $w_\beta(x)=|x|^{n\,(\beta-1)}$ with $0<\beta<2$, and let $\L_{w_\beta}$ be the associated block operator
	\begin{equation}\label{def:L-block:power}
	\L_{w_\beta} u(x,t)
	=
	-|x|^{-n\,(\beta-1)} \divv_x \big(|x|^{n\,(\beta-1)}\,A(x) \nabla_x u(x,t)\big)-\partial_{t}^2 u(x,t)
	\end{equation}
	where $A$ is an $n\times n$ matrix of complex $L^\infty$-valued coefficients defined on $\R^n$, $n\geq 2$ satisfying the uniform ellipticity condition 
	\eqref{eq:elliptic-intro}. 
	
	Assume that
	\[
	\frac{n}{n+2} \le \beta \le \min\Big\{\sqrt{2}, \tfrac{1+\sqrt{1+\frac8n}}{2}\Big\}
	\]
	then for every $f\in C_c^\infty(\re^n)$ if one sets $u(x,t)=e^{-t\sqrt L_\beta}f(x)$, $(x,t)\in\R_+^{n+1}$, then 
	\begin{equation}\label{corol:main-est:power}
	\norm{\Ncal_{w_\beta} (\nabla_{x,t} u)}_{L^p(dx)} \leq C\norm{\nabla f}_{L^p(dx)},
	\end{equation}
	for every $p$ satisfying
	\[
	\max\Big\{1,\beta,\frac{n\, q_-(L_w)}{n+q_-(L_w)},\frac{n\,\beta\, q_-(L_w)}{n\,\beta+q_-(L_w)}\Big\}\,\max\{1,\beta^{-1}\}< p<\frac{q_+(L_w)}{\max\{1,\beta\}}.
	\]
In particular, in the non-empty range
	\[
	\max\Big\{1, \beta,\frac{2\,n}{n+4},\frac{2\,n\,\beta}{n\,\beta+4}\Big\}\,\max\{1,\beta^{-1}\}< p\le \frac{2}{\max\{1,\beta\}}
	\] 
	provided 
	\[
	\frac{n}{n+2} < \beta < \min\Big\{\sqrt{2}, \tfrac{1+\sqrt{1+\frac8n}}{2}\Big\}.
	\]

Moreover, there exists $\epsilon_1=\epsilon_1(n,\Lambda/\lambda)\in (0,\frac1{2\,n})$ such that if
\[
\frac{n}{n+2} < \beta <1+\epsilon_1 
\]
then \eqref{corol:main-est:power} holds with $p=2$, or equivalently $(R^{\L_{w_\beta}})_{L^2(dx)}$ is solvable.
\end{corollary}

\begin{proof}
Write  $w_\beta(x)=|x|^{n\,(\beta-1)}$ with $0<\beta<2$ so that $w_\beta\in A_2(dx)$. It is not difficult to see that 
\[
r_{w_\beta}=\max\{1,\beta\}
\qquad\text{and}\qquad
s_{w_\beta}=\max\{1, \beta^{-1}\}.
\]

Consider first the case $0< \beta\le 1$ so that $w_\beta\in A_1(dx)$, $r_{w_\beta}=1$, and $s_{w_\beta}=\beta^{-1}$. If $\beta\ge \frac{n}{n+2}$ then
\[
\max\{r_{w_\beta},q_-(L_{w_\beta})\}\, s_w
\le 
\frac{2\,n}{n+2}\,\frac1{\beta}
\le 
2
<q_+(L_{w_\beta})
= \frac{q_+(L_{w_\beta})}{r_{w_\beta}}.
\]
Thus \eqref{cond1:apps} holds and if $\frac{n}{n+2} \le \beta \le 1$ we have $(R^{\L_{w\beta}})_{L^p(dx)}$-solvability for $p$ such that
\[
\max\Big\{1,\frac{n\, q_-(L_{w_\beta})}{n+q_-(L_{w_\beta})}\Big\}\, \frac1{\beta}< p<q_+(L_{w_\beta}).
\]
In particular, if  $\frac{n}{n+2} < \beta \le 1$, the solvability holds in the range $\max\{1,\frac{2\,n}{n+4}\}\, \beta^{-1}< p\le 2$.

Let us treat the case $1<\beta<2$, so that we have $r_{w_\beta}=\beta$ and $s_{w_\beta}=1$. If $1<\beta \le \min\{\sqrt{2}, \frac{1+\sqrt{1+\frac8n}}{2}\}$
then 
\[
r_{w_\beta}\,\max\{r_{w_\beta},q_-(L_{w_\beta})\}
\le
\max\Big\{\beta^2, \frac{2\,n\,\beta^2}{n\,\beta+2}\Big\}
\le
2
<q_+(L_{w_\beta}).
\]
This implies that \eqref{cond1:apps} holds and, as a consequence, \eqref{cond2:apps} yields that if  $1<\beta \le \min\Big\{\sqrt{2}, \frac{1+\sqrt{1+\frac8n}}{2}\Big\}$ then $(R^{\L_{w\beta}})_{L^p(dx)}$ is solvable 
in the range
\[
\max\Big\{\beta,\frac{n\,\beta\, q_-(L_w)}{n\,\beta+q_-(L_w)}\Big\}< p<\frac{q_+(L_w)}{\beta}.
\]
In particular, if  $1<\beta < \min\Big\{\sqrt{2}, \frac{1+\sqrt{1+\frac8n}}{2}\Big\}$ one can solve $(R^{\L_{w\beta}})_{L^p(dx)}$  for $p$ satisfying 
\[
\max\Big\{\beta,\frac{2\,n\,\beta}{n\,\beta+4}\Big\}< p\le \frac{2}{\beta}.
\]

Let us finally focus on the $(R^{\L_{w_\beta}})_{L^2}$-solvability. Consider first the case when $\frac{n}{n+2} < \beta\le 1$ then
\[
\max\Big\{1, \beta,\frac{2\,n}{n+4},\frac{2\,n\,\beta}{n\,\beta+4}\Big\}\,\max\{1,\beta^{-1}\}
=
\max\Big\{1, \frac{2\,n}{n+4}\Big\}\,\frac1\beta
<
2
=
\frac{2}{\max\{1,\beta\}}.
\]
Hence what we have proved so far gives the  $(R^{\L_{w_\beta}})_{L^2(dx)}$-solvability. To consider the case $\beta>1$ we first assume that $\beta<\frac{2\,n+1}{2\,n}$ so that  $w_\beta\in A_{1+\frac1{2\,n}}(dx)$. Note that one can easily see that there exists $\Theta\ge 1$ depending just on $n$ (and independent of $\beta$) such that $[w_\beta]_{A_2(dx)}\le \Theta$. Next we repeat the argument given in the  proof of Corollary \ref{corol:weights} to find the corresponding $\epsilon_0\in (0,\frac1{2\,n}]$, which depends only on $n$ and $\Lambda/\lambda$. Set $\epsilon_1=\epsilon_0$ and assume that $1<\beta<1+\epsilon_1\le \frac{2\,n+1}{2\,n}$. Pick $\epsilon'>0$ so that $1<\beta<1+\epsilon'<1+\epsilon_1$. Hence $w_{\beta}\in A_{1+\epsilon'}(dx)$ with $0<\epsilon'<\epsilon_1=\epsilon_0$ and we can invoke $(d)$ in Corollary \ref{corol:main-est} to conclude the $(R^{\L_{w_\beta}})_{L^2(dx)}$-solvability.
\end{proof}

\begin{proof}[Proof of Corollary \ref{corol:weighted:intro}]
It suffices  to observe that the first part is just item $(d)$ in Corollary \ref{corol:weights}. Regarding power weights, setting $\alpha=-n(\beta-1)$ and with a slight abuse of notation the desired estimate follows at once from Corollary \ref{corol:weights:power}.
\end{proof}

%
%

\addtocontents{toc}{\protect\setcounter{tocdepth}{1}}


\bibliographystyle{acm}

\begin{thebibliography}{99}
\parskip=4pt	

\bibitem{Au07}
{\sc Auscher, P.}
\newblock On necessary and sufficient conditions for {$L^p$}-estimates of
  {R}iesz transforms associated to elliptic operators on {$\mathbb{R}^n$} and
  related estimates.
\newblock {\em Mem. Amer. Math. Soc. 186}, 871 (2007), xviii+75.

\bibitem{ChanglesAuscher}
{\sc  Auscher, P.}  Change of Angles in tent spaces,   \newblock {\em C. R. Math. Acad. Sci. Paris.}, 349(5-6):297--301, 2011. 

\bibitem{AC05}
{\sc Auscher, P., and Coulhon, T.}
\newblock Riesz transform on manifolds and {P}oincar{\'e} inequalities.
\newblock {\em Ann. Sc. Norm. Super. Pisa Cl. Sci. (5)}, 4(3):531--555, 2005.

\bibitem{AHLMT02}
{\sc Auscher, P., Hofmann, S., Lacey, M., McIntosh, A., and Tchamitchian, P.}
\newblock The solution of the {K}ato square root problem for second order
  elliptic operators on {${\R}^n$}.
\newblock {\em Ann. of Math. (2)}, 156(2):633--654, 2002.

\bibitem{AHM}
{\sc Auscher, P., Hofmann, S., Martell, J.M.} 
{Vertical versus conical square functions}.  
{\em Trans. Amer. Math. Soc.} 364(10):5469--5489, 2012. 



\bibitem{AMIII06}
{\sc Auscher, P., and Martell, J.~M.}
\newblock Weighted norm inequalities, off-diagonal estimates and elliptic
  operators. {III}. {H}armonic analysis of elliptic operators.
\newblock {\em J. Funct. Anal.} 241(2):703--746, 2006.

\bibitem{AMI07}
{\sc Auscher, P., and Martell, J.~M.}
\newblock Weighted norm inequalities, off-diagonal estimates and elliptic
  operators. {I}. {G}eneral operator theory and weights.
\newblock {\em Adv. Math.}, 212(1):225--276, 2007.

\bibitem{AMII07}
{\sc Auscher, P., and Martell, J.~M.}
\newblock Weighted norm inequalities, off-diagonal estimates and elliptic
  operators. {II}. {O}ff-diagonal estimates on spaces of homogeneous type.
\newblock {\em J. Evol. Equ.}, 7(2):265--316, 2007.

\bibitem{ARR15}
{\sc Auscher, P., Ros{\'e}n, A., and Rule, D.}
\newblock Boundary value problems for degenerate elliptic equations and
  systems.
\newblock {\em Ann. Sci. {\'E}c. Norm. Sup{\'e}r. (4)}, 48(4):951--1000, 2015.



\bibitem{Ba09}
{\sc Badr, N.}
\newblock Real interpolation of {S}obolev spaces.
\newblock {\em Math. Scand.}, 105(2):235--264, 2009.

\bibitem{ChMP-A18}
{\sc Chen, L., Martell, J.~M., and Prisuelos Arribas, C.}
\newblock  The regularity problem for uniformly elliptic operators in weighted spaces. 
\url{https://arxiv.org/abs/1908.03328}. 

\bibitem{ChMPA16}
{\sc Chen, L., Martell, J.~M., and Prisuelos-Arribas, C.}
\newblock Conical square functions for degenerate elliptic operators.
\newblock  {\em Adv. Cal. Var.} 13(1):75-113, 2020.




\bibitem{CD03}
{\sc Coulhon, T., and Duong, X.~T.}
\newblock Riesz transform and related inequalities on noncompact {R}iemannian
  manifolds.
\newblock {\em Comm. Pure Appl. Math.}, 56(12):1728--1751, 2003.

\bibitem{CMR15}
{\sc Cruz-Uribe, D., Martell, J.~M., and Rios, C.}
\newblock On the Kato problem and extensions for degenerate elliptic operators.
\newblock {\em Anal. PDE}, 11(3):609--660, 2018.


\bibitem{CUMPe11}
{\sc Cruz-Uribe, D., Martell, J.~M., and P{{\'e}}rez, C.}
\newblock {\em Weights, extrapolation and the theory of {R}ubio de {F}rancia},
  vol.~215 of {\em Operator Theory: Advances and Applications}.
\newblock Birkh{\"a}user/Springer Basel AG, Basel, 2011.


\bibitem{CUR08} 
{\sc Cruz-Uribe, D., and Rios, C.}
{\em Gaussian bounds for degenerate parabolic equations.}
\newblock {\em J. Funct. Anal.}, 255(2):283--312, 2008.

\bibitem{CUR12}
{\sc Cruz-Uribe, D., and Rios. C.}
{\em  The solution of the Kato problem for degenerate elliptic operators with
Gaussian bounds.}
\newblock {\em Trans. Amer. Math. Soc.}, 364(7):3449--3478, 2012.




\bibitem{CUR15}
{\sc Cruz-Uribe, D., and Rios, C.}
\newblock{The {K}ato problem for operators with weighted ellipticity.}
\newblock {\em Trans. Amer. Math. Soc.}, 367(7):4727--4756, 2015.

\bibitem{Du01}
{\sc Duoandikoetxea, J.}
\newblock {\em Fourier analysis}, vol.~29 of {\em Graduate Studies in
  Mathematics}.
\newblock American Mathematical Society, Providence, RI, 2001.
\newblock Translated and revised from the 1995 Spanish original by David
  Cruz-Uribe.

\bibitem{FJK82}
{\sc Fabes, E., Jerison, D., and Kenig, C.}
\newblock The {W}iener test for degenerate elliptic equations.
\newblock {\em Ann. Inst. Fourier (Grenoble)}, 32(3):151--182, 1982.

\bibitem{FKJ83}
{\sc Fabes, E.~B., Kenig, C.~E., and Jerison, D.}
\newblock Boundary behavior of solutions to degenerate elliptic equations.
\newblock In {\em Conference on harmonic analysis in honor of {A}ntoni
  {Z}ygmund, {V}ol. {I}, {II} ({C}hicago, {I}ll., 1981)}, Wadsworth Math. Ser.
  Wadsworth, Belmont, CA, 1983, pp.~577--589.

\bibitem{FKS82}
{\sc Fabes, E.~B., Kenig, C.~E., and Serapioni, R.~P.}
\newblock The local regularity of solutions of degenerate elliptic equations.
\newblock {\em Comm. Partial Differential Equations}, 7(1): 77--116, 1982.

\bibitem{GCRF85}
{\sc Garc{\'{\i}}a-Cuerva, J., and Rubio~de Francia, J.~L.}
\newblock {\em Weighted norm inequalities and related topics}, vol.~116 of {\em
  North-Holland Mathematics Studies}.
\newblock North-Holland Publishing Co., Amsterdam, 1985.
\newblock Notas de Matem{{\'a}}tica [Mathematical Notes], 104.

\bibitem{Gra14}
{\sc Grafakos, L.}
\newblock {\em Classical {F}ourier analysis}, third~ed., vol.~249 of {\em
  Graduate Texts in Mathematics}.
\newblock Springer, New York, 2014.



\bibitem{HMay09}
{\sc Hofmann, S., and Mayboroda, S.}
\newblock Hardy and {BMO} spaces associated to divergence form elliptic
  operators.
\newblock {\em Math. Ann.}, 344(1):37--116, 2009.

\bibitem{HMayMc11}
{\sc Hofmann, S., Mayboroda, S., and McIntosh, A.}
\newblock Second order elliptic operators with complex bounded measurable
coefficients in {$L^p$}, {S}obolev and {H}ardy spaces.
\newblock {\em Ann. Sci. {\'E}c. Norm. Sup{\'e}r. (4)}, 44(5):723--800, 2011.  
    
\bibitem{MaPAI17}
{\sc Martell, J.~M., and Prisuelos-Arribas, C.}
\newblock  Weighted {H}ardy spaces associated with elliptic operators. {P}art
  {I}: weighted norm inequalities for conical square functions. {\em Trans. Amer. Math. Soc.}, 369(6):4193--4233, 2017.
  

\bibitem{MaPAII17}
{\sc Martell, J.~M., and Prisuelos-Arribas, C.}
\newblock Weighted {H}ardy spaces associated with elliptic operators. {P}art
  {II}: Characterizations of {$H^1_L(w)$}.
\newblock  {\em Publ. Mat.}, 62(2):475--535, 2018.

\bibitem{May10}
{\sc Mayboroda, S.}
\newblock The connections between {D}irichlet, regularity and {N}eumann
  problems for second order elliptic operators with complex bounded measurable
  coefficients.
\newblock {\em Adv. Math.}, 225(4):1786--1819, 2010.

\bibitem{PA17}
{\sc Prisuelos-Arribas, C.}
\newblock Weighted {H}ardy spaces associated with elliptic operators. {P}art
  {III}: Characterizations of {$H_L^{p}(w)$} and the weighted {H}ardy space
  associated with the {R}iesz transform. {\em  J. Geom. Anal}, 29(1):451--509, 2019.

\bibitem{PAII18}
{\sc Prisuelos-Arribas, C.}
\newblock Vertical square functions and other operators associated with an elliptic operator.
\newblock{\em J. Funct. Anal.},  277(12):108296, 63 pp.,  2019.

\end{thebibliography}

\end{document}